\documentclass[11pt,
]{amsart}
\usepackage{
amssymb, graphicx
}
 \usepackage[foot]{amsaddr}
\usepackage{amsthm}
\usepackage{xcolor}
\usepackage{graphicx,color}
\usepackage{mathtools}
\usepackage{bbm}
\usepackage{amsmath}
\newtheorem{theorem}{Theorem}
\newtheorem{lemma}{Lemma}

\newtheorem{remark}{Remark}

\theoremstyle{remark}

\newcommand{\sym}{\hbox{Sym}}

\newcommand{\be}[1]{\begin{equation}\label{#1}} 
\newcommand{\ee}{\end{equation}}

\title[Transverse and mixed ray transforms]{Invertibility of local geodesic transverse and mixed ray transforms II: higher order tensors}


    \author[G. Uhlmann]{Gunther Uhlmann}
\address{G. Uhlmann: Department of Mathematics, University of Washington, Seattle, WA 98195, USA; Institute for Advanced Study, 
The Hong Kong University of Science and Technology, Kowloon, Hong Kong, China (\tt{gunther@math.washington.edu})
}
\thanks{G. Uhlmann is partly supported by NSF}
  
  \author[J. Zhai]{Jian Zhai}
\address{J. Zhai: School of Mathematical Sciences, Fudan University, 220 Handan Road, Shanghai 200433, China
  (\tt{jianzhai@fudan.edu.cn}).}
\thanks{J. Zhai is supported by National Key Research and Development Programs of China (No. 2023YFA1009103), Science and Technology Commission of Shanghai Municipality (23JC1400501)}
 
\begin{document}

\begin{abstract}
Consider a compact Riemannian manifold in dimension $n$ with strictly convex boundary. We show the local invertibility near a boundary point of the transverse ray transform of $2$ tensors for $n\geq 3$ and the mixed ray transform of $2+2$ tensors for $n=3$. When the manifold admits a strictly convex function, this local invertibility result leads to global invertibility.
\end{abstract}

\maketitle 
\section{Introduction}
This is a subsequent paper of our work \cite{paper1} studying the ($s$-)injectivity of the transverse and mixed ray transforms. We will use the same notations and definitions adopted there.\\

Let us first recall the definition of the mixed ray transform. Assume $(M,g)$ is an $n$ dimensional compact Riemannian manifold with smooth boundary $\partial M$ and the dimension $n\geq 3$. 
The mixed ray transform $L_{k,\ell}f$ of a symmetric $(k,\ell)$-tensor $f\in C^\infty(S^k_\ell TM)$ is defined by the formula
\[
(L_{k,\ell}f)(\gamma,\eta)=\int f^{i_1\cdots i_k}_{j_1\cdots j_\ell}(\gamma(t))\eta(t)_{i_1}\cdots\eta(t)_{i_k}\dot{\gamma}(t)^{j_1}\cdots\dot{\gamma}(t)^{j_\ell}\mathrm{d}t,
\]
where $\gamma(t)$ is a geodesic in $(M,g)$ connecting two end points on $\partial M$, and $\eta$ is a unit parallel (co)vector field along $\gamma$ and $\eta$ is conormal to $\dot{\gamma}$.  The longitudinal ray transform $I_\ell=L_{0,\ell}$ has been studied quite extensively \cite{Muk2,Muk1,AR,Shara,stefanov2004stability,dairbekov2006integral,paternain2013tensor,paternain2015invariant,UV,stefanov2018inverting,de2019inverting}. For $\ell=0$, $T_k=L_{k,0}$ is called the transverse ray transform. For the results on the transverse and mixed ray transform, we refer to \cite{Shara,de2018mixed,de2021generic}. For more detailed literature review, see our previous paper \cite{paper1}.

In \cite{paper1}, we proved the local injectivity of the transverse ray transform of $(1,0)$-tensors and $s$-injectivity of the mixed ray transform of $(1,1)$-tensors near a strictly convex boundary point. The local results also leads to global ($s$-)injectivity results when $\partial M$ is strictly convex and $(M,g)$ admits a strictly convex function. In this paper, we will establish same results for the transverse ray transform of symmetric $(2,0)$-tensors and the mixed ray transform of symmetric $(2,2)$-tensors.\\

Let us first recall the redefinition of the mixed ray transform introduced in \cite{paper1}. For any $f\in S^k_\ell T_zM$ and $w\in T_zM$, define
\[
\Lambda_w:S^k_\ell T_zM\rightarrow S^{k}T_zM,\quad (\Lambda_wf)^{i_1\cdots i_k}=f^{i_1\cdots i_k}_{j_1\cdots j_\ell}w^{j_1}\cdots w^{j_\ell},
\]
For any $v\in T_z^*M$, denote $v^{\perp}=\{\eta\in T_zM|\langle \eta,v\rangle=0\}$.
Define also
\[
P_{w,v}:S^{k}T_zM\rightarrow S^{k}T_zM,\quad (P_{w,v}f)^{i_1\cdots i_k}=f^{m_1\cdots m_k}(p_{w,v})_{m_1}^{i_1}\cdots (p_{w,v})_{m_k}^{i_k}.
\]
where
\[
p_{w,v}:T_zM\rightarrow v^{\perp},\quad (p_{w,v}\zeta)^i=(p_{w,v})^i_j\zeta^j=\zeta^i-\frac{w^iv_j}{\langle w,v\rangle}\zeta^j.
\]
For any $z\in M$, and $\zeta\in T_zM$, we let $\gamma$ to be the geodesic such that
\begin{equation}\label{geodesic_initial}
\gamma(0)=z,\quad \dot{\gamma}(0)=\zeta.
\end{equation}
Denote $\mathcal{T}_\gamma^{s,t}$ to be parallel transport  from $\gamma(t)$ to $\gamma(s)$.
Let $\eta(t)\in T^*M$ be a covector field along $\gamma$ such that $\eta(t)$ is conormal to $\dot{\gamma}(t)$ and $\eta$ is parallel along $\gamma$. 
Take $\vartheta\in T^*_zM$ such that $\langle \vartheta,\zeta\rangle\neq 0$, and let $\vartheta(t)$ be the parallel transport of $\vartheta$ from $\gamma(0)$ to $\gamma(t)$.
As in \cite{paper1}, we can redefine the mixed ray transform as
\[
L_{k,\ell}f(z,\zeta,\vartheta)=\int\mathcal{T}_\gamma^{0,t}P_{\dot{\gamma}(t),\vartheta(t)}\Lambda_{\dot{\gamma}(t)}f(\gamma(t))\mathrm{d}t.
\]
In this article, we will consider the ($s$-)injectivity of $L_{k,\ell}$ for two cases: $(k,\ell)=(2,0)$ and $(k,\ell)=(2,2)$.
~\\

The main tool for the study is Melrose's scattering calculus \cite{melrose2020spectral}. Uhlmann and Vasy \cite{UV} first used this scheme and proved the local injectivity of $I_0$. Then the results have been extended to the longitudinal ray transforms of higher order tensors \cite{stefanov2018inverting,de2019inverting}. The method has also been used to study the boundary rigidity problem \cite{stefanov2021local} and other various geometrical inverse problems \cite{de2020recovery,zou2019partial,zhou2018local,zhou2018lens,paternain2019lens}.

For the problem to fit into the setup of scattering calculus, we recall some definitions and notations used in \cite{paper1}.
Let $(\widetilde{M},g)$ be a manifold without boundary extending $(M,g)$. So $\gamma\in M$ can be extended to a geodesic on $\widetilde{M}$. Let $\rho\in C^\infty(\widetilde{M})$ be a local boundary defining function for $M$, i.e., $\rho> 0$ in $M$, $\rho<0$ on $\widetilde{M}\setminus\overline{M}$, and $\rho=0$ on $\partial M$. We can assume $\rho$ is a smooth function on $\widetilde{M}$. For an arbitrary point $p\in \partial M$, we choose a function $\tilde{x}$ defined on a neighborhood $U$ of $p$ in $\widetilde{M}$, with $\tilde{p}=0$, $\mathrm{d}\tilde{x}(p)=-\mathrm{d}\rho(p)$, $\mathrm{d}\tilde{x}\neq 0$ on $U$, and $\tilde{x}$ has strictly concave level sets from the side of super-level sets. Let $\Omega_c=\{\tilde{x}>-c,\rho\geq 0\}$ such that $\overline{\Omega_c}$ is compact.

Let $X=\{x\geq 0\}$ with $x=\tilde{x}+c$. The $\partial X=\{x=0\}$ is the boundary of $X$. Now one can add $y^1,\cdots,y^{n-1}$ to $x$ such that $(x,y_1,\cdots,y_{n-1})$ are local coordinates on $X$. To use Melrose's scattering calculus \cite{melrose2020spectral}, we endow $X$ a scattering metic $g_{sc}$ of the form $g_{sc}=x^{-4}\mathrm{d}x^2+x^{-2}h$, where $h$ is a metric on the level sets of $x$.

Denote the space of scattering symmetric $(k,\ell)$-tensors $S^k_\ell{^{sc}T}X$ to be the subbundle of $S^k_\ell TX$, for which
\[
\begin{split}
&\underbrace{(x^2\partial_x)\otimes^s\cdots\otimes^s(x^2\partial_x)}_{k-m}\otimes^s \underbrace{(x\partial_{y^{i_1}})\otimes^s\cdots\otimes^s(x\partial_{y^{i_m}})}_{m}\\
&\quad\quad\quad\quad\otimes \underbrace{(\frac{\mathrm{d}x}{x^2})\otimes^s\cdots\otimes^s(\frac{\mathrm{d}x}{x^2})}_{\ell-p}\otimes^s \underbrace{(\frac{\mathrm{d}y^{j_1}}{x})\otimes^s\cdots\otimes^s(\frac{\mathrm{d}y^{j_p}}{x})}_{p},\quad 0\leq m\leq k, 0\leq p\leq \ell,
\end{split}
\]
is a local basis. Here $\otimes^s$ denotes the symmetric direct product. We also denote $\mathcal{B}S^k_\ell{^{sc}T}X$ to be the space of scattering trace-free symmetric $(k,\ell)$-tensors. See \cite{paper1} for more details.\\

For the transverse ray transform of symmetric $(2,0)$-tensors, we have the following local and global injectivity results.
 \begin{theorem}\label{thmlocal20}
Assume $\partial M$ is strictly convex at $p\in\partial M$. There exists a function $\tilde{x}\in C^\infty(\widetilde{M})$ with $O_p=\{\tilde{x}>-c\}\cap M$ for sufficiently small $c>0$, such that a symmetric $(2,0)$-tensor $f$ can be uniquely determined by $T_2f$ restricted to $O_p$-local geodesics.
\end{theorem}

\begin{theorem}\label{thmglobal20}
Assume $\partial M$ is strictly convex and $(M,g)$ admits a smooth strictly convex function. If
\[
T_2f(\gamma,\eta)=0
\]
for any geodesic $\gamma$ in $M$ with endpoints on $\partial M$, and $\eta$ a parallel vector field conormal to $\gamma$. Then $f=0$.
\end{theorem}

For the mixed ray transform of symmetric $(2,2)$-tensors, we only consider the three dimensional case. We refer to our previous paper \cite{paper1} for the definition of the operator $\mathrm{d}^\mathcal{B}$.

\begin{theorem}\label{thmlocal22}
Assume $n=3$ and $\partial M$ is strictly convex at $p\in\partial M$. There exists a function $\tilde{x}\in C^\infty(\widetilde{M})$ with $O_p=\{\tilde{x}>-c\}\cap M$ for sufficiently small $c>0$, such that for a given trace-free symmetric $(2,2)$-tensor $f$, there exists a trace-free symmetric $(2,1)$-tensor $v$ with $v\vert_{O_p\cap\partial M}=0$ such that $f-\mathrm{d}^\mathcal{B}v$ can be uniquely determined by $L_{2,2}f$ restricted to $O_p$-local geodesics.
\end{theorem}

\begin{theorem}\label{thmglobal22}
Assume $n=3$ and $\partial M$ is strictly convex and $(M,g)$ admits a smooth strictly convex function. If  $f$ is a trace-free symmetric $(2,2)$-tensor and
\[
L_{2,2}f(\gamma,\eta)=0
\]
for any geodesic $\gamma$ in $M$ with endpoints on $\partial M$, and $\eta$ a parallel vector field conormal to $\gamma$. Then there exists a trace-free symmetric $(2,1)$-tensor $v$ satisfying $v\vert_{\partial M}=0$ such that
\[
f=\mathrm{d}^\mathcal{B}v.
\]
\end{theorem}
Note that the above results are parallel with the results in \cite{paper1}. The rest of the the paper is devoted to the proof of above theorems. We will prove Theorem \ref{thmlocal20} and Theorem \ref{thmglobal20} for $T_2$ in Section \ref{transversesection}, Theorem \ref{thmlocal22} and Theorem \ref{thmglobal22} for $L_{2,2}$ in Section \ref{mixedsection}.

\section{Transverse ray transform of $2+0$ tensors}\label{transversesection}

Assume in local coordinates $z=(x,y)$ and $(\lambda,\omega)\in T_zM$ with $x$ the boundary defining function and $\omega$ renormalized to $|\omega|=1$. Denote $\gamma_{x,y,\lambda,\omega}(t)$ to be the geodesic such that
\[
\gamma_{x,y,\lambda,\omega}(0)=(x,y),\quad \dot{\gamma}_{x,y,\lambda,\omega}(0)=\lambda\partial_x+\omega\partial_y.
\]
We will only use geodesics almost tangential to level sets of $x$.
The maps 
\[
\Gamma_+:S\widetilde{M}\times[0,\infty)\rightarrow [\widetilde{M}\times\widetilde{M};\mathrm{diag}],\quad \Gamma_+(x,y,\lambda,\omega,t)=((x,y),\gamma_{x,y,\lambda,\omega}(t))
\]
and
\[
\Gamma_-:S\widetilde{M}\times(-\infty,0]\rightarrow [\widetilde{M}\times\widetilde{M};\mathrm{diag}],\quad \Gamma_-(x,y,\lambda,\omega,t)=((x,y),\gamma_{x,y,\lambda,\omega}(t))
\]
and diffeomorphisms near $S\widetilde{M}\times\{0\}$. Here $[\widetilde{M}\times\widetilde{M};\mathrm{diag}]$ is the blow-up of $\widetilde{M}$ at the diagonal $z=z'$.
Writing the local coordinates $(x,y,x',y')$ for $\widetilde{M}\times\widetilde{M}$, in the region $|x-x'|<C|y-y'|$.
we use local coordinates
\[
x,y,X=\frac{x'-x}{x^2},Y=\frac{y'-y}{x}
\]
on Melrose's scattering double space near the lifted scattering diagonal.\\

Let $f\in S^2TX$ be a symmetric $(2,0)$-tensor, then the transverse ray transform of $f$ is defined as
\[
\begin{split}
T_2f(x,y,\lambda,\omega)
=&\int_{\mathbb{R}}\mathcal{T}_{\gamma_{x,y,\lambda,\omega}}^{0,t}P_{\dot{\gamma}_{x,y,\lambda,\omega}(t),\vartheta_{x,y,\omega}(t)}f({\gamma_{x,y,\lambda,\omega}(t)})\mathrm{d}t\\
=&\int_{\mathbb{R}}\mathcal{T}_{\gamma_{x,y,\lambda,\omega}}^{0,t} \left(\mathrm{Id}-\frac{\dot{\gamma}_{x,y,\lambda,\omega}(t)\langle\vartheta_{x,y,\omega}(t),\,\cdot\,\rangle}{\langle\dot{\gamma}_{x,y,\lambda,\omega}(t),\vartheta_{x,y,\omega}(t)\rangle}\right)^{\otimes 2}f({\gamma_{x,y,\lambda,\omega}(t)})\mathrm{d}t,
\end{split}
\]
where $\gamma_{x,y,\lambda,\omega}$ is a geodesic in $(X\cap \widetilde{M},g)$ such that
\[
\gamma_{x,y,\lambda,\omega}(0)=(x,y),\quad\dot{\gamma}_{x,y,\lambda,\omega}(0)=(\lambda,\omega),
\]
$\vartheta_{x,y,\omega}(t)$ is parallel along $\gamma_{x,y,\lambda,\omega}$ with 
\begin{equation}\label{vartheta_initial}
\vartheta=\vartheta_{x,y,\omega}(0)=(\vartheta_x(0),\vartheta_y(0))=(0,h(\omega))=h(\omega)\mathrm{d}y,
\end{equation}
and $\mathcal{T}_{\gamma_{x,y,\lambda,\omega}}^{0,t}$ is a parallel transport (w.r.t. $g$) from $\gamma_{x,y,\lambda,\omega}(t)$ to $\gamma_{x,y,\lambda,\omega}(0)$. 
Here we have used the notation $v^{\otimes 2}=v\otimes v$ for $v\in S^1_1TX$, such that componentwisely $(v^{\otimes 2})^{ij}_{mn}=v^i_m v^j_n$.

Take $\chi$ a smooth, even, non-negative, compactly supported function and define
\[
\begin{split}
T_2'v(x,y)=x^{-2}\int\left(\mathrm{Id}-\frac{\omega\partial_y\otimes g_{sc}(\lambda\partial_x+\omega\partial_y)}{h(\omega\partial_y,\omega\partial_y)}\right)^{\otimes 2}\chi(\lambda/x)v(\gamma_{x,y,\lambda,\omega}(0),\dot{\gamma}_{x,y,\lambda,\omega}(0))\mathrm{d}\lambda\mathrm{d}\omega.
\end{split}
\]

Let us examine how the parallel transport $\mathcal{T}_{\gamma_{x,y,\lambda,\omega}}^{0,t}$ acts on scattering symmetric $(2,0)$-tensors at the scattering front face $x=0$. Here we choose coordinates $(x,y)$ such that $\partial_{y^i}$ and $\partial_x$ are orthogonal, and the metric is of the form $f(x,y)\mathrm{d}x^2+k(x,y)\mathrm{d}y^2$ near $x=0$.
 Assume that $V$ is a symmetric $(2,0)$-tensor parallel along $\gamma_{x,y,\lambda,\omega}$, such that
\[
V=v^{xx}(x^2\partial_x)\otimes (x^2\partial_x)+2v^{xy^i}(x^2\partial_x)\otimes^s(x\partial_{y^i})+v^{y^iy^j}(x\partial_{y^i})\otimes^s(x\partial_{y^j}).
\]
Then with $H_{ij}(x,y)=-\frac{1}{2}g^{xx}\partial_x k_{y^iy^j}(x,y)$, we calculate
\[
\dot{V}^{xx}(0)=V^{xy^j}\dot{\gamma}^{y^k}\Gamma_{y^jy^k}^x+V^{y^jx}\dot{\gamma}^{y^k}\Gamma_{y^jy^k}^x+\mathcal{O}(x^4)=-2V^{xy^j}H_{jk}\omega^k+\mathcal{O}(x^4),
\]
\[
\dot{V}^{xy^i}(0)=V^{y^iy^k}\dot{\gamma}^{y^i}\Gamma_{y^jy^k}^x+\mathcal{O}(x^3)=-V^{y^iy^k}H_{jk}\omega^k+\mathcal{O}(x^3),
\]
\[
\dot{V}^{y^iy^j}(0)=\mathcal{O}(x^2).
\]
Therefore
\[
V^{xx}(t)=V^{xx}(0)-2H_{jk}\omega^kt+\mathcal{O}(x^5)=x^4(v^{xx}-2|Y|\langle H\hat{Y},v^{xy}\rangle)+\mathcal{O}(x^5),
\]
and similarly,
\[
V^{xy}(t)=x^3(v^{xy}-|Y|\langle H\hat{Y},v^{yy}\rangle)+\mathcal{O}(x^4),
\]
\[
V^{yy}(t)=x^2v^{yy}+\mathcal{O}(x^3).
\]
To summarize we have
\[
\left(\begin{array}{c}
v^{xx}(t)\\
v^{xy}(t)\\
v^{yy}(t)
\end{array}\right)=
\left(\begin{array}{ccc}
1 &-2|Y|\langle H\hat{Y},\,\cdot\,\rangle&0\\
0&1&-|Y|\langle H\hat{Y},\,\cdot\,\rangle\\
0&0&1
\end{array}\right)
\left(\begin{array}{c}
v^{xx}(0)\\
v^{xy}(0)\\
v^{yy}(0)
\end{array}\right)+\mathcal{O}(x).
\]
Recall that as in \cite{paper1} the coordinates need to be chosen such that $H\hat{Y}=\alpha\hat{Y}$ at the point of interest.\\

The Schwartz kernel of $N_{\mathsf{F}}=e^{-\mathsf{F}/x}T_2'T_2e^{\mathsf{F}/x}$ is then
\[
\begin{split}
&K^\flat(x,y,X,Y)\\
=&\sum_{\pm}e^{-\mathsf{F}X/(1+xX)}\chi\left(\frac{X-\alpha(x,y,x|Y|,\frac{xX}{|Y|},\hat{Y})}{|Y|}+x\tilde{\Lambda}_\pm\left(x,y,x|Y|,\frac{xX}{|Y|},\hat{Y})\right)\right)\\
&\left(\mathrm{Id}-\frac{\left((\Omega\circ\Gamma^{-1}_{\pm})x\partial_y\right)\otimes\left(x^{-1}(\Lambda\circ\Gamma^{-1}_{\pm})x^2\partial_x+(\Omega\circ\Gamma^{-1}_{\pm})\frac{h(\partial_y)}{x}\right)}{|\Omega\circ\Gamma^{-1}_{\pm}|^2h(\partial_y,\partial_y)}\right)^{\otimes 2}\mathcal{T}_\pm(x,y,X, Y)\\
&\left(\mathrm{Id}-\frac{\left(x^{-1}(\Lambda'\circ\Gamma^{-1}_{\pm})x^2\partial_x+(\Omega'\circ\Gamma^{-1}_{\pm})x\partial_y\right)\otimes\left(x^{-1}(\Xi_x'\circ\Gamma^{-1}_{\pm})\frac{\mathrm{d}x}{x^2}+(\Xi_y'\circ\Gamma^{-1}_{\pm})\frac{h(\partial_y)}{x}\right)}{x^{-2}(\Lambda'\circ\Gamma^{-1}_{\pm})(\Xi_x'\circ\Gamma_{\pm}^{-1})+(\Omega'\circ\Gamma^{-1}_{\pm})(\Xi_y'\circ\Gamma^{-1}_{\pm})h(\partial_y,\partial_y)}\right)^{\otimes 2}\\
&|Y|^{-n+1}J_\pm\left(x,y,x|Y|,\frac{xX}{|Y|},\hat{Y})\right),
\end{split}
\]
where
\[
\begin{split}
\mathcal{T}_\pm(0,y,X,Y)=&\left(\begin{array}{ccc}
1 &-2\alpha|Y|\langle\hat{Y},\,\cdot\,\rangle&0\\
0&1&-\alpha|Y|\langle\hat{Y},\,\cdot\,\rangle\\
0&0&1
\end{array}\right)^{-1}\\
=&\left(\begin{array}{ccc}
1 &0&2\alpha^2|Y|^2\langle \hat{Y}\otimes \hat{Y},\,\cdot\,\rangle \\
0&1&0\\
0&0&1
\end{array}\right)\left(\begin{array}{ccc}
1 &2\alpha|Y|\langle \hat{Y},\,\cdot\,\rangle&0\\
0&1&\alpha|Y|\langle \hat{Y},\,\cdot\,\rangle\\
0&0&1
\end{array}\right).
\end{split}
\]

We will prove that $N_{\mathsf{F}}\in\Psi_{sc}^{-1,0}(X)$ is \textit{fully elliptic} as a scattering pseudodifferential operator for $x$ sufficiently close to $0$. For this, we need to show that the principal symbol is elliptic at the \textit{fiber infinity} and at the \textit{base infinity}.

\begin{lemma}\label{ellipticity20fiber}
The operator $N_\mathsf{F}$ is elliptic at fiber infinity of $^{sc}T^*X$ in $\widetilde{M}$.
\end{lemma}
\begin{proof}
We only need to analyze the principal symbol at $x=0$, since the Schwartz kernel $K^\flat$ is smooth in $(x,y)$ down to $x=0$.
At the front face $x=0$, the kernel of $N_\mathsf{F}$ is
\[
\begin{split}
K^\flat(0,y,X,Y)&=e^{-\mathsf{F}X}|Y|^{-n+1}\chi(S)\left(\mathrm{Id}-\left(\hat{Y}\cdot(x\partial_y)\right)\left(S\frac{\mathrm{d}x}{x^2}+\hat{Y}\cdot\frac{\mathrm{d}y}{x}\right)\right)^{\otimes 2}\\
&\mathcal{T}(0,y,X,Y)\left(\mathrm{Id}-\left((S+2\alpha |Y|)(x^2\partial_x)+\hat{Y}\cdot(x\partial_y)\right)\left(\hat{Y}\cdot\frac{\mathrm{d}y}{x}\right)\right)^{\otimes 2}.
\end{split}
\]
Note that both
\[
\left(\mathrm{Id}-\left(\hat{Y}\cdot(x\partial_y)\right)\left(S\frac{\mathrm{d}x}{x^2}+\hat{Y}\cdot\frac{\mathrm{d}y}{x}\right)\right)^{\otimes 2}
\] 
and 
\[
\left(\mathrm{Id}-\left((S+2\alpha |Y|)(x^2\partial_x)+\hat{Y}\cdot(x\partial_y)\right)\left(\hat{Y}\cdot\frac{\mathrm{d}y}{x}\right)\right)^{\otimes 2}
\]
map symmetric $(2,0)$-tensors to symmetric $(2,0)$-tensors.

We use
\[
(x^2\partial_x)\otimes (x^2\partial_x),\quad 2(x^2\partial_x)\otimes^s (x\partial_y),\quad  (x\partial_y)\otimes (x\partial_y)
\]
as a basis for scattering symmetric $(2,0)$-tensors. Here $\otimes^s$ represents the symmetric product where
\[
 (x^2\partial_x)\otimes^s (x\partial_y)=\frac{1}{2}\left( (x^2\partial_x)\otimes (x\partial_y)+ (x\partial_y)\otimes (x^2\partial_x)\right).
\]
 This means that if $f\in S^2 {^{sc}}TX$ can be written as
\[
f=f^{xx}(x^2\partial_x)\otimes (x^2\partial_x)+2f^{xy}(x^2\partial_x)\otimes^s (x\partial_y)+f^{yy}(x\partial_y)\otimes (x\partial_y),
\]
we can express $f$ in vector-form as
\[
f=\left(\begin{array}{c}
f^{xx}\\
f^{xy}\\
f^{yy}
\end{array}\right).
\]
\begin{remark}
The symmetry of the matrix in this section needs to be checked with respect to the inner product
\[
M(2,0)=\left(\begin{array}{ccc}1&0&0\\
0&2\times\mathrm{Id}&0\\
0&0&\mathrm{Id}
\end{array}\right),
\]
that is 
\[
(f,g)=f^TM(2,0)g.
\]
\end{remark}

Using this basis, the kernel of $N_\mathsf{F}$ at the front face $x=0$ can be expressed in matrix-form as
\begin{equation}\label{kernel20_matrix}
{\scriptsize
\begin{split}
e^{-\mathsf{F}X}|Y|^{-n+1}\chi(S)&\left(\begin{array}{ccc}
1 & 0 &0\\
-S\hat{Y} &\mathrm{Id}-\hat{Y}\langle\hat{Y},\,\cdot\,\rangle &0\\
S^2\hat{Y}\otimes\hat{Y} &-2S\hat{Y}\otimes^s+2S\hat{Y}\otimes\hat{Y}\langle\hat{Y},\,\cdot\,\rangle &\mathrm{Id}-2\hat{Y}\otimes^s\langle\hat{Y},\,\cdot\,\rangle+\hat{Y}\otimes\hat{Y}\langle\hat{Y}\otimes\hat{Y},\,\cdot\,\rangle
\end{array}\right)\\
&\left(\begin{array}{ccc}
1 &0&2\alpha^2|Y|^2\langle \hat{Y}\otimes \hat{Y},\,\cdot\,\rangle \\
0&\mathrm{Id}&0\\
0&0&\mathrm{Id}
\end{array}\right)\left(\begin{array}{ccc}
1 &2\alpha|Y|\langle\hat{Y},\,\cdot\,\rangle&0\\
0&\mathrm{Id}&\alpha|Y|\langle \hat{Y},\,\cdot\,\rangle\\
0&0&\mathrm{Id}
\end{array}\right)\\
&\left(\begin{array}{ccc}
1 &-2(S+2\alpha|Y|)\langle\hat{Y},\,\cdot\,\rangle & (S+2\alpha|Y|)^2\langle\hat{Y}\otimes\hat{Y},\,\cdot\,\rangle\\
0&\mathrm{Id}-\hat{Y}\langle\hat{Y},\,\cdot\,\rangle &-(S+2\alpha|Y|)\langle\hat{Y},\,\cdot\,\rangle+(S+2\alpha|Y|)\hat{Y}\langle\hat{Y}\otimes\hat{Y},\,\cdot\,\rangle\\
0&0&\mathrm{Id}-2\hat{Y}\otimes^s\langle\hat{Y},\,\cdot\,\rangle+\hat{Y}\otimes\hat{Y}\langle\hat{Y}\otimes\hat{Y},\,\cdot\,\rangle
\end{array}\right)\\
=e^{-\mathsf{F}X}|Y|^{-n+1}\chi(S)&\left(\begin{array}{ccc}
1 & 0 &0\\
-S\hat{Y} &\mathrm{Id}-\hat{Y}\langle\hat{Y},\,\cdot\,\rangle &0\\
S^2\hat{Y}\otimes\hat{Y} &-2S\hat{Y}\otimes^s+2S\hat{Y}\otimes\hat{Y}\langle\hat{Y},\,\cdot\,\rangle &\mathrm{Id}-2\hat{Y}\otimes^s\langle\hat{Y},\,\cdot\,\rangle+\hat{Y}\otimes\hat{Y}\langle\hat{Y}\otimes\hat{Y},\,\cdot\,\rangle
\end{array}\right)\\
&\left(\begin{array}{ccc}
1 &-2(S+2\alpha|Y|)\langle\hat{Y},\,\cdot\,\rangle & (S+2\alpha|Y|)^2\langle\hat{Y}\otimes\hat{Y},\,\cdot\,\rangle\\
0&\mathrm{Id}-\hat{Y}\langle\hat{Y},\,\cdot\,\rangle &-(S+2\alpha|Y|)\langle\hat{Y},\cdot\rangle+(S+2\alpha|Y|)\hat{Y}\langle\hat{Y}\otimes\hat{Y},\,\cdot\,\rangle\\
0&0&\mathrm{Id}-2\hat{Y}\otimes^s\langle\hat{Y},\,\cdot\,\rangle+\hat{Y}\otimes\hat{Y}\langle\hat{Y}\otimes\hat{Y},\,\cdot\,\rangle
\end{array}\right).
\end{split}
}
\end{equation}
Here $\otimes^s$ denotes the symmetrized product.

The principal symbol $\sigma_p(N_\mathsf{F})(0,y,\xi,\eta)$ of $K^\flat(0,y,X,Y)$ at fiber infinity can be obtained by integrating the restriction of the Schwartz kernel to the front face of the lifted diagonal, $|Y|=0$, after removing the singular factor $|Y|^{-n+1}$, along the $(\xi,\eta)$-equatorial sphere $\frac{X}{|Y|}\xi+\hat{Y}\cdot\eta=0$ (so $X=0$). See the proof of \cite[Lemma 3.4]{stefanov2018inverting} for more details.
So we need to integrate the matrix
\begin{equation}\label{matrix20_ff}
\begin{split}
\chi(\tilde{S})&\left(\begin{array}{ccc}
1 & 0 &0\\
-\tilde{S}\hat{Y} &\mathrm{Id}-\hat{Y}\langle\hat{Y},\,\cdot\,\rangle &0\\
\tilde{S}^2\hat{Y}\otimes\hat{Y} &-2\tilde{S}\hat{Y}\otimes^s+2\tilde{S}\hat{Y}\otimes\hat{Y}\langle\hat{Y},\,\cdot\,\rangle &\mathrm{Id}-2\hat{Y}\otimes^s\langle\hat{Y},\,\cdot\,\rangle+\hat{Y}\otimes\hat{Y}\langle\hat{Y}\otimes\hat{Y},\,\cdot\,\rangle
\end{array}\right)\\
&\left(\begin{array}{ccc}
1 &-2\tilde{S}\langle\hat{Y},\,\cdot\,\rangle & \tilde{S}^2\langle\hat{Y}\otimes\hat{Y},\,\cdot\,\rangle\\
0&\mathrm{Id}-\hat{Y}\langle\hat{Y},\,\cdot\,\rangle &-\tilde{S}\langle\hat{Y},\,\cdot\,\rangle+\tilde{S}\hat{Y}\langle\hat{Y}\otimes\hat{Y},\,\cdot\,\rangle\\
0&0&\mathrm{Id}-2\hat{Y}\otimes^s\langle\hat{Y},\,\cdot\,\rangle+\hat{Y}\otimes\hat{Y}\langle\hat{Y}\otimes\hat{Y},\,\cdot\,\rangle
\end{array}\right)
\end{split}
\end{equation}
with $(\tilde{S},\hat{Y})$ running through the $(\xi,\eta)$-equatorial sphere $\xi\tilde{S}+\eta\cdot\hat{Y}=0$ with $\tilde{S}=\frac{X}{|Y|}$. Note that the above matrix is positive semidefinite (w.r.t. inner product $M(2,0)$), and we need to show that its integral is nonsingular.\\

Now assume that $f$ is in the kernel of \eqref{matrix20_ff}
for $\tilde{S}$ sufficiently close to $0$ such that $\chi(\tilde{S})>0$.
 Then it is in the kernel of the matrix
\[
\left(\begin{array}{ccc}
1 &-2\tilde{S}\langle\hat{Y},\cdot\rangle & \tilde{S}^2\langle\hat{Y}\otimes\hat{Y},\cdot\rangle\\
0&\mathrm{Id}-\hat{Y}\langle\hat{Y},\cdot\rangle &-\tilde{S}\langle\hat{Y},\cdot\rangle+\tilde{S}\hat{Y}\langle\hat{Y}\otimes\hat{Y},\cdot\rangle\\
0&0&\mathrm{Id}-2\hat{Y}\otimes^s\langle\hat{Y},\cdot\rangle+\hat{Y}\otimes\hat{Y}\langle\hat{Y}\otimes\hat{Y},\cdot\rangle
\end{array}\right).
\]
First take $\tilde{S}=0$ and $\hat{Y}\cdot\eta=0$, we have
\[
f^{xx}=0,\quad f^{xy}-\langle\hat{\eta}^\perp,f^{xy}\rangle\hat{\eta}^\perp=0,\quad f^{yy}-2\langle\hat{\eta}^\perp,f^{yy}\rangle\otimes^s\hat{\eta}^\perp+\langle\hat{\eta}^\perp\otimes\hat{\eta}^\perp,f^{yy}\rangle \hat{\eta}^\perp\otimes\hat{\eta}^\perp=0.
\]
for any unit vector $\hat{\eta}^\perp\perp\eta$. If $\eta=0$, we can conclude that the above identities hold for any unit vector $\hat{\eta}^\perp$. Then we can already conclude that $f^{xx}=f^{xy}=f^{yy}=0$.\\

So in the following we assume that $\eta\neq 0$. First we have
\begin{equation}\label{eq_fyy}
\langle \hat{\eta},f^{xy}\rangle=\langle\hat{\eta}\otimes\hat{\eta},f^{yy}\rangle=0,
\end{equation}
where $\hat{\eta}=\frac{\eta}{|\eta|}$.
 Next we consider the following equations:
 \begin{equation}\label{eq53}
 \begin{split}
 f^{xy}-\hat{Y}\langle\hat{Y},f^{xy}\rangle-\tilde{S}\langle\hat{Y},f^{yy}\rangle+\tilde{S}\hat{Y}\langle\hat{Y}\otimes\hat{Y},f^{yy}\rangle=0,\\
 f^{xx}-2\tilde{S}\langle\hat{Y},f^{xy}\rangle+ \tilde{S}^2\langle\hat{Y}\otimes\hat{Y},f^{yy}\rangle=0.
 \end{split}
 \end{equation}

\textit{Case 1:}  If $\xi\neq 0$, we take $\hat{Y}=\epsilon\hat{\eta}+\sqrt{1-\epsilon^2}\hat{\eta}^\perp$ and $\tilde{S}=-\frac{\epsilon|\eta|}{\xi}$.
Then we have from the second equation in \eqref{eq53},
 \[
 \begin{split}
f^{xx}+2\epsilon^2\frac{|\eta|}{\xi}\langle\hat{\eta},f^{xy}\rangle+2\epsilon\sqrt{1-\epsilon^2}\frac{|\eta|}{\xi}\langle\hat{\eta}^\perp,f^{xy}\rangle+\epsilon^4\frac{|\eta|^2}{\xi^2}\langle\hat{\eta}\otimes\hat{\eta},f^{yy}\rangle&\\
+2\epsilon^3\sqrt{1-\epsilon^2}\frac{|\eta|^2}{\xi^2}\langle\hat{\eta}\otimes\hat{\eta}^\perp,f^{yy}\rangle+\epsilon^2(1-\epsilon^2)\frac{|\eta|^2}{\xi^2}\langle\hat{\eta}^\perp\otimes\hat{\eta}^\perp,f^{yy}\rangle&=0.
\end{split}
 \]
 So
 \[
2\epsilon\sqrt{1-\epsilon^2}\frac{|\eta|}{\xi}\langle\hat{\eta}^\perp,f^{xy}\rangle
+2\epsilon^3\sqrt{1-\epsilon^2}\frac{|\eta|^2}{\xi^2}\langle\hat{\eta}\otimes\hat{\eta}^\perp,f^{yy}\rangle+\epsilon^2(1-\epsilon^2)\frac{|\eta|^2}{\xi^2}\langle\hat{\eta}^\perp\otimes\hat{\eta}^\perp,f^{yy}\rangle=0.
 \]
 Taking 1st order derivative in $\epsilon$ at $\epsilon=0$, we have
 \[
 \langle\hat{\eta}^\perp,f^{xy}\rangle=0.
 \]
 Taking 3rd and 4th order derivative yields
 \[
\langle\hat{\eta}\otimes\hat{\eta}^\perp,f^{yy}\rangle=\langle\hat{\eta}^\perp\otimes\hat{\eta}^\perp,f^{yy}\rangle=0.
 \]
 Now we can conclude that $f=0$.\\
 
 \textit{Case 2:} If $\xi= 0$, we take $\hat{Y}\perp\hat{\eta}$ and $\tilde{S}$ to be an arbitrary small number, and obtain
 \[
 \langle\hat{\eta}^\perp, f^{xy}\rangle=\langle\hat{\eta}^\perp\otimes \hat{\eta}^\perp,f^{yy}\rangle=0.
 \]
 Then we have from the first equation in \eqref{eq53}
 \[
 \langle\hat{\eta}^\perp,f^{yy}\rangle=0.
 \]
 So we can also conclude that $f=0$.
 
Therefore the integral of the matrix \eqref{matrix20_ff} along the equatorial sphere is positive definite. This proves that the boundary principal symbol of $N_\mathsf{F}$ is elliptic at fiber infinity.
 \end{proof}

 \begin{lemma}
The operator $N_\mathsf{F}$ is elliptic at base infinity of $^{sc}T^*X$ in $\widetilde{M}$ for $\mathsf{F}$ sufficiently large.
\end{lemma}
\begin{proof}
To get the principal symbol at finite points we need to calculate the $(X,Y)$-Fourier transform of the matrix \eqref{kernel20_matrix}. As in \cite{stefanov2018inverting}, we can take $\chi(s)=e^{-s^2/(2\nu(\hat{Y}))}$ to do the actual computation. 

Taking $\nu=\mathsf{F}^{-1}\alpha$ with $h=\mathsf{F}^{-1}$ considered as a semiclassical parameter, and doing tedious calculation as in \cite{paper1}, we obtain the (semiclassical) principal symbol of $N_\mathsf{F}$ at the scattering front face $x=0$:
\begin{equation}\label{integral_ff_matrix20}
\int_{\mathbb{S}^{n-2}} (\xi_\mathsf{F}^2+1)^{-1/2}\mathcal{M}_{20}e^{-(\hat{Y}\cdot\eta_\mathsf{F})^2/(2h\phi_\mathsf{F}(\xi_\mathsf{F},\hat{Y}))}\mathrm{d}\hat{Y},
\end{equation}
where
 \[
 {\scriptsize
 \begin{split}
\mathcal{M}_{20}=\left(\begin{array}{ccc}
1 & 0 &\\
2(\xi_\mathsf{F}+\mathrm{i})\frac{\hat{Y}\cdot\eta_\mathsf{F}}{\xi_\mathsf{F}^2+1}\hat{Y}&I-\hat{Y}\langle\hat{Y},\,\cdot\,\rangle& 0\\
(\xi_\mathsf{F}+\mathrm{i})^2\left(\frac{\hat{Y}\cdot\eta_\mathsf{F}}{\xi_\mathsf{F}^2+1}\right)^2\hat{Y}\otimes\hat{Y}&(\xi_\mathsf{F}+\mathrm{i})\frac{\hat{Y}\cdot\eta_\mathsf{F}}{\xi_\mathsf{F}^2+1}\left(\hat{Y}\otimes^s-\hat{Y}\otimes\hat{Y}\langle\hat{Y},\,\cdot\,\rangle\right)& I-2\hat{Y}\otimes^s\langle\hat{Y},\,\cdot\,\rangle+\hat{Y}\otimes\hat{Y}\langle\hat{Y}\otimes\hat{Y},\,\cdot\,\rangle
\end{array}
\right)\\
\left(\begin{array}{ccc}
1 & 2(\xi_\mathsf{F}-\mathrm{i})\frac{\hat{Y}\cdot\eta_\mathsf{F}}{\xi_\mathsf{F}^2+1}\langle\hat{Y},\,\cdot\,\rangle &(\xi_\mathsf{F}-\mathrm{i})^2\left(\frac{\hat{Y}\cdot\eta_\mathsf{F}}{\xi_\mathsf{F}^2+1}\right)^2\langle\hat{Y}\otimes\hat{Y},\,\cdot\,\rangle\\
0&I-\hat{Y}\langle\hat{Y},\,\cdot\,\rangle& (\xi_\mathsf{F}-\mathrm{i})\frac{\hat{Y}\cdot\eta_\mathsf{F}}{\xi_\mathsf{F}^2+1}\left(\langle\hat{Y},\,\cdot\,\rangle-\hat{Y}\langle\hat{Y}\otimes\hat{Y},\,\cdot\,\rangle\right)\\
0&0& I-2\hat{Y}\otimes^s\langle\hat{Y},\,\cdot\,\rangle+\hat{Y}\otimes\hat{Y}\langle\hat{Y}\otimes\hat{Y},\,\cdot\,\rangle
\end{array}
\right).
\end{split}
}
\]
Here $(\xi_\mathsf{F},\eta_\mathsf{F})=h(\xi,\eta)$, where $(\xi,\eta)$ is the Fourier dual variables of $(X,Y)$.
Assume $f$ is in the kernel of above matrix.
Take a unit vector $\hat{Y}=\hat{\eta}^\perp_\mathsf{F}$ perpendicular to $\eta_\mathsf{F}$, then the above matrix becomes
 \[
\left(\begin{array}{ccc}
1 & 0 & 0\\
0 & I-\hat{\eta}^\perp_\mathsf{F}\langle\hat{\eta}^\perp_\mathsf{F},\,\cdot\,\rangle &0\\
0 &0&I-2\hat{\eta}^\perp_\mathsf{F}\otimes^s\langle\hat{\eta}^\perp_\mathsf{F},\,\cdot\,\rangle+\hat{\eta}^\perp_\mathsf{F}\otimes\hat{\eta}^\perp_\mathsf{F}\langle\hat{\eta}^\perp_\mathsf{F}\otimes\hat{\eta}^\perp_\mathsf{F},\,\cdot\,\rangle
\end{array}
\right).
\]
If $\eta_\mathsf{F}=0$, then $\hat{\eta}^\perp_\mathsf{F}$ can be any vector and we can conclude $f=0$.
So now we assume $\eta_\mathsf{F}\neq 0$ and get
\[
f^{xx}=0,\quad\langle\hat{\eta}_\mathsf{F},f^{xy}\rangle=0,\quad\langle \hat{\eta}_\mathsf{F}\otimes  \hat{\eta}_\mathsf{F},f^{yy}\rangle=0,
\]
 Taking $\hat{Y}=\epsilon\hat{\eta}_\mathsf{F}+\sqrt{1-\epsilon^2}\hat{\eta}_\mathsf{F}^\perp$, we have
 \[
 \begin{split}
 f^{xx}+2\epsilon^2(\xi_\mathsf{F}-\mathrm{i})\frac{|\eta_\mathsf{F}|}{\xi_\mathsf{F}^2+1}\langle\hat{\eta}_\mathsf{F},f^{xy}\rangle+2\epsilon\sqrt{1-\epsilon^2}(\xi_\mathsf{F}-\mathrm{i})\frac{|\eta_\mathsf{F}|}{\xi_\mathsf{F}^2+1}\langle\hat{\eta}^\perp_\mathsf{F},f^{xy}\rangle\\
 +\epsilon^4(\xi_\mathsf{F}-\mathrm{i})^2\left(\frac{|\eta_\mathsf{F}|}{\xi_\mathsf{F}^2+1}\right)^2\langle\hat{\eta}_\mathsf{F}\otimes \hat{\eta}_\mathsf{F},f^{yy}\rangle+2\epsilon^3\sqrt{1-\epsilon^2}(\xi_\mathsf{F}-\mathrm{i})^2\left(\frac{|\eta_\mathsf{F}|}{\xi_\mathsf{F}^2+1}\right)^2\langle\hat{\eta}_\mathsf{F}\otimes \hat{\eta}_\mathsf{F}^\perp,f^{yy}\rangle\\
 +\epsilon^2(1-\epsilon^2)(\xi_\mathsf{F}-\mathrm{i})^2\left(\frac{|\eta_\mathsf{F}|}{\xi_\mathsf{F}^2+1}\right)^2\langle\hat{\eta}_\mathsf{F}^\perp\otimes \hat{\eta}_\mathsf{F}^\perp,f^{yy}\rangle=0,
 \end{split}
 \]
 which simplifies to
  \[
 \begin{split}
2\epsilon\sqrt{1-\epsilon^2}(\xi_\mathsf{F}-\mathrm{i})\frac{|\eta_\mathsf{F}|}{\xi_\mathsf{F}^2+1}\langle\hat{\eta}^\perp_\mathsf{F},f^{xy}\rangle
 +2\epsilon^3\sqrt{1-\epsilon^2}(\xi_\mathsf{F}-\mathrm{i})^2\left(\frac{|\eta_\mathsf{F}|}{\xi_\mathsf{F}^2+1}\right)^2\langle\hat{\eta}_\mathsf{F}\otimes \hat{\eta}_\mathsf{F}^\perp,f^{yy}\rangle\\
 +\epsilon^2(1-\epsilon^2)(\xi_\mathsf{F}-\mathrm{i})^2\left(\frac{|\eta_\mathsf{F}|}{\xi_\mathsf{F}^2+1}\right)^2\langle\hat{\eta}_\mathsf{F}^\perp\otimes \hat{\eta}_\mathsf{F}^\perp,f^{yy}\rangle=0.
 \end{split}
 \]
Taking 1st order derivative in $\epsilon$ at $\epsilon=0$, we have
 \[
 \langle\hat{\eta}_\mathsf{F}^\perp,f^{xy}\rangle=0.
 \]
 Taking 3rd and 4th order derivative yields
 \[
\langle\hat{\eta}_\mathsf{F}\otimes\hat{\eta}_\mathsf{F}^\perp,f^{yy}\rangle=\langle\hat{\eta}_\mathsf{F}^\perp\otimes\hat{\eta}_\mathsf{F}^\perp,f^{yy}\rangle=0.
 \]
 Now we can conclude that $f=0$. This shows that the integral \eqref{integral_ff_matrix20} is a positive definite matrix, meaning that the principal symbol of $N_\mathsf{F}$ is elliptic at base infinity.
 \end{proof}
  
  As in \cite{paper1}, we can show the local and global injectivity results stated in Theorem \ref{thmlocal20} and  \ref{thmglobal20}. See \cite{UV} for more details.

\section{Mixed ray transform of $2+2$ tensors}\label{mixedsection}
The local mixed ray transform of a trace-free symmetric $(2,2)$-tensor is defined by
\[
L_{2,2}f(x,y,\lambda,\omega)=\int_{\mathbb{R}}\mathcal{T}_{\gamma_{x,y,\lambda,\omega}}^{0,t}P_{\dot{\gamma}_{x,y,\lambda,\omega}(t),\vartheta_{x,y,\omega}(t)}\Lambda_{\dot{\gamma}_{x,y,\lambda,\omega}(t)}f^{sc}({\gamma_{x,y,\lambda,\omega}(t)})\mathrm{d}t,
\]
for $f\in\mathcal{B}S^2_2TX$. 
Then we define
\[
L_{22}'v(x,y)=x^2\int\chi(\lambda/x)\left(\left(\mathrm{Id}-\frac{\omega\partial_y\otimes g_{sc}(\lambda\partial_x+\omega \partial_y)}{h(\omega\partial_y,\omega\partial_y)}\right)^{\otimes 2}v(x,y,\lambda,\omega)\right)\otimes g_{sc}(\lambda\partial_x+\omega\partial_y)^{\otimes 2}\mathrm{d}\lambda\mathrm{d}\omega,
\]
and
\[
N_\mathsf{F}=e^{-\mathsf{F}/x}L_{22}'L_{2,2}e^{\mathsf{F}/x}.
\]

Then one can see that
\[
N_\mathsf{F}\in \Psi_{sc}^{-1,0}(X;\mathcal{B}S_2^2{^{sc}T}X,\mathcal{B}S_2^2{^{sc}T}X).
\]
At the front face $x=0$,  the kernel of $N_\mathsf{F}$ is
\[
\begin{split}
&K^\flat(0,y,X,Y)\\
=&e^{-\mathsf{F}X}|Y|^{-n+1}\chi(S)\left(\mathrm{Id}-\left(\hat{Y}\cdot(x\partial_y)\right)\left(S\frac{\mathrm{d}x}{x^2}+\hat{Y}\cdot\frac{\mathrm{d}y}{x}\right)\right)^{\otimes 2}\otimes \left(S\frac{\mathrm{d}x}{x^2}+\hat{Y}\cdot\frac{\mathrm{d}y}{x}\right)^{\otimes 2}\mathcal{T}(0,y,X,Y)\\
&\left(\left(\mathrm{Id}-\left((S+2\alpha |Y|)(x^2\partial_x)+\hat{Y}\cdot(x\partial_y)\right)\left(\hat{Y}\cdot\frac{\mathrm{d}y}{x}\right)\right)^{\otimes 2}\otimes\left((S+2\alpha |Y|)(x^2\partial_x)+\hat{Y}\cdot(x\partial_y)\right)^{\otimes 2}\right).
\end{split}
\]
Notice that
\[
\left(\mathrm{Id}-\left((S+2\alpha |Y|)(x^2\partial_x)+\hat{Y}\cdot(x\partial_y)\right)\left(\hat{Y}\cdot\frac{\mathrm{d}y}{x}\right)\right)^{\otimes 2}\otimes\left((S+2\alpha |Y|)(x^2\partial_x)+\hat{Y}\cdot(x\partial_y)\right)^{\otimes 2}
\]
maps a symmetric $(2,2)$-tensor to a scattering symmetric $(2,0)$-tensor, and
\[
\left(\mathrm{Id}-\left(\hat{Y}\cdot(x\partial_y)\right)\left(S\frac{\mathrm{d}x}{x^2}+\hat{Y}\cdot\frac{\mathrm{d}y}{x}\right)\right)^{\otimes 2}\otimes \left(S\frac{\mathrm{d}x}{x^2}+\hat{Y}\cdot\frac{\mathrm{d}y}{x}\right)^{\otimes 2}
\]
maps a symmetric $(2,0)$-tensor to a scattering symmetric $(2,2)$-tensor.\\

 We use
\[
\begin{array}{rrrr}
(x^2\partial_x)\otimes (x^2\partial_x)\otimes\frac{\mathrm{d}x}{x^2}\otimes\frac{\mathrm{d}x}{x^2},&\quad 2(x^2\partial_x)\otimes (x^2\partial_x)\otimes\frac{\mathrm{d}x}{x^2}\otimes^s\frac{\mathrm{d}y}{x},&\quad (x^2\partial_x)\otimes (x^2\partial_x)\otimes\frac{\mathrm{d}y}{x}\otimes\frac{\mathrm{d}y}{x},\\
 2(x^2\partial_x)\otimes^s (x\partial_y)\otimes\frac{\mathrm{d}x}{x^2}\otimes\frac{\mathrm{d}x}{x^2},&\quad 4(x^2\partial_x)\otimes^s (x\partial_y)\otimes\frac{\mathrm{d}x}{x^2}\otimes^s\frac{\mathrm{d}y}{x},&\quad 2(x^2\partial_x)\otimes^s (x\partial_y)\otimes\frac{\mathrm{d}y}{x}\otimes\frac{\mathrm{d}y}{x}, \\
  (x\partial_y)\otimes (x\partial_y)\otimes\frac{\mathrm{d}x}{x^2}\otimes\frac{\mathrm{d}x}{x^2},&\quad  2(x\partial_y)\otimes (x\partial_y)\otimes\frac{\mathrm{d}x}{x^2}\otimes^s\frac{\mathrm{d}y}{x},&\quad  (x\partial_y)\otimes (x\partial_y)\otimes\frac{\mathrm{d}y}{x}\otimes\frac{\mathrm{d}y}{x},
  \end{array}
\]
as a basis for symmetric $(2,2)$-tensors. 
Then the symmetry of the matrix below needs to be checked with respect to the inner product
\[
\mathrm{diag}(1,2\times\mathrm{Id},\mathrm{Id},2\times\mathrm{Id},4\times\mathrm{Id},2\times\mathrm{Id},\mathrm{Id},2\times\mathrm{Id},\mathrm{Id}).
\]

Then the kernel of $N_\mathsf{F}$ at the front face $x=0$ can be written as 
 \[
 {\scriptsize
  \begin{split}
 e^{-\mathsf{F}X}\chi(S)|Y|^{-n+1}&\left(\begin{array}{ccc}
 S^2 &&\\
 S\hat{Y}_2 &&\\
\hat{Y}_2\otimes\hat{Y}_2&&\\
 & S^2 &\\
& S\hat{Y}_2&\\
& \hat{Y}_2\otimes\hat{Y}_2&\\
  &&S^2 \\
 &&S\hat{Y}_2\\
 &&\hat{Y}_2\otimes\hat{Y}_2\\
 \end{array}
 \right)\\
 &\left(\begin{array}{ccc}
1 & 0 &0\\
-S\hat{Y}_1 &\mathrm{Id}-\hat{Y}_1\langle\hat{Y}_1,\,\cdot\,\rangle &0\\
S^2\hat{Y}\otimes\hat{Y}_1 &-2S\hat{Y}_1\otimes^s+2S\hat{Y}_1\otimes\hat{Y}_1\langle\hat{Y}_1,\,\cdot\,\rangle &\mathrm{Id}-2\hat{Y}_1\otimes^s\langle\hat{Y}_1,\,\cdot\,\rangle+\hat{Y}_1\otimes\hat{Y}_1\langle\hat{Y}_1\otimes\hat{Y}_1,\,\cdot\,\rangle
\end{array}\right)\\
&\left(\begin{array}{ccc}
1 &-2(S+2\alpha|Y|)\langle\hat{Y}_1,\cdot\rangle & (S+2\alpha|Y|)^2\langle\hat{Y}_1\otimes\hat{Y}_1,\cdot\rangle\\
0&\mathrm{Id}-\hat{Y}_1\langle\hat{Y}_1,\cdot\rangle &-(S+2\alpha|Y|)\langle\hat{Y}_1,\cdot\rangle+(S+2\alpha|Y|)\hat{Y}_1\langle\hat{Y}_1\otimes\hat{Y}_1,\cdot\rangle\\
0&0&\mathrm{Id}-2\hat{Y}_1\otimes^s\langle\hat{Y}_1,\cdot\rangle+\hat{Y}_1\otimes\hat{Y}_1\langle\hat{Y}_1\otimes\hat{Y}_1,\cdot\rangle
\end{array}\right)\\
&\left(\begin{array}{ccccccccc}
 (S+2\alpha|Y|)^2 &&\\
 2(S+2\alpha|Y|)\langle\hat{Y}_2,\,\cdot\,\rangle &&\\
 \langle\hat{Y}_2\otimes\hat{Y}_2,\,\cdot\,\rangle&&\\
 & (S+2\alpha|Y|)^2 &\\
& 2(S+2\alpha|Y|)\langle\hat{Y}_2 ,\,\cdot\,\rangle&\\
& \langle\hat{Y}_2\otimes\hat{Y}_2,\,\cdot\,\rangle&\\
  &&(S+2\alpha|Y|)^2 \\
 &&2(S+2\alpha|Y|)\langle\hat{Y}_2,\,\cdot\,\rangle \\
 &&\langle\hat{Y}_2\otimes\hat{Y}_2,\,\cdot\,\rangle\\
 \end{array}
 \right)^T.
 \end{split}
 }
 \]
 Here subscripts $1$ and $2$ represent they are acting on upper (covariant) and lower (contravariant) indices respectively.\\
 
 We only focus on the case when the dimension $n=3$.
\subsection{Ellipticity}
Define
\[
\mathrm{d}'_\mathsf{F}=e^{-\mathsf{F}/x}\mathrm{d}'e^{\mathsf{F}/x}:C^\infty(\mathcal{B}S^2_1{^{sc}T}X)\rightarrow C^\infty(S^2_2{^{sc}T}X),
\]
and
\[
\mathrm{d}^\mathcal{B}_\mathsf{F}=e^{-\mathsf{F}/x}\mathrm{d}^\mathcal{B}e^{\mathsf{F}/x}=\mathcal{B}\mathrm{d}'_\mathsf{F}:C^\infty(\mathcal{B}S^2_1{^{sc}T}X)\rightarrow C^\infty(\mathcal{B}S^2_2{^{sc}T}X),
\]
Let $-\delta$ be the formal adjoint of the operator $\mathrm{d}'$ with respect to $g_{sc}$ (not $g$).
Denote
\[
\delta_\mathsf{F}^\mathcal{B}:C^\infty(\mathcal{B}S^2_2{^{sc}T}X)\rightarrow C^\infty(\mathcal{B}S^2_1{^{sc}T}X),\quad \delta_\mathsf{F}^\mathcal{B}=\delta_\mathsf{F}=e^{\mathsf{F}/x}\delta e^{-\mathsf{F}/x}.
\]

 \begin{lemma}
Assume $n=3$. The boundary symbol of $N_\mathsf{F}\in \Psi_{sc}^{-1,0}(X;\mathcal{B}S^2_2{^{sc}T}X,\mathcal{B}S^2_2{^{sc}T}X)$ is elliptic at fiber infinity of $^{sc}T^*X$ when restricted to the kernel of the principal symbol of $\delta^\mathcal{B}_\mathsf{F}$.
\end{lemma}
 \begin{proof}
 Similar as the proof of Lemma \ref{ellipticity20fiber}, we need to show that the integral of the matrix 
 \begin{equation}\label{matrix22_fiberinfinity}
 {\scriptsize
   \begin{split}
   & \chi(\tilde{S})\left(\begin{array}{ccc}
 \tilde{S}^2 &&\\
 \tilde{S}\hat{Y}_2 &&\\
\hat{Y}_2\otimes\hat{Y}_2&&\\
 & \tilde{S}^2 &\\
& \tilde{S}\hat{Y}_2&\\
& \hat{Y}_2\otimes\hat{Y}_2&\\
  &&\tilde{S}^2 \\
 &&\tilde{S}\hat{Y}_2 \\
 &&\hat{Y}_2\otimes\hat{Y}_2\\
 \end{array}
 \right)\\
&\left(\begin{array}{ccc}
1 & 0 &0\\
-\tilde{S}\hat{Y}_1 &\mathrm{Id}-\hat{Y}_1\langle\hat{Y}_1,\,\cdot\,\rangle &0\\
\tilde{S}^2\hat{Y}_1\otimes\hat{Y}_1 &-2\tilde{S}\hat{Y}_1\otimes^s+2\tilde{S}\hat{Y}_1\otimes\hat{Y}_1\langle\hat{Y}_1,\,\cdot\,\rangle &\mathrm{Id}-2\hat{Y}_1\otimes^s\langle\hat{Y}_1,\,\cdot\,\rangle+\hat{Y}_1\otimes\hat{Y}_1\langle\hat{Y}_1\otimes\hat{Y}_1,\,\cdot\,\rangle
\end{array}\right)\\
&\left(\begin{array}{ccc}
1 &-2\tilde{S}\langle\hat{Y}_1,\cdot\rangle & \tilde{S}^2\langle\hat{Y}_1\otimes\hat{Y}_1,\cdot\rangle\\
0&\mathrm{Id}-\hat{Y}_1\langle\hat{Y}_1,\cdot\rangle &-\tilde{S}\langle\hat{Y}_1,\cdot\rangle+\tilde{S}\hat{Y}_1\langle\hat{Y}_1\otimes\hat{Y}_1,\cdot\rangle\\
0&0&\mathrm{Id}-2\hat{Y}_1\otimes^s\langle\hat{Y}_1,\cdot\rangle+\hat{Y}_1\otimes\hat{Y}_1\langle\hat{Y}_1\otimes\hat{Y}_1,\cdot\rangle
\end{array}\right)\left(\begin{array}{ccc}
 \tilde{S}^2 &&\\
 2\tilde{S}\langle\hat{Y}_2,\,\cdot\,\rangle &&\\
 \langle\hat{Y}_2\otimes\hat{Y}_2,\,\cdot\,\rangle&&\\
 & \tilde{S}^2 &\\
& 2\tilde{S}\langle\hat{Y}_2 ,\,\cdot\,\rangle&\\
& \langle\hat{Y}_2\otimes\hat{Y}_2,\,\cdot\,\rangle&\\
  &&\tilde{S}^2 \\
 &&2\tilde{S}\langle\hat{Y}_2,\,\cdot\,\rangle \\
 &&\langle\hat{Y}_2\otimes\hat{Y}_2,\,\cdot\,\rangle\\
 \end{array}
 \right)^T,
  \end{split}
  }
\end{equation}
along the equatorial sphere $\tilde{S}\xi+\hat{Y}\cdot\eta=0$ corresponding to any $\zeta=(\xi,\eta)\neq 0$, is positive definite on the kernel of the principal symbol of $\delta^\mathcal{B}_\mathsf{F}$ at fiber infinity, where $\tilde{S}=\frac{X}{|Y|}$. It is clearly positive semidefinite. 

Assume first that $f$ is in the kernel of the principal symbol of $\delta^\mathcal{B}_\mathsf{F}$ at fiber infinity, that is
\begin{equation}\label{deltakernel22}
\begin{split}
\xi f^{xx}_{xx}+\langle\eta_2,f^{xx}_{xy}\rangle=0,\\
\xi f^{xy}_{xx}+\langle\eta_2,f^{xy}_{xy}\rangle=0,\\
\xi f^{yy}_{xx}+\langle\eta_2,f^{yy}_{xy}\rangle=0,\\
\xi f^{xx}_{xy}+\langle\eta_2,f^{xx}_{yy}\rangle=0,\\
\xi f^{xy}_{xy}+\langle\eta_2,f^{xy}_{yy}\rangle=0,\\
\xi f^{yy}_{xy}+\langle\eta_2,f^{yy}_{yy}\rangle=0.
\end{split}
\end{equation}

Now assume that $f$ is in the kernel of the matrix \eqref{matrix22_fiberinfinity} for $\tilde{S}$ sufficiently close to $0$ such that $\chi(\tilde{S})=1$.
First taking $\tilde{S}=0$ and $\hat{Y}=\hat{\eta}^\perp$ perpendicular to $\eta$, we have
\begin{eqnarray}
\langle\hat{\eta}_2^\perp\otimes\hat{\eta}_2^\perp,f^{xx}_{yy}\rangle=0,\label{S0_identity_1}\\
\langle\hat{\eta}_2^\perp\otimes \hat{\eta}_2^\perp,f^{xy}_{yy}\rangle-\hat{\eta}_1^\perp\langle\hat{\eta}^\perp_1\otimes \hat{\eta}^\perp_2\otimes \hat{\eta}^\perp_2,f^{xy}_{yy}\rangle=0,\label{S0_identity_2}\\
\langle\hat{\eta}_2^\perp\otimes \hat{\eta}_2^\perp,f^{yy}_{yy}\rangle-2\hat{\eta}_1^\perp\otimes^s\langle \hat{\eta}_1^\perp\otimes\hat{\eta}^\perp_2\otimes \hat{\eta}_2^\perp,f^{yy}_{yy}\rangle+\hat{\eta}_1^\perp\otimes \hat{\eta}_1^\perp\langle \hat{\eta}_1^\perp\otimes\hat{\eta}_1^\perp\otimes \hat{\eta}_2^\perp\otimes \hat{\eta}_2^\perp,f^{yy}_{yy}\rangle=0,\label{S0_identity_3}
\end{eqnarray}
and so
\[
\begin{split}
\langle\eta_1\otimes\hat{\eta}_2^\perp\otimes \hat{\eta}_2^\perp,f^{xy}_{yy}\rangle=0,\\
\langle\eta_1\otimes\eta_1\otimes\hat{\eta}_2^\perp\otimes \hat{\eta}_2^\perp,f^{yy}_{yy}\rangle=0.
\end{split}
\]

We also have
\[
\begin{split}
\left(\begin{array}{ccc}
1 &-2\tilde{S}\langle\hat{Y}_1,\,\cdot\,\rangle & \tilde{S}^2\langle\hat{Y}_1\otimes\hat{Y}_1,\,\cdot\,\rangle\\
0&\mathrm{Id}-\hat{Y}\langle\hat{Y}_1,\,\cdot\,\rangle &-\tilde{S}\langle\hat{Y}_1,\,\cdot\,\rangle+\tilde{S}\hat{Y}_1\langle\hat{Y}_1\otimes\hat{Y}_1,\,\cdot\,\rangle\\
0&0&\mathrm{Id}-2\hat{Y}_1\otimes^s\langle\hat{Y}_1,\,\cdot\,\rangle+\hat{Y}_1\otimes\hat{Y}\langle\hat{Y}_1\otimes\hat{Y}_1,\,\cdot\,\rangle
\end{array}\right)\\
\left(\begin{array}{c}
\tilde{S}^2f^{xx}_{xx}+2\tilde{S}\langle\hat{Y}_2,f^{xx}_{xy}\rangle+\langle\hat{Y}_2\otimes\hat{Y}_2,f^{xx}_{yy}\rangle\\
\tilde{S}^2f^{xy}_{xx}+2\tilde{S}\langle\hat{Y}_2,f^{xy}_{xy}\rangle+\langle\hat{Y}_2\otimes\hat{Y}_2,f^{xy}_{yy}\rangle\\
\tilde{S}^2f^{yy}_{xx}+2\tilde{S}\langle\hat{Y}_2,f^{yy}_{xy}\rangle+\langle\hat{Y}_2\otimes\hat{Y}_2,f^{yy}_{yy}\rangle
\end{array}
\right)=0.
\end{split}
\]
Applying
\[
\left(\begin{array}{ccc}
1 & &\\
&\langle \hat{Y}^\perp_1,\cdot\rangle&\\
&&\langle \hat{Y}^\perp_1\otimes\hat{Y}^\perp_1,\,\cdot\,\rangle
\end{array}
\right)
\]
from the left side of above matrix , we obtain
\[
\begin{split}
\left(\begin{array}{ccc}
1 & -2\tilde{S}\langle\hat{Y}_1,\,\cdot\,\rangle &\tilde{S}^2\langle\hat{Y}_1\otimes\hat{Y}_1,\,\cdot\,\rangle\\
0&\langle \hat{Y}_1^\perp,\,\cdot\,\rangle& -\tilde{S}\langle\hat{Y}_1^\perp\otimes\hat{Y}_1,\,\cdot\,\rangle\\
0 & 0 &\langle \hat{Y}_1^\perp\otimes \hat{Y}_1^\perp,\,\cdot\,\rangle
\end{array}
\right)
\left(\begin{array}{c}
\tilde{S}^2f^{xx}_{xx}+2\tilde{S}\langle\hat{Y}_2,f^{xx}_{xy}\rangle+\langle\hat{Y}_2\otimes\hat{Y}_2,f^{xx}_{yy}\rangle\\
\tilde{S}^2f^{xy}_{xx}+2\tilde{S}\langle\hat{Y}_2,f^{xy}_{xy}\rangle+\langle\hat{Y}_2\otimes\hat{Y}_2,f^{xy}_{yy}\rangle\\
\tilde{S}^2f^{yy}_{xx}+2\tilde{S}\langle\hat{Y}_2,f^{yy}_{xy}\rangle+\langle\hat{Y}_2\otimes\hat{Y}_2,f^{yy}_{yy}\rangle
\end{array}
\right)=0.
\end{split}
\]
~\\
%
~\\
\textbf{Case 1: $\xi\neq 0$.} The case $\eta=0$ is simple as in previous section, so we only need to consider the case $\eta\neq 0$. Now $\{\hat{\eta},\hat{\eta}^\perp\}$ forms an orthonormal basis for $\mathbb{R}^2$.\\

\text{Step 1.} We start with the identity
\begin{equation}\label{step1eq}
\tilde{S}^2\langle \hat{Y}^\perp_1\otimes \hat{Y}^\perp_1,f^{yy}_{xx}\rangle+2\tilde{S}\langle\hat{Y}^\perp_1\otimes \hat{Y}^\perp_1\otimes\hat{Y}_2,f^{yy}_{xy}\rangle+\langle\hat{Y}_1^\perp\otimes \hat{Y}^\perp_1\otimes\hat{Y}_2\otimes\hat{Y}_2,f^{yy}_{yy}\rangle=0.
\end{equation}
Taking $\hat{Y}=\epsilon\hat{\eta}+\sqrt{1-\epsilon^2}\hat{\eta}^\perp$, $\hat{Y}^\perp=\epsilon\hat{\eta}^\perp-\sqrt{1-\epsilon^2}\hat{\eta}$ and $\tilde{S}=-\frac{\epsilon|\eta|}{\xi}$, we obtain
\begin{equation}\label{kernel_identity_22_1}
{\scriptsize
\begin{split}
\epsilon^4\frac{|\eta|^2}{\xi^2}\langle\hat{\eta}_1^\perp\otimes\hat{\eta}_1^\perp,f^{yy}_{xx}\rangle-2\epsilon^3\sqrt{1-\epsilon^2}\frac{|\eta|^2}{\xi^2}\langle\hat{\eta}_1\otimes\hat{\eta}_1^\perp,f^{yy}_{xx}\rangle+\epsilon^2(1-\epsilon^2)\frac{|\eta|^2}{\xi^2}\langle\hat{\eta}_1\otimes\hat{\eta}_1,f^{yy}_{xx}\rangle\\
-2\epsilon^4\frac{|\eta|}{\xi}\langle\hat{\eta}_1^\perp\otimes\hat{\eta}_1^\perp\otimes\hat{\eta}_2,f^{yy}_{xy}\rangle-2\epsilon^3\sqrt{1-\epsilon^2}\frac{|\eta|}{\xi}\langle\hat{\eta}_1^\perp\otimes\hat{\eta}_1^\perp\otimes\hat{\eta}_2^\perp,f^{yy}_{xy}\rangle\\
+4\epsilon^3\sqrt{1-\epsilon^2}\frac{|\eta|}{\xi}\langle\hat{\eta}_1\otimes\hat{\eta}_1^\perp\otimes\hat{\eta}_2,f^{yy}_{xy}\rangle+4\epsilon^2(1-\epsilon^2)\frac{|\eta|}{\xi}\langle\hat{\eta}_1\otimes\hat{\eta}_1^\perp\otimes\hat{\eta}_2^\perp,f^{yy}_{xy}\rangle\\
-2\epsilon^2(1-\epsilon^2)\frac{|\eta|}{\xi}\langle\hat{\eta}_1\otimes\hat{\eta}_1\otimes\hat{\eta}_2,f^{yy}_{xy}\rangle-2\epsilon\sqrt{(1-\epsilon^2)^3}\frac{|\eta|}{\xi}\langle\hat{\eta}_1\otimes\hat{\eta}_1\otimes\hat{\eta}_2^\perp,f^{yy}_{xy}\rangle\\
+\epsilon^4\langle\hat{\eta}_1^\perp\otimes\hat{\eta}_1^\perp\otimes\hat{\eta}_2\otimes{\eta}_2,f^{yy}_{yy}\rangle+2\epsilon^3\sqrt{1-\epsilon^2}\langle\hat{\eta}_1^\perp\otimes\hat{\eta}_1^\perp\otimes\hat{\eta}_2\otimes{\eta}_2^\perp,f^{yy}_{yy}\rangle+\epsilon^2(1-\epsilon^2)\langle\hat{\eta}_1^\perp\otimes\hat{\eta}_1^\perp\otimes\hat{\eta}^\perp_2\otimes{\eta}^\perp_2,f^{yy}_{yy}\rangle\\
-2\epsilon^3\sqrt{1-\epsilon^2}\langle\hat{\eta}_1\otimes\hat{\eta}^\perp_1\otimes\hat{\eta}_2\otimes{\eta}_2,f^{yy}_{yy}\rangle-4\epsilon^2(1-\epsilon^2)\langle\hat{\eta}_1\otimes\hat{\eta}^\perp_1\otimes\hat{\eta}_2\otimes{\eta}_2^\perp,f^{yy}_{yy}\rangle-2\epsilon\sqrt{(1-\epsilon^2)^3}\langle\hat{\eta}_1\otimes\hat{\eta}^\perp_1\otimes\hat{\eta}^\perp_2\otimes{\eta}^\perp_2,f^{yy}_{yy}\rangle\\
+\epsilon^2(1-\epsilon^2)\langle\hat{\eta}_1\otimes\hat{\eta}_1\otimes\hat{\eta}_2\otimes{\eta}_2,f^{yy}_{yy}\rangle+2\epsilon\sqrt{(1-\epsilon^2)^3}\langle\hat{\eta}_1\otimes\hat{\eta}_1\otimes\hat{\eta}_2\otimes{\eta}_2^\perp,f^{yy}_{yy}\rangle+(1-\epsilon^2)^2\langle\hat{\eta}_1\otimes\hat{\eta}_1\otimes\hat{\eta}^\perp_2\otimes{\eta}^\perp_2,f^{yy}_{yy}\rangle=0.
\end{split}
}
\end{equation}
Letting $\epsilon=0$, we have
\[
\langle\hat{\eta}_1\otimes\hat{\eta}_1\otimes\hat{\eta}^\perp_2\otimes{\eta}^\perp_2,f^{yy}_{yy}\rangle=0.
\]
Taking 1st order derivative of \eqref{kernel_identity_22_1} in $\epsilon$ at $\epsilon=0$, we obtain
\[
-2\frac{|\eta|}{\xi}\langle\hat{\eta}_1\otimes\hat{\eta}_1\otimes\hat{\eta}_2^\perp, f^{yy}_{xy}\rangle+2\langle\hat{\eta}_1\otimes\hat{\eta}_1\otimes\hat{\eta}_2\otimes\hat{\eta}_2^\perp,f^{yy}_{yy}\rangle-2\langle\hat{\eta}_1\otimes\hat{\eta}^\perp_1\otimes \hat{\eta}^\perp_2\otimes\hat{\eta}^\perp_2,f^{yy}_{yy}\rangle=0.
\]
Using the condition in \eqref{deltakernel22}, the above equality reduces to
\begin{equation}\label{eq_37}
\left(\frac{|\eta|^2}{\xi^2}+1\right)\langle\hat{\eta}_1\otimes\hat{\eta}_1\otimes\hat{\eta}_2\otimes\hat{\eta}_2^\perp,f^{yy}_{yy}\rangle-\langle\hat{\eta}_1\otimes\hat{\eta}^\perp_1\otimes \hat{\eta}^\perp_2\otimes\hat{\eta}^\perp_2,f^{yy}_{yy}\rangle=0.
\end{equation}
Taking 3rd order derivative and using the above identity, we have
\[
\begin{split}
-2\frac{|\eta|^2}{\xi^2}\langle\hat{\eta}_1\otimes\hat{\eta}_1^\perp,f^{yy}_{xx}\rangle-2\frac{|\eta|}{\xi}\langle\hat{\eta}_1^\perp\otimes\hat{\eta}^\perp_1\otimes\hat{\eta}_2,f^{yy}_{xy}\rangle+4\frac{|\eta|}{\xi}\langle\hat{\eta}_1\otimes\hat{\eta}^\perp_1\otimes\hat{\eta}_2,f^{yy}_{xy}\rangle\\
+2\langle\hat{\eta}_1^\perp\otimes\hat{\eta}_1^\perp\otimes\hat{\eta}_2\otimes{\eta}_2^\perp,f^{yy}_{yy}\rangle-2\langle\hat{\eta}_1\otimes\hat{\eta}^\perp_1\otimes\hat{\eta}_2\otimes{\eta}_2,f^{yy}_{yy}\rangle=0,
\end{split}
\]
and, using \eqref{deltakernel22},
\begin{equation}\label{eq_38}
\left(\frac{|\eta|^4}{\xi^4}+2\frac{|\eta|^2}{\xi^2}+1\right)\langle\hat{\eta}_1\otimes\hat{\eta}^\perp_1\otimes\hat{\eta}_2\otimes\hat{\eta}_2,f^{yy}_{yy}\rangle-\left(\frac{|\eta|^2}{\xi^2}+1\right)\langle\hat{\eta}_1^\perp\otimes\hat{\eta}_1^\perp\otimes\hat{\eta}_2\otimes{\eta}_2^\perp,f^{yy}_{yy}\rangle=0.
\end{equation}
Taking 2nd order derivative, together with \eqref{deltakernel22}, gives
\begin{equation}\label{ttt_ss1}
\begin{split}
\left(\frac{|\eta|^4}{\xi^4}+2\frac{|\eta|^2}{\xi^2}+1\right)\langle\hat{\eta}_1\otimes\hat{\eta}_1\otimes\hat{\eta}_2\otimes\hat{\eta}_2,f^{yy}_{yy}\rangle-4\left(\frac{|\eta|^2}{\xi^2}+1\right)\langle\hat{\eta}_1\otimes\hat{\eta}^\perp_1\otimes\hat{\eta}_2\otimes\hat{\eta}^\perp_2, f^{yy}_{yy}\rangle\\
+\langle\hat{\eta}^\perp_1\otimes\hat{\eta}^\perp_1\otimes\hat{\eta}^\perp_2\otimes\hat{\eta}^\perp_2, f^{yy}_{yy}\rangle=0.
\end{split}
\end{equation}
Taking 4th order derivative, together with \eqref{deltakernel22}, we obtain
\[
\left(\frac{|\eta|^4}{\xi^4}+2\frac{|\eta|^2}{\xi^2}+1\right)\langle\hat{\eta}^\perp_1\otimes\hat{\eta}^\perp_1\otimes\hat{\eta}_2\otimes\hat{\eta}_2, f^{yy}_{yy}\rangle=0,
\]
and then
\[
\langle\hat{\eta}^\perp_1\otimes\hat{\eta}^\perp_1, f^{yy}_{xx}\rangle=\langle\hat{\eta}^\perp_1\otimes\hat{\eta}^\perp_1\otimes\hat{\eta}_2, f^{yy}_{xy}\rangle=\langle\hat{\eta}^\perp_1\otimes\hat{\eta}^\perp_1\otimes\hat{\eta}_2\otimes\hat{\eta}_2, f^{yy}_{yy}\rangle=0.
\]
~\\

\text{Step 2.} Then, we consider the identity
\begin{equation}\label{step2_eq}
\begin{split}
\tilde{S}^2\langle\hat{Y}^\perp_1,f^{xy}_{xx}\rangle+2\tilde{S}\langle\hat{Y}^\perp_1\otimes\hat{Y}_2,f^{xy}_{xy}\rangle+\langle\hat{Y}^\perp_1\otimes\hat{Y}_2\otimes\hat{Y}_2,f^{xy}_{yy}\rangle\\
-\tilde{S}^3\langle\hat{Y}^\perp_1\otimes\hat{Y}_1,f^{yy}_{xx}\rangle-2\tilde{S}^2\langle\hat{Y}^\perp_1\otimes\hat{Y}_1\otimes\hat{Y}_2,f^{yy}_{xy}\rangle-\tilde{S}\langle\hat{Y}^\perp_1\otimes\hat{Y}_1\otimes\hat{Y}_2\otimes\hat{Y}_2,f^{yy}_{yy}\rangle=0.
\end{split}
\end{equation}
Take $\hat{Y}=\epsilon\hat{\eta}+\sqrt{1-\epsilon^2}\hat{\eta}^\perp$, $\hat{Y}^\perp=\epsilon\hat{\eta}^\perp-\sqrt{1-\epsilon^2}\hat{\eta}$ and $\tilde{S}=-\frac{\epsilon|\eta|}{\xi}$.
Similar as above, we obtain
\begin{equation}\label{kernel_identity_22_2}
{\scriptsize
\begin{split}
\epsilon^3\frac{|\eta|^2}{\xi^2}\langle\hat{\eta}_1^\perp,f^{xy}_{xx}\rangle-\epsilon^2\sqrt{1-\epsilon^2}\frac{|\eta|^2}{\xi^2}\langle\hat{\eta}_1,f^{xy}_{xx}\rangle\\
-2\epsilon^3\frac{|\eta|}{\xi}\langle\hat{\eta}^\perp_1\otimes\hat{\eta}_2,f^{xy}_{xy}\rangle-2\epsilon^2\sqrt{1-\epsilon^2}\frac{|\eta|}{\xi}\langle\hat{\eta}^\perp_1\otimes\hat{\eta}^\perp_2-\hat{\eta}_1\otimes\hat{\eta}_2,f^{xy}_{xy}\rangle+2\epsilon(1-\epsilon^2)\frac{|\eta|}{\xi}\langle\hat{\eta}_1\otimes\hat{\eta}_2^\perp, f^{xy}_{xy}\rangle\\
+\epsilon^3\langle\hat{\eta}^\perp_1\otimes\hat{\eta}_2\otimes\hat{\eta}_2,f^{xy}_{yy}\rangle-\epsilon^2\sqrt{1-\epsilon^2}\langle \hat{\eta}_1\otimes\hat{\eta}_2\otimes\hat{\eta}_2,f^{xy}_{yy}\rangle
+2\epsilon^2\sqrt{1-\epsilon^2}\langle \hat{\eta}_1^\perp\otimes\hat{\eta}_2\otimes\hat{\eta}_2^\perp,f^{xy}_{yy}\rangle\\
+\epsilon(1-\epsilon^2)\langle\hat{\eta}^\perp_1\otimes \hat{\eta}^\perp_2\otimes\hat{\eta}^\perp_2,f^{xy}_{yy}\rangle-2\epsilon(1-\epsilon^2)\langle\hat{\eta}_1\otimes \hat{\eta}_2\otimes\hat{\eta}^\perp_2,f^{xy}_{yy}\rangle
-\sqrt{(1-\epsilon^2)^3}\langle\hat{\eta}_1\otimes\hat{\eta}_2^\perp\otimes\hat{\eta}_2^\perp,f^{xy}_{yy}\rangle\\
+\epsilon^5\frac{|\eta|^3}{\xi^3}\langle\hat{\eta}_1^\perp\otimes\hat{\eta}_1,f^{yy}_{xx}\rangle+\epsilon^4\sqrt{1-\epsilon^2}\frac{|\eta|^3}{\xi^3}\langle\hat{\eta}^\perp_1\otimes\hat{\eta}_1^\perp-\hat{\eta}_1\otimes\hat{\eta}_1,f^{yy}_{xx}\rangle-\epsilon^3(1-\epsilon^2)\frac{|\eta|^3}{\xi^3}\langle\hat{\eta}_1\otimes\hat{\eta}_1^\perp, f^{yy}_{xx}\rangle\\
-2\epsilon^5\frac{|\eta|^2}{\xi^2}\langle\hat{\eta}^\perp_1\otimes\hat{\eta}_1\otimes\hat{\eta}_2,f^{yy}_{xy}\rangle+2\epsilon^4\sqrt{1-\epsilon^2}\frac{|\eta|^2}{\xi^2}\langle \hat{\eta}_1\otimes\hat{\eta}_1\otimes\hat{\eta}_2,f^{yy}_{xy}\rangle
-2\epsilon^4\sqrt{1-\epsilon^2}\frac{|\eta|^2}{\xi^2}\langle \hat{\eta}_1^\perp\otimes\hat{\eta}_1\otimes\hat{\eta}_2^\perp,f^{yy}_{xy}\rangle\\
-2\epsilon^4\sqrt{1-\epsilon^2}\frac{|\eta|^2}{\xi^2}\langle \hat{\eta}_1^\perp\otimes\hat{\eta}_1^\perp\otimes\hat{\eta}_2,f^{yy}_{xy}\rangle
-2\epsilon^3(1-\epsilon^2)\frac{|\eta|^2}{\xi^2}\langle\hat{\eta}^\perp_1\otimes \hat{\eta}^\perp_1\otimes\hat{\eta}^\perp_2,f^{yy}_{xy}\rangle\\
+2\epsilon^3(1-\epsilon^2)\frac{|\eta|^2}{\xi^2}\langle\hat{\eta}_1\otimes \hat{\eta}_1\otimes\hat{\eta}^\perp_2,f^{yy}_{xy}\rangle+2\epsilon^3(1-\epsilon^2)\frac{|\eta|^2}{\xi^2}\langle\hat{\eta}_1\otimes \hat{\eta}^\perp_1\otimes\hat{\eta}_2,f^{yy}_{xy}\rangle\\
+2\epsilon^2\sqrt{(1-\epsilon^2)^3}\frac{|\eta|^2}{\xi^2}\langle\hat{\eta}_1\otimes\hat{\eta}_1^\perp\otimes\hat{\eta}_2^\perp,f^{yy}_{xy}\rangle
+(-\epsilon^3(1-\epsilon^2)+\epsilon^5)\frac{|\eta|}{\xi}\langle\hat{\eta}_1\otimes\hat{\eta}_1^\perp\otimes\hat{\eta}_2\otimes\hat{\eta}_2,f^{yy}_{yy}\rangle\\
-\epsilon^4\sqrt{1-\epsilon^2}\frac{|\eta|}{\xi}\langle\hat{\eta}_1\otimes\hat{\eta}_1\otimes\hat{\eta}_2\otimes\hat{\eta}_2,f^{yy}_{yy}\rangle+\epsilon^4\sqrt{1-\epsilon^2}\frac{|\eta|}{\xi}\langle\hat{\eta}^\perp_1\otimes\hat{\eta}^\perp_1\otimes\hat{\eta}_2\otimes\hat{\eta}_2,f^{yy}_{yy}\rangle\\
+(-2\epsilon^2(1-\epsilon^2)+2\epsilon^4)\sqrt{1-\epsilon^2}\frac{|\eta|}{\xi}\langle\hat{\eta}_1\otimes\hat{\eta}_1^\perp\otimes\hat{\eta}_2\otimes\hat{\eta}^\perp_2,f^{yy}_{yy}\rangle\\-2\epsilon^3(1-\epsilon^2)\frac{|\eta|}{\xi}\langle\hat{\eta}_1\otimes\hat{\eta}_1\otimes\hat{\eta}_2\otimes\hat{\eta}^\perp_2,f^{yy}_{yy}\rangle+2\epsilon^3(1-\epsilon^2)\frac{|\eta|}{\xi}\langle\hat{\eta}^\perp_1\otimes\hat{\eta}^\perp_1\otimes\hat{\eta}_2\otimes\hat{\eta}^\perp_2,f^{yy}_{yy}\rangle\\
+(-\epsilon(1-\epsilon^2)^2+\epsilon^3(1-\epsilon^2))\frac{|\eta|}{\xi}\langle\hat{\eta}_1\otimes\hat{\eta}_1^\perp\otimes\hat{\eta}^\perp_2\otimes\hat{\eta}^\perp_2,f^{yy}_{yy}\rangle-\epsilon^2\sqrt{(1-\epsilon^2)^3}\frac{|\eta|}{\xi}\langle\hat{\eta}_1\otimes\hat{\eta}_1\otimes\hat{\eta}^\perp_2\otimes\hat{\eta}^\perp_2,f^{yy}_{yy}\rangle\\
+\epsilon^2\sqrt{(1-\epsilon^2)^3}\frac{|\eta|}{\xi}\langle\hat{\eta}^\perp_1\otimes\hat{\eta}^\perp_1\otimes\hat{\eta}^\perp_2\otimes\hat{\eta}^\perp_2,f^{yy}_{yy}\rangle=0.
\end{split}
}
\end{equation}
Setting $\epsilon=0$, we have
\[
\langle\hat{\eta}_1\otimes\hat{\eta}_2^\perp\otimes\hat{\eta}_2^\perp,f^{xy}_{yy}\rangle=0.
\]
Taking 1st order derivative of \eqref{kernel_identity_22_2} in $\epsilon$ at $\epsilon=0$, we obtain
\[
2\frac{|\eta|}{\xi}\langle\hat{\eta}_1\otimes\hat{\eta}_2^\perp, f^{xy}_{xy}\rangle-2\langle\hat{\eta}_1\otimes\hat{\eta}_2\otimes\hat{\eta}_2^\perp,f^{xy}_{yy}\rangle+\langle\hat{\eta}^\perp_1\otimes \hat{\eta}^\perp_2\otimes\hat{\eta}^\perp_2,f^{xy}_{yy}\rangle-\frac{|\eta|}{\xi}\langle\hat{\eta}_1\otimes\hat{\eta}_1^\perp\otimes\hat{\eta}^\perp_2\otimes\hat{\eta}^\perp_2,f^{yy}_{yy}\rangle=0.
\]
Using the condition in \eqref{deltakernel22}, the above equality reduces to
\begin{equation}\label{eq_41}
2\left(\frac{|\eta|^2}{\xi^2}+1\right)\langle\hat{\eta}_1\otimes\hat{\eta}_2\otimes\hat{\eta}_2^\perp,f^{xy}_{yy}\rangle-\langle\hat{\eta}^\perp_1\otimes \hat{\eta}^\perp_2\otimes\hat{\eta}^\perp_2,f^{xy}_{yy}\rangle+\frac{|\eta|}{\xi}\langle\hat{\eta}_1\otimes\hat{\eta}_1^\perp\otimes\hat{\eta}^\perp_2\otimes\hat{\eta}^\perp_2,f^{yy}_{yy}\rangle=0.
\end{equation}
Taking 3rd order derivative and using the above identity, we have
\[
\begin{split}
\frac{|\eta|^2}{\xi^2}\langle\hat{\eta}_1^\perp,f^{xy}_{xx}\rangle-2\frac{|\eta|}{\xi}\langle\hat{\eta}^\perp_1\otimes\hat{\eta}_2,f^{xy}_{xy}\rangle+\langle\hat{\eta}^\perp_1\otimes\hat{\eta}_2\otimes\hat{\eta}_2,f^{xy}_{yy}\rangle\\
-\frac{|\eta|^3}{\xi^3}\langle\hat{\eta}_1\otimes\hat{\eta}_1^\perp,f^{yy}_{xx}\rangle-2\frac{|\eta|^2}{\xi^2}\langle\hat{\eta}_1^\perp\otimes\hat{\eta}_1^\perp\otimes\hat{\eta}_2^\perp,f^{yy}_{xy}\rangle+2\frac{|\eta|^2}{\xi^2}\langle\hat{\eta}_1\otimes\hat{\eta}_1\otimes\hat{\eta}_2^\perp,f^{yy}_{xy}\rangle\\
+2\frac{|\eta|^2}{\xi^2}\langle\hat{\eta}_1\otimes\hat{\eta}_1^\perp\otimes\hat{\eta}_2,f^{yy}_{xy}\rangle-\frac{|\eta|}{\xi}\langle\hat{\eta}_1\otimes\hat{\eta}^\perp_1\otimes\hat{\eta}_2\otimes\hat{\eta}_2,f^{yy}_{yy}\rangle+2\frac{|\eta|}{\xi}\langle\hat{\eta}_1^\perp\otimes\hat{\eta}^\perp_1\otimes\hat{\eta}_2\otimes\hat{\eta}^\perp_2,f^{yy}_{yy}\rangle\\
-2\frac{|\eta|}{\xi}\langle\hat{\eta}_1\otimes\hat{\eta}_1\otimes\hat{\eta}_2\otimes\hat{\eta}_2^\perp,f^{yy}_{yy}\rangle+\frac{|\eta|}{\xi}\langle\hat{\eta}_1\otimes\hat{\eta}^\perp_1\otimes\hat{\eta}^\perp_2\otimes\hat{\eta}_2^\perp,f^{yy}_{yy}\rangle=0
\end{split}
\]
and, using \eqref{deltakernel22},
\begin{equation}\label{eq_42}
\begin{split}
\left(\frac{|\eta|^4}{\xi^4}+2\frac{|\eta|^2}{\xi^2}+1\right)\langle\hat{\eta}^\perp_1\otimes\hat{\eta}_2\otimes\hat{\eta}_2,f^{xy}_{yy}\rangle-\left(\frac{|\eta|^5}{\xi^5}+2\frac{|\eta|^3}{\xi^3}+\frac{|\eta|}{\xi}\right)\langle\hat{\eta}\otimes\hat{\eta}^\perp_1\otimes\hat{\eta}_2\otimes\hat{\eta}_2,f^{yy}_{yy}\rangle\\
+2\left(\frac{|\eta|^3}{\xi^3}+\frac{|\eta|}{\xi}\right)\langle\hat{\eta}_1^\perp\otimes\hat{\eta}^\perp_1\otimes\hat{\eta}_2\otimes\hat{\eta}^\perp_2,f^{yy}_{yy}\rangle-2\left(\frac{|\eta|^3}{\xi^3}+\frac{|\eta|}{\xi}\right)\langle\hat{\eta}_1\otimes\hat{\eta}_1\otimes\hat{\eta}_2\otimes\hat{\eta}^\perp_2,f^{yy}_{yy}\rangle\\
+\frac{|\eta|}{\xi}\langle\hat{\eta}_1\otimes\hat{\eta}^\perp_1\otimes\hat{\eta}^\perp_2\otimes\hat{\eta}_2^\perp,f^{yy}_{yy}\rangle=0.
\end{split}
\end{equation}

Taking 5th order derivative
\[
\frac{|\eta|^3}{\xi^3}\langle\hat{\eta}_1^\perp\otimes\hat{\eta}_1,f^{yy}_{xx}\rangle-2\frac{|\eta|^2}{\xi^2}\langle\hat{\eta}^\perp_1\otimes\hat{\eta}_1\otimes\hat{\eta}_2,f^{yy}_{xy}\rangle+\frac{|\eta|}{\xi}\langle\hat{\eta}_1\otimes\hat{\eta}_1^\perp\otimes\hat{\eta}_2\otimes\hat{\eta}_2,f^{yy}_{yy}\rangle=0.
\]
Using the gauge condition \eqref{deltakernel22}, we have
\[
\langle\hat{\eta}_1^\perp\otimes\hat{\eta}_1,f^{yy}_{xx}\rangle=\langle\hat{\eta}^\perp_1\otimes\hat{\eta}_1\otimes\hat{\eta}_2,f^{yy}_{xy}\rangle=\langle\hat{\eta}_1\otimes\hat{\eta}_1^\perp\otimes\hat{\eta}_2\otimes\hat{\eta}_2,f^{yy}_{yy}\rangle=0.
\]
Substituting into \eqref{eq_38}, we have
\[
\langle\hat{\eta}_1^\perp\otimes\hat{\eta}_1^\perp\otimes\hat{\eta}_2\otimes\hat{\eta}_2^\perp,f^{yy}_{yy}\rangle=0,
\]
and, using the gauge condition,
\[
\langle\hat{\eta}_1^\perp\otimes\hat{\eta}_1^\perp\otimes\hat{\eta}_2^\perp,f^{yy}_{xy}\rangle=0,
\]
Using the trace-free condition
\[
\langle\hat{\eta}_1^\perp\otimes\hat{\eta}_2,f^{xy}_{xy}\rangle+\langle\hat{\eta}_1\otimes\hat{\eta}_1^\perp\otimes\hat{\eta}_2\otimes\hat{\eta}_2,f^{yy}_{yy}\rangle+\langle\hat{\eta}_1^\perp\otimes\hat{\eta}_1^\perp\otimes\hat{\eta}_2\otimes{\eta}_2^\perp,f^{yy}_{yy}\rangle=0,
\]
we have
\[
\langle\hat{\eta}_1^\perp,f^{xy}_{xx}\rangle=\langle\hat{\eta}_1^\perp\otimes\hat{\eta}_2,f^{xy}_{xy}\rangle=\langle\hat{\eta}_1^\perp\otimes\hat{\eta}_2\otimes\hat{\eta}_2,f^{xy}_{yy}\rangle=0.
\]

Eliminating the vanishing terms in \eqref{eq_42}, we have
\[
-2\left(\frac{|\eta|^2}{\xi^2}+1\right)\langle\hat{\eta}_1\otimes\hat{\eta}_1\otimes\hat{\eta}_2\otimes\hat{\eta}^\perp_2,f^{yy}_{yy}\rangle\\
+\langle\hat{\eta}_1\otimes\hat{\eta}^\perp_1\otimes\hat{\eta}^\perp_2\otimes\hat{\eta}_2^\perp,f^{yy}_{yy}\rangle=0
\]
and together with \eqref{eq_37}, 
\[
\langle\hat{\eta}_1\otimes\hat{\eta}_1\otimes\hat{\eta}_2\otimes\hat{\eta}^\perp_2,f^{yy}_{yy}\rangle=\langle\hat{\eta}_1\otimes\hat{\eta}^\perp_1\otimes\hat{\eta}^\perp_2\otimes\hat{\eta}_2^\perp,f^{yy}_{yy}\rangle=0.
\]
Using the gauge condition, we also get
\[
\langle\hat{\eta}_1\otimes\hat{\eta}_1\otimes\hat{\eta}^\perp_2,f^{yy}_{xy}\rangle=0.
\]
Using the trace-free condition
\[
\langle\hat{\eta}_1\otimes\hat{\eta}_2^\perp,f^{xy}_{xy}\rangle+\langle\hat{\eta}_1\otimes\hat{\eta}_1\otimes\hat{\eta}_2\otimes\hat{\eta}^\perp_2,f^{yy}_{yy}\rangle+\langle\hat{\eta}_1\otimes\hat{\eta}_1^\perp\otimes\hat{\eta}_2^\perp\otimes{\eta}_2^\perp,f^{yy}_{yy}\rangle=0,
\]
we have
\[
\langle\hat{\eta}_1\otimes\hat{\eta}^\perp_2,f^{xy}_{xy}\rangle=\langle\hat{\eta}_1\otimes\hat{\eta}^\perp_2\otimes\hat{\eta}_2,f^{xy}_{yy}\rangle=0.
\]
Using \eqref{eq_41} again, we have
\[
\langle\hat{\eta}^\perp_1\otimes \hat{\eta}^\perp_2\otimes\hat{\eta}^\perp_2,f^{xy}_{yy}\rangle=0.
\]

Next we consider the even powers of $\epsilon$. Dividing by $\epsilon^2\sqrt{1-\epsilon^2}$, we obtain
\begin{equation}\label{kernel_identity_22_2_2}
\begin{split}
-\frac{|\eta|^2}{\xi^2}\langle\hat{\eta}_1,f^{xy}_{xx}\rangle
-2\frac{|\eta|}{\xi}\langle\hat{\eta}^\perp_1\otimes\hat{\eta}^\perp_2-\hat{\eta}_1\otimes\hat{\eta}_2,f^{xy}_{xy}\rangle
-\langle \hat{\eta}_1\otimes\hat{\eta}_2\otimes\hat{\eta}_2,f^{xy}_{yy}\rangle\\
+2\langle \hat{\eta}_1^\perp\otimes\hat{\eta}_2\otimes\hat{\eta}_2^\perp,f^{xy}_{yy}\rangle
+\epsilon^2\frac{|\eta|^3}{\xi^3}\langle\hat{\eta}^\perp_1\otimes\hat{\eta}_1^\perp-\hat{\eta}_1\otimes\hat{\eta}_1,f^{yy}_{xx}\rangle
+2\epsilon^2\frac{|\eta|^2}{\xi^2}\langle \hat{\eta}_1\otimes\hat{\eta}_1\otimes\hat{\eta}_2,f^{yy}_{xy}\rangle\\
-2\epsilon^2\frac{|\eta|^2}{\xi^2}\langle \hat{\eta}_1^\perp\otimes\hat{\eta}_1\otimes\hat{\eta}_2^\perp,f^{yy}_{xy}\rangle
-2\epsilon^2\frac{|\eta|^2}{\xi^2}\langle \hat{\eta}_1^\perp\otimes\hat{\eta}_1^\perp\otimes\hat{\eta}_2,f^{yy}_{xy}\rangle\\
+2(1-\epsilon^2)\frac{|\eta|^2}{\xi^2}\langle\hat{\eta}_1\otimes\hat{\eta}_1^\perp\otimes\hat{\eta}_2^\perp,f^{yy}_{xy}\rangle
-\epsilon^2\frac{|\eta|}{\xi}\langle\hat{\eta}_1\otimes\hat{\eta}_1\otimes\hat{\eta}_2\otimes\hat{\eta}_2,f^{yy}_{yy}\rangle\\
+\epsilon^2\frac{|\eta|}{\xi}\langle\hat{\eta}^\perp_1\otimes\hat{\eta}^\perp_1\otimes\hat{\eta}_2\otimes\hat{\eta}_2,f^{yy}_{yy}\rangle\
-2(1-2\epsilon^2)\frac{|\eta|}{\xi}\langle\hat{\eta}_1\otimes\hat{\eta}_1^\perp\otimes\hat{\eta}_2\otimes\hat{\eta}^\perp_2,f^{yy}_{yy}\rangle\\
-(1-\epsilon^2)\frac{|\eta|}{\xi}\langle\hat{\eta}_1\otimes\hat{\eta}_1\otimes\hat{\eta}^\perp_2\otimes\hat{\eta}^\perp_2,f^{yy}_{yy}\rangle
+(1-\epsilon^2)\frac{|\eta|}{\xi}\langle\hat{\eta}^\perp_1\otimes\hat{\eta}^\perp_1\otimes\hat{\eta}^\perp_2\otimes\hat{\eta}^\perp_2,f^{yy}_{yy}\rangle=0.
\end{split}
\end{equation}
Setting $\epsilon=0$, we have
\[
\begin{split}
-\frac{|\eta|^2}{\xi^2}\langle\hat{\eta}_1,f^{xy}_{xx}\rangle-2\frac{|\eta|}{\xi}\langle\hat{\eta}^\perp_1\otimes\hat{\eta}^\perp_2-\hat{\eta}_1\otimes\hat{\eta}_2,f^{xy}_{xy}\rangle-\langle \hat{\eta}_1\otimes\hat{\eta}_2\otimes\hat{\eta}_2,f^{xy}_{yy}\rangle+2\langle \hat{\eta}_1^\perp\otimes\hat{\eta}_2\otimes\hat{\eta}_2^\perp,f^{xy}_{yy}\rangle\\
+2\frac{|\eta|^2}{\xi^2}\langle\hat{\eta}_1\otimes\hat{\eta}_1^\perp\otimes\hat{\eta}_2^\perp,f^{yy}_{xy}\rangle-2\frac{|\eta|}{\xi}\langle\hat{\eta}_1\otimes\hat{\eta}_1^\perp\otimes\hat{\eta}_2\otimes\hat{\eta}^\perp_2,f^{yy}_{yy}\rangle-\frac{|\eta|}{\xi}\langle\hat{\eta}_1\otimes\hat{\eta}_1\otimes\hat{\eta}^\perp_2\otimes\hat{\eta}^\perp_2,f^{yy}_{yy}\rangle\\
+\frac{|\eta|}{\xi}\langle\hat{\eta}^\perp_1\otimes\hat{\eta}^\perp_1\otimes\hat{\eta}^\perp_2\otimes\hat{\eta}^\perp_2,f^{yy}_{yy}\rangle=0.
\end{split}
\]
Using the gauge condition, we end up with
\[
\begin{split}
-\left(\frac{|\eta|^4}{\xi^4}+2\frac{|\eta|^2}{\xi^2}+1\right)\langle\hat{\eta}_1\otimes\hat{\eta}_2\otimes\hat{\eta}_2,f^{xy}_{yy}\rangle-2\left(\frac{|\eta|}{\xi}+\frac{\xi}{|\eta|}\right)\langle\hat{\eta}^\perp_1\otimes\hat{\eta}^\perp_2,f^{xy}_{xy}\rangle\\
-2\left(\frac{|\eta|^3}{\xi^3}+\frac{|\eta|}{\xi}\right)\langle\hat{\eta}_1\otimes\hat{\eta}_1^\perp\otimes\hat{\eta}_2\otimes\hat{\eta}^\perp_2,f^{yy}_{yy}\rangle-\frac{|\eta|}{\xi}\langle\hat{\eta}_1\otimes\hat{\eta}_1\otimes\hat{\eta}^\perp_2\otimes\hat{\eta}^\perp_2,f^{yy}_{yy}\rangle\\
+\frac{|\eta|}{\xi}\langle\hat{\eta}^\perp_1\otimes\hat{\eta}^\perp_1\otimes\hat{\eta}^\perp_2\otimes\hat{\eta}^\perp_2,f^{yy}_{yy}\rangle=0.
\end{split}
\]
which can be rewritten as
\[
\begin{split}
\left(\frac{|\eta|^4}{\xi^4}+2\frac{|\eta|^2}{\xi^2}+1\right)\langle\hat{\eta}_1\otimes\hat{\eta}_2,f^{xy}_{xy}\rangle-2\left(\frac{|\eta|^2}{\xi^2}+1\right)\langle\hat{\eta}^\perp_1\otimes\hat{\eta}^\perp_2,f^{xy}_{xy}\rangle\\
-2\left(\frac{|\eta|^4}{\xi^4}+\frac{|\eta|^2}{\xi^2}\right)\langle\hat{\eta}_1\otimes\hat{\eta}_1^\perp\otimes\hat{\eta}_2\otimes\hat{\eta}^\perp_2,f^{yy}_{yy}\rangle
+\frac{|\eta|^2}{\xi^2}\langle\hat{\eta}^\perp_1\otimes\hat{\eta}^\perp_1\otimes\hat{\eta}^\perp_2\otimes\hat{\eta}^\perp_2,f^{yy}_{yy}\rangle=0.
\end{split}
\]
Taking 2nd order derivative gives
\begin{equation}\label{ttt_ss}
\begin{split}
\frac{|\eta|^3}{\xi^3}\langle\hat{\eta}^\perp_1\otimes\hat{\eta}_1^\perp-\hat{\eta}_1\otimes\hat{\eta}_1,f^{yy}_{xx}\rangle
+2\frac{|\eta|^2}{\xi^2}\langle \hat{\eta}_1\otimes\hat{\eta}_1\otimes\hat{\eta}_2,f^{yy}_{xy}\rangle\\
-2\frac{|\eta|^2}{\xi^2}\langle \hat{\eta}_1^\perp\otimes\hat{\eta}_1\otimes\hat{\eta}_2^\perp,f^{yy}_{xy}\rangle
-2\frac{|\eta|^2}{\xi^2}\langle \hat{\eta}_1^\perp\otimes\hat{\eta}_1^\perp\otimes\hat{\eta}_2,f^{yy}_{xy}\rangle
-2\frac{|\eta|^2}{\xi^2}\langle\hat{\eta}_1\otimes\hat{\eta}_1^\perp\otimes\hat{\eta}_2^\perp,f^{yy}_{xy}\rangle\\
-\frac{|\eta|}{\xi}\langle\hat{\eta}_1\otimes\hat{\eta}_1\otimes\hat{\eta}_2\otimes\hat{\eta}_2,f^{yy}_{yy}\rangle+\frac{|\eta|}{\xi}\langle\hat{\eta}^\perp_1\otimes\hat{\eta}^\perp_1\otimes\hat{\eta}_2\otimes\hat{\eta}_2,f^{yy}_{yy}\rangle\\
+4\frac{|\eta|}{\xi}\langle\hat{\eta}_1\otimes\hat{\eta}_1^\perp\otimes\hat{\eta}_2\otimes\hat{\eta}^\perp_2,f^{yy}_{yy}\rangle\\
+\frac{|\eta|}{\xi}\langle\hat{\eta}_1\otimes\hat{\eta}_1\otimes\hat{\eta}^\perp_2\otimes\hat{\eta}^\perp_2,f^{yy}_{yy}\rangle
-\frac{|\eta|}{\xi}\langle\hat{\eta}^\perp_1\otimes\hat{\eta}^\perp_1\otimes\hat{\eta}^\perp_2\otimes\hat{\eta}^\perp_2,f^{yy}_{yy}\rangle=0.
\end{split}
\end{equation}
Together with \eqref{deltakernel22}, we have
\[
\begin{split}
\left(\frac{|\eta|^5}{\xi^5}+2\frac{|\eta|^3}{\xi^3}+\frac{|\eta|}{\xi}\right)\langle\hat{\eta}^\perp_1\otimes\hat{\eta}^\perp_1\otimes\hat{\eta}_2\otimes\hat{\eta}_2,f^{yy}_{yy}\rangle-\left(\frac{|\eta|^5}{\xi^5}+2\frac{|\eta|^3}{\xi^3}+\frac{|\eta|}{\xi}\right)\langle\hat{\eta}_1\otimes\hat{\eta}_1\otimes\hat{\eta}_2\otimes\hat{\eta}_2,f^{yy}_{yy}\rangle\\
+4\left(\frac{|\eta|^3}{\xi^3}+\frac{|\eta|}{\xi}\right)\langle\hat{\eta}_1\otimes\hat{\eta}_1^\perp\otimes\hat{\eta}_2\otimes\hat{\eta}^\perp_2,f^{yy}_{yy}\rangle\\+\frac{|\eta|}{\xi}\langle\hat{\eta}_1\otimes\hat{\eta}_1\otimes\hat{\eta}^\perp_2\otimes\hat{\eta}^\perp_2,f^{yy}_{yy}\rangle
-\frac{|\eta|}{\xi}\langle\hat{\eta}^\perp_1\otimes\hat{\eta}^\perp_1\otimes\hat{\eta}^\perp_2\otimes\hat{\eta}^\perp_2,f^{yy}_{yy}\rangle=0.
\end{split}
\]
which further reduces to
\[
\begin{split}
-\left(\frac{|\eta|^5}{\xi^5}+2\frac{|\eta|^3}{\xi^3}+\frac{|\eta|}{\xi}\right)\langle\hat{\eta}_1\otimes\hat{\eta}_1\otimes\hat{\eta}_2\otimes\hat{\eta}_2,f^{yy}_{yy}\rangle
+4\left(\frac{|\eta|^3}{\xi^3}+\frac{|\eta|}{\xi}\right)\langle\hat{\eta}_1\otimes\hat{\eta}_1^\perp\otimes\hat{\eta}_2\otimes\hat{\eta}^\perp_2,f^{yy}_{yy}\rangle\\
-\frac{|\eta|}{\xi}\langle\hat{\eta}^\perp_1\otimes\hat{\eta}^\perp_1\otimes\hat{\eta}^\perp_2\otimes\hat{\eta}^\perp_2,f^{yy}_{yy}\rangle=0.
\end{split}
\]
~\\

Together with \eqref{ttt_ss1} and the trace-free conditions
\[
\langle\hat{\eta}^\perp\otimes\hat{\eta}^\perp,f^{xy}_{xy}\rangle+\langle\hat{\eta}_1\otimes\hat{\eta}_1^\perp\otimes\hat{\eta}_2^\perp\otimes \hat{\eta}_2, f^{yy}_{yy}\rangle+\langle\hat{\eta}_1^\perp\otimes\hat{\eta}_1^\perp\otimes\hat{\eta}_2^\perp\otimes\hat{\eta}_2^\perp,f^{yy}_{yy}\rangle=0,
\]
\[
\langle\hat{\eta}_1\otimes\hat{\eta}_2,f^{xy}_{xy}\rangle+\langle\hat{\eta}_1\otimes\hat{\eta}_1\otimes\hat{\eta}_2\otimes\hat{\eta}_2,f^{yy}_{yy}\rangle+\langle\hat{\eta}_1\otimes\hat{\eta}_1^\perp\otimes\hat{\eta}_2^\perp\otimes \hat{\eta}_2, f^{yy}_{yy}\rangle=0,
\]
we now write
\[
{\scriptsize
\left(\begin{array}{ccccc}
0&0&\frac{|\eta|^4}{\xi^4}+2\frac{|\eta|^2}{\xi^2}+1&\frac{|\eta|^2}{\xi^2}+1&1\\
\frac{|\eta|^4}{\xi^4}+2\frac{|\eta|^2}{\xi^2}+1&-2(\frac{|\eta|^2}{\xi^2}+1)&0&2(\frac{|\eta|^4}{\xi^4}+\frac{|\eta|^2}{\xi^2})&\frac{|\eta|^2}{\xi^2}\\
0&0&-(\frac{|\eta|^5}{\xi^5}+2\frac{|\eta|^3}{\xi^3}+\frac{|\eta|}{\xi})&4(\frac{|\eta|^3}{\xi^3}+\frac{|\eta|}{\xi})&-\frac{|\eta|}{\xi}\\
1 & 0 & 1&1&1\\
0&1&0&1&1
\end{array}\right)
\left(\begin{array}{c}
\langle\hat{\eta}_1\otimes\hat{\eta}_2,f^{xy}_{xy}\rangle\\
\langle\hat{\eta}^\perp_1\otimes\hat{\eta}^\perp_2,f^{xy}_{xy}\rangle\\
\langle\hat{\eta}_1\otimes\hat{\eta}_1\otimes\hat{\eta}_2\otimes\hat{\eta}_2,f^{yy}_{yy}\rangle\\
\langle\hat{\eta}_1\otimes\hat{\eta}_1^\perp\otimes \hat{\eta}_2\otimes\hat{\eta}_2^\perp, f^{yy}_{yy}\rangle\\
\langle\hat{\eta}_1^\perp\otimes\hat{\eta}_1^\perp\otimes\hat{\eta}_2^\perp\otimes\hat{\eta}_2^\perp,f^{yy}_{yy}\rangle
\end{array}\right)=0.
}
\]
Note that the above linear system is nonsingular, so
\[
\begin{split}
\langle\hat{\eta}_1\otimes\hat{\eta}_2,f^{xy}_{xy}\rangle=
\langle\hat{\eta}_1^\perp\otimes\hat{\eta}_2^\perp,f^{xy}_{xy}\rangle=&
\langle\hat{\eta}_1\otimes\hat{\eta}_1\otimes\hat{\eta}_2\otimes\hat{\eta}_2,f^{yy}_{yy}\rangle
=\langle\hat{\eta}_1\otimes\hat{\eta}_1^\perp\otimes \hat{\eta}_2\otimes\hat{\eta}_2^\perp, f^{yy}_{yy}\rangle\\
&\quad\quad\quad=\langle\hat{\eta}_1^\perp\otimes\hat{\eta}_1^\perp\otimes\hat{\eta}_2^\perp\otimes\hat{\eta}_2^\perp,f^{yy}_{yy}\rangle=0.
\end{split}
\]
Using the gauge conditions, we also have
\[
\begin{split}
\langle\hat{\eta}_1,f^{xy}_{xx}\rangle=\langle\hat{\eta}_1\otimes\hat{\eta}_2\otimes\hat{\eta}_2,f^{xy}_{yy}\rangle=\langle\hat{\eta}_1^\perp\otimes\hat{\eta}_2^\perp\otimes\hat{\eta}_2,f^{xy}_{yy}\rangle=\langle\hat{\eta}_1\otimes\hat{\eta}_1\otimes\hat{\eta}_2,f^{yy}_{xy}\rangle\\
=\langle\hat{\eta}_1\otimes\hat{\eta}_1,f^{yy}_{xx}\rangle=\langle\hat{\eta}_1\otimes\hat{\eta}_1^\perp\otimes\hat{\eta}_2^\perp, f^{yy}_{xy}\rangle=0.
\end{split}
\]

\text{Step 3.} Now we consider the following identity:
\begin{equation}\label{step3_eq}
\begin{split}
\tilde{S}^2f^{xx}_{xx}+2\tilde{S}\langle\hat{Y}_2,f^{xx}_{xy}\rangle+\langle\hat{Y}_2\otimes\hat{Y}_2,f^{xx}_{yy}\rangle-2\tilde{S}^3\langle\hat{Y}_1,f^{xy}_{xx}\rangle-4\tilde{S}^2\langle\hat{Y}_1\otimes\hat{Y}_2,f^{xy}_{xy}\rangle\\
-2\tilde{S}\langle\hat{Y}_1\otimes\hat{Y}_2\otimes\hat{Y}_2,f^{xy}_{yy}\rangle
+\tilde{S}^4\langle\hat{Y}_1\otimes\hat{Y}^1,f^{yy}_{xx}\rangle+2\tilde{S}^3\langle\hat{Y}_1\otimes\hat{Y}_1\otimes\hat{Y}_2,f^{yy}_{xy}\rangle\\
+\tilde{S}^2\langle\hat{Y}_1\otimes\hat{Y}_1\otimes\hat{Y}_2\otimes\hat{Y}_2,f^{yy}_{yy}\rangle=0.
\end{split}
\end{equation}
Taking $\hat{Y}=\epsilon\hat{\eta}+\sqrt{1-\epsilon^2}\hat{\eta}^\perp$ and $\tilde{S}=-\frac{\epsilon|\eta|}{\xi}$ as above,
we get

\begin{equation}\label{long_identity}
{\tiny
\begin{split}
\frac{\epsilon^2|\eta|^2}{\xi^2}f^{xx}_{xx}-2\frac{\epsilon^2|\eta|}{\xi}\langle\hat{\eta}_2,f^{xx}_{xy}\rangle-2\frac{\epsilon\sqrt{1-\epsilon^2}|\eta|}{\xi}\langle\hat{\eta}_2^\perp,f^{xx}_{xy}\rangle
+\epsilon^2\langle\hat{\eta}_2\otimes\hat{\eta}_2,f^{xx}_{yy}\rangle+2\epsilon\sqrt{1-\epsilon^2}\langle\hat{\eta}_2^\perp\otimes\hat{\eta}_2,f^{xx}_{yy}\rangle+(1-\epsilon^2)\langle\hat{\eta}_2^\perp\otimes\hat{\eta}_2^\perp,f^{xx}_{yy}\rangle\\
+2\frac{\epsilon^4|\eta|^3}{\xi^3}\langle\hat{\eta}_1,f^{xy}_{xx}\rangle+2\frac{\epsilon^3\sqrt{1-\epsilon^2}|\eta|^3}{\xi^3}\langle\hat{\eta}_1^\perp,f^{xy}_{xx}\rangle
-4\frac{\epsilon^4|\eta|^2}{\xi^2}\langle\hat{\eta}_1\otimes\hat{\eta}_2,f^{xy}_{xy}\rangle-4\frac{\epsilon^3\sqrt{1-\epsilon^2}|\eta|^2}{\xi^2}\langle\hat{\eta}_1^\perp\otimes\hat{\eta}_2,f^{xy}_{xy}\rangle-4\frac{\epsilon^3\sqrt{1-\epsilon^2}|\eta|^2}{\xi^2}\langle\hat{\eta}_1\otimes\hat{\eta}_2^\perp,f^{xy}_{xy}\rangle\\
-4\frac{\epsilon^2(1-\epsilon^2)|\eta|^2}{\xi^2}\langle\hat{\eta}_1^\perp\otimes\hat{\eta}_2^\perp,f^{xy}_{xy}\rangle+2\frac{\epsilon^4|\eta|}{\xi}\langle\hat{\eta}_1\otimes\hat{\eta}_2\otimes\hat{\eta}_2,f^{xy}_{yy}\rangle+2\frac{\epsilon^3\sqrt{1-\epsilon^2}|\eta|}{\xi}\langle\hat{\eta}_1^\perp\otimes\hat{\eta}_2\otimes\hat{\eta}_2,f^{xy}_{yy}\rangle
+4\frac{\epsilon^3\sqrt{1-\epsilon^2}|\eta|}{\xi}\langle\hat{\eta}_1\otimes\hat{\eta}_2^\perp\otimes\hat{\eta}_2,f^{xy}_{yy}\rangle\\
+4\frac{\epsilon^2(1-\epsilon^2)|\eta|}{\xi}\langle\hat{\eta}_1^\perp\otimes\hat{\eta}_2^\perp\otimes\hat{\eta}_2,f^{xy}_{yy}\rangle+2\frac{\epsilon^2(1-\epsilon^2)|\eta|}{\xi}\langle\hat{\eta}_1\otimes\hat{\eta}^\perp_2\otimes\hat{\eta}_2^\perp,f^{xy}_{yy}\rangle+2\frac{\epsilon\sqrt{(1-\epsilon^2)^3}|\eta|}{\xi}\langle\hat{\eta}_1^\perp\otimes\hat{\eta}_2^\perp\otimes\hat{\eta}_2^\perp,f^{xy}_{yy}\rangle\\
+\frac{\epsilon^6|\eta|^4}{\xi^4}\langle\hat{\eta}_1\otimes\hat{\eta}_1,f^{yy}_{xx}\rangle+2\frac{\epsilon^5\sqrt{1-\epsilon^2}|\eta|^4}{\xi^4}\langle\hat{\eta}_1^\perp\otimes\hat{\eta}_1,f^{yy}_{xx}\rangle+\frac{\epsilon^4(1-\epsilon^2)|\eta|^4}{\xi^4}\langle\hat{\eta}_2^\perp\otimes\hat{\eta}_2^\perp,f_{xx}^{yy}\rangle\\
-2\frac{\epsilon^6|\eta|^3}{\xi^3}\langle\hat{\eta}_1\otimes\hat{\eta}_2\otimes\hat{\eta}_2,f_{xy}^{yy}\rangle-2\frac{\epsilon^5\sqrt{1-\epsilon^2}|\eta|^3}{\xi^3}\langle\hat{\eta}_1\otimes\hat{\eta}_2\otimes\hat{\eta}_2^\perp,f_{xy}^{yy}\rangle
-4\frac{\epsilon^5\sqrt{1-\epsilon^2}|\eta|^3}{\xi^3}\langle\hat{\eta}_1\otimes\hat{\eta}_2^\perp\otimes\hat{\eta}_2,f_{xy}^{yy}\rangle-4\frac{\epsilon^4(1-\epsilon^2)|\eta|^3}{\xi^3}\langle\hat{\eta}_1\otimes\hat{\eta}_2^\perp\otimes\hat{\eta}_2^\perp,f_{xy}^{yy}\rangle\\
-2\frac{\epsilon^4(1-\epsilon^2)|\eta|^3}{\xi^3}\langle\hat{\eta}_1^\perp\otimes\hat{\eta}_2^\perp\otimes\hat{\eta}_2,f_{xy}^{yy}\rangle-2\frac{\epsilon^3\sqrt{(1-\epsilon^2)^3}|\eta|^3}{\xi^3}\langle\hat{\eta}_1^\perp\otimes\hat{\eta}_2^\perp\otimes\hat{\eta}_2^\perp,f_{xy}^{yy}\rangle\\
+\frac{\epsilon^6|\eta|^2}{\xi^2}\langle\hat{\eta}_1\otimes\hat{\eta}_1\otimes\hat{\eta}_2\otimes\hat{\eta}_2,f^{yy}_{yy}\rangle+2\frac{\epsilon^5\sqrt{1-\epsilon^2}|\eta|^2}{\xi^2}\langle\hat{\eta}_1^\perp\otimes\hat{\eta}_1\otimes\hat{\eta}_2\otimes\hat{\eta}_2,f^{yy}_{yy}\rangle
+2\frac{\epsilon^5\sqrt{1-\epsilon^2}|\eta|^2}{\xi^2}\langle\hat{\eta}_1\otimes\hat{\eta}_1\otimes\hat{\eta}_2^\perp\otimes\hat{\eta}_2,f^{yy}_{yy}\rangle\\
+\frac{\epsilon^4(1-\epsilon^2)|\eta|^2}{\xi^2}\langle\hat{\eta}_1^\perp\otimes\hat{\eta}_1^\perp\otimes\hat{\eta}_2\otimes\hat{\eta}_2,f^{yy}_{yy}\rangle+\frac{\epsilon^4(1-\epsilon^2)|\eta|^2}{\xi^2}\langle\hat{\eta}_1\otimes\hat{\eta}_1\otimes\hat{\eta}_2^\perp\otimes\hat{\eta}_2^\perp,f^{yy}_{yy}\rangle+4\frac{\epsilon^4(1-\epsilon^2)|\eta|^2}{\xi^2}\langle\hat{\eta}_1^\perp\otimes\hat{\eta}_1\otimes\hat{\eta}_2^\perp\otimes\hat{\eta}_2,f^{yy}_{yy}\rangle\\
+2\frac{\epsilon^3\sqrt{(1-\epsilon^2)^3}|\eta|^2}{\xi^2}\langle\hat{\eta}_1^\perp\otimes\hat{\eta}_1^\perp\otimes\hat{\eta}_2^\perp\otimes\hat{\eta}_2,f^{yy}_{yy}\rangle+2\frac{\epsilon^3\sqrt{(1-\epsilon^2)^3}|\eta|^2}{\xi^2}\langle\hat{\eta}_1^\perp\otimes\hat{\eta}_1\otimes\hat{\eta}_2^\perp\otimes\hat{\eta}_2^\perp,f^{yy}_{yy}\rangle\\
+\frac{\epsilon^2(1-\epsilon^2)^2|\eta|^2}{\xi^2}\langle\hat{\eta}_1^\perp\otimes\hat{\eta}_1^\perp\otimes\hat{\eta}_2^\perp\otimes\hat{\eta}_2^\perp,f^{yy}_{yy}\rangle=0.
\end{split}
}
\end{equation}

Eliminating all the terms proved to be vanishing in Step 1 and Step 2, we have
\[
\begin{split}
\frac{\epsilon^2|\eta|^2}{\xi^2}f^{xx}_{xx}-2\frac{\epsilon^2|\eta|}{\xi}\langle\hat{\eta}_1,f^{xx}_{xy}\rangle-2\frac{\epsilon\sqrt{1-\epsilon^2}|\eta|}{\xi}\langle\hat{\eta}_1^\perp,f^{xx}_{xy}\rangle
+\epsilon^2\langle\hat{\eta}_2\otimes\hat{\eta}_2,f^{xx}_{yy}\rangle+2\epsilon\sqrt{1-\epsilon^2}\langle\hat{\eta}_2^\perp\otimes\hat{\eta}_2,f^{xx}_{yy}\rangle=0.
\end{split}
\]
Using the identity $\langle\hat{\eta}_1\otimes\hat{\eta}_2,f^{xy}_{xy}\rangle=\langle\hat{\eta}_1^\perp\otimes\hat{\eta}_2^\perp,f^{xy}_{xy}\rangle=0$ and the trace-free condition
\[
f^{xx}_{xx}+\langle\hat{\eta}_1\otimes\hat{\eta}_2,f^{xy}_{xy}\rangle+\langle\hat{\eta}_1^\perp\otimes\hat{\eta}_2^\perp,f^{xy}_{xy}\rangle=0,
\]
we have $f^{xx}_{xx}=0$, and then $\langle\hat{\eta}_2,f^{xx}_{xy}\rangle=\langle\hat{\eta}_2\otimes\hat{\eta}_2,f^{xx}_{yy}\rangle=0$. So we have
\[
-2\frac{|\eta|}{\xi}\langle\hat{\eta}_2^\perp,f^{xx}_{xy}\rangle+2\langle\hat{\eta}_2^\perp\otimes\hat{\eta}_2,f^{xx}_{yy}\rangle=0.
\]
Using the gauge condition $f^{xx}_{xy}=-\frac{|\eta|}{\xi}\langle\hat{\eta}_2,f^{xx}_{yy}\rangle$, we have
\[
(\frac{|\eta|^2}{\xi^2}+1)\langle\hat{\eta}^\perp_2\otimes\hat{\eta}_2,f^{xx}_{yy}\rangle=0
\]
and consequently
\[
\langle\hat{\eta}^\perp_2,f^{xx}_{xy}\rangle=\langle\hat{\eta}_2^\perp\otimes\hat{\eta}_2,f^{xx}_{yy}\rangle=0.
\]
~\\
\textbf{Case 2: $\xi=0$} (so $\eta\neq 0$).
Now the gauge condition becomes
\begin{equation}\label{degnerategauge}
\langle\hat{\eta}_2,f^{xx}_{xy}\rangle=
\langle\hat{\eta}_2,f^{xy}_{xy}\rangle=
\langle\hat{\eta}_2,f^{yy}_{xy}\rangle=
\langle\hat{\eta}_2,f^{xx}_{yy}\rangle=
\langle\hat{\eta}_2,f^{xy}_{yy}\rangle=
\langle\hat{\eta}_2,f^{yy}_{yy}\rangle=0.
\end{equation}
In this case, we can take $\tilde{S}$ to be an arbitrary number sufficiently close to $0$, and $\hat{Y}=\hat{\eta}^\perp$. Taking first and second derivatives of \eqref{step1eq} in $\tilde{S}$ at $\tilde{S}=0$, we obtain
\[
\langle \hat{\eta}_1\otimes \hat{\eta}_1,f^{yy}_{xx}\rangle=0,\quad \langle\hat{\eta}_1\otimes \hat{\eta}_1\otimes\hat{\eta}^\perp_2,f^{yy}_{xy}\rangle=0.
\]

Taking $3$rd order derivatives of \eqref{step2_eq} in $\tilde{S}$ at $\tilde{S}=0$, we have
\[
\langle \hat{\eta}_1\otimes \hat{\eta}_1^\perp, f^{yy}_{xx}\rangle=0.
\]
 Taking $1$st, $2$nd derivatives, we have
 \[
2\langle\hat{\eta}_1\otimes\hat{\eta}_2^\perp, f^{xy}_{xy}\rangle-\langle\hat{\eta}_1\otimes \hat{\eta}_1^\perp\otimes \hat{\eta}_2^\perp\otimes \hat{\eta}_2^\perp, f^{yy}_{yy}\rangle=0,
\]
\[
\langle\hat{\eta}_1,f^{xy}_{xx}\rangle-2\langle\hat{\eta}_1\otimes\hat{\eta}_1^\perp\otimes \hat{\eta}_2^\perp, f^{yy}_{xy}\rangle=0.
\]
Combined with the trace-free conditions
\[
\langle\hat{\eta}_1\otimes\hat{\eta}_2^\perp, f^{xy}_{xy}\rangle+\langle\hat{\eta}_1\otimes \hat{\eta}_1^\perp\otimes \hat{\eta}_2^\perp\otimes \hat{\eta}_2^\perp, f^{yy}_{yy}\rangle+\langle\hat{\eta}_1\otimes \hat{\eta}_1\otimes \hat{\eta}_2\otimes \hat{\eta}_2^\perp, f^{yy}_{yy}\rangle=0,
\]
\[
\langle\hat{\eta}_1,f^{xy}_{xx}\rangle+\langle\hat{\eta}_1\otimes\hat{\eta}_1^\perp\otimes \hat{\eta}_2^\perp, f^{yy}_{xy}\rangle+\langle\hat{\eta}_1\otimes\hat{\eta}_1\otimes \hat{\eta}_2, f^{yy}_{xy}\rangle=0
\]
and the gauge conditions \eqref{degnerategauge}, we conclude that
\[
\langle\hat{\eta}_1\otimes\hat{\eta}_2^\perp, f^{xy}_{xy}\rangle=\langle\hat{\eta}_1\otimes \hat{\eta}_1^\perp\otimes \hat{\eta}_2^\perp\otimes \hat{\eta}_2^\perp, f^{yy}_{yy}\rangle=\langle\hat{\eta}_1,f^{xy}_{xx}\rangle=\langle\hat{\eta}_1\otimes\hat{\eta}_1^\perp\otimes \hat{\eta}_2^\perp, f^{yy}_{xy}\rangle=0.
\]

Taking $4$th derivatives of \eqref{step3_eq} in $\tilde{S}$ at $\tilde{S}=0$, we obtain
\[
\langle\hat{\eta}_1^\perp\otimes \hat{\eta}_1^\perp,f^{yy}_{xx}\rangle=0.
\]
Taking $1$st, $3$rd derivatives of \eqref{step3_eq} in $\tilde{S}$ at $\tilde{S}=0$, we obtain
\[
\langle \hat{\eta}_2^\perp,f^{xx}_{xy}\rangle-\langle \hat{\eta}_1^\perp\otimes \hat{\eta}_2^\perp\otimes\hat{\eta}_2^\perp, f^{xy}_{yy}\rangle=0,
\]
\[
\langle \hat{\eta}_1^\perp,f^{xy}_{xx}\rangle-\langle \hat{\eta}_1^\perp\otimes \hat{\eta}_1^\perp\otimes \hat{\eta}_2^\perp,f^{yy}_{xy}\rangle=0.
\]
Together with the trace-free conditions
\[
\langle \hat{\eta}_2^\perp,f^{xx}_{xy}\rangle+\langle \hat{\eta}_1^\perp\otimes \hat{\eta}_2^\perp\otimes\hat{\eta}_2^\perp, f^{xy}_{yy}\rangle+\langle \hat{\eta}_1\otimes \hat{\eta}_2\otimes\hat{\eta}_2^\perp, f^{xy}_{yy}\rangle=0,
\]
\[
\langle \hat{\eta}_1^\perp,f^{xy}_{xx}\rangle+\langle \hat{\eta}_1^\perp\otimes \hat{\eta}_1^\perp\otimes \hat{\eta}_2^\perp,f^{yy}_{xy}\rangle+\langle \hat{\eta}_1^\perp\otimes \hat{\eta}_1\otimes \hat{\eta}_2,f^{yy}_{xy}\rangle=0,
\]
and the gauge conditions \eqref{degnerategauge}, we conclude that
\[
\langle \hat{\eta}_2^\perp,f^{xx}_{xy}\rangle=\langle \hat{\eta}_1^\perp\otimes \hat{\eta}_2^\perp\otimes\hat{\eta}_2^\perp, f^{xy}_{yy}\rangle=\langle \hat{\eta}_1^\perp,f^{xy}_{xx}\rangle=\langle \hat{\eta}_1^\perp\otimes \hat{\eta}_1^\perp\otimes \hat{\eta}_2^\perp,f^{yy}_{xy}\rangle=0.
\]
Taking $2$nd derivatives of \eqref{step3_eq} in $\tilde{S}$ at $\tilde{S}=0$, we obtain
\[
f^{xx}_{xx}-4\langle \hat{\eta}_1^\perp\otimes \hat{\eta}_2^\perp,f^{xy}_{xy}\rangle+\langle \hat{\eta}_1^\perp\otimes \hat{\eta}_1^\perp\otimes \hat{\eta}_2^\perp\otimes\hat{\eta}_2^\perp,f^{yy}_{yy}\rangle=0.
\]
Considering also the gauge conditions
\[
f^{xx}_{xx}+\langle \hat{\eta}_1^\perp\otimes \hat{\eta}_2^\perp,f^{xy}_{xy}\rangle+\langle \hat{\eta}_1\otimes \hat{\eta}_2,f^{xy}_{xy}\rangle=0
\]
and
\[
\langle \hat{\eta}_1^\perp\otimes \hat{\eta}_2^\perp,f^{xy}_{xy}\rangle+\langle \hat{\eta}_1^\perp\otimes \hat{\eta}_1^\perp\otimes \hat{\eta}_2^\perp\otimes\hat{\eta}_2^\perp,f^{yy}_{yy}\rangle+\langle \hat{\eta}_1^\perp\otimes \hat{\eta}_1\otimes \hat{\eta}_2\otimes\hat{\eta}_2^\perp,f^{yy}_{yy}\rangle=0,
\]
and the gauge conditions \eqref{degnerategauge}, we have
\[
f^{xx}_{xx}=\langle \hat{\eta}_1^\perp\otimes \hat{\eta}_2^\perp,f^{xy}_{xy}\rangle=\langle \hat{\eta}_1^\perp\otimes \hat{\eta}_1^\perp\otimes \hat{\eta}_2^\perp\otimes\hat{\eta}_2^\perp,f^{yy}_{yy}\rangle=0.
\]

The above argument shows that if $f$ is in the kernel of the principal symbol of $\delta_\mathsf{F}$ and in the kernel of the matrix \eqref{matrix22_fiberinfinity} for any $(\tilde{S},\hat{Y})$ on the equatorial sphere $\xi\tilde{S}+\eta\cdot\hat{Y}$, we must have $f\equiv 0$. This proves the ellipticity at $x=0$ of $N_\mathsf{F}$ at fiber infinity restricted to the kernel of the principal symbol of $\delta^\mathcal{B}_\mathsf{F}$.
\end{proof}

\begin{lemma}
Assume $n=3$. The operator $N_\mathsf{F}\in \Psi_{sc}^{-1,0}(X;\mathcal{B}S^2_2{^{sc}T}X,\mathcal{B}S^2_2{^{sc}T}X)$ is elliptic at base infinity of $^{sc}T^*X$ when restricted to the kernel of the principal symbol of $\delta^\mathcal{B}_\mathsf{F}$.
\end{lemma}
\begin{proof}
The (semiclassical) principal symbol of $N_\mathsf{F}$ at the scattering front face is
\[
\int_{\mathbb{S}^{n-2}} (\xi_\mathsf{F}^2+1)^{-1/2}\mathcal{M}_{22}e^{-(\hat{Y}\cdot\eta_\mathsf{F})^2/(2h\phi_\mathsf{F}(\xi_\mathsf{F},\hat{Y}))}\mathrm{d}\hat{Y},
\]
where
 \[
{\tiny
\begin{split}
\mathcal{M}_{22}=& \left(\begin{array}{ccc}
(\xi_\mathsf{F}+\mathrm{i})^2\left(\frac{\hat{Y}\cdot\eta_\mathsf{F}}{\xi_\mathsf{F}^2+1}\right)^2 &&\\
-(\xi_\mathsf{F}+\mathrm{i})\left(\frac{\hat{Y}\cdot\eta_\mathsf{F}}{\xi_\mathsf{F}^2+1}\right)\hat{Y}_2 &&\\
\hat{Y}_2\otimes\hat{Y}_2&&\\
&(\xi_\mathsf{F}+\mathrm{i})^2\left(\frac{\hat{Y}\cdot\eta_\mathsf{F}}{\xi_\mathsf{F}^2+1}\right)^2 &\\
&-(\xi_\mathsf{F}+\mathrm{i})\left(\frac{\hat{Y}\cdot\eta_\mathsf{F}}{\xi_\mathsf{F}^2+1}\right)\hat{Y}_2 &\\
&\hat{Y}_2\otimes\hat{Y}_2&\\
&&(\xi_\mathsf{F}+\mathrm{i})^2\left(\frac{\hat{Y}\cdot\eta_\mathsf{F}}{\xi_\mathsf{F}^2+1}\right)^2 \\
&&-(\xi_\mathsf{F}+\mathrm{i})\left(\frac{\hat{Y}\cdot\eta_\mathsf{F}}{\xi_\mathsf{F}^2+1}\right)\hat{Y}_2 \\
&&\hat{Y}_2\otimes\hat{Y}_2
 \end{array}
 \right)\\
&\left(\begin{array}{ccc}
1 & 0 &\\
2(\xi_\mathsf{F}+\mathrm{i})\frac{\hat{Y}\cdot\eta_\mathsf{F}}{\xi_\mathsf{F}^2+1}\hat{Y}_1&I-\hat{Y}_1\langle\hat{Y}_1,\,\cdot\,\rangle& 0\\
(\xi_\mathsf{F}+\mathrm{i})^2\left(\frac{\hat{Y}\cdot\eta_\mathsf{F}}{\xi_\mathsf{F}^2+1}\right)^2\hat{Y}_1\otimes\hat{Y}_1&(\xi_\mathsf{F}+\mathrm{i})\frac{\hat{Y}\cdot\eta_\mathsf{F}}{\xi_\mathsf{F}^2+1}\left(\hat{Y}_1\otimes^s-\hat{Y}_1\otimes\hat{Y}_1\langle\hat{Y}_1,\,\cdot\,\rangle\right)& I-2\hat{Y}_1\otimes^s\langle\hat{Y}_1,\,\cdot\,\rangle+\hat{Y}_1\otimes\hat{Y}_1\langle\hat{Y}_1\otimes\hat{Y}_1,\,\cdot\,\rangle
\end{array}
\right)\\
&\left(\begin{array}{ccc}
1 & 2(\xi_\mathsf{F}-\mathrm{i})\frac{\hat{Y}_1\cdot\eta_\mathsf{F}}{\xi_\mathsf{F}^2+1}\langle\hat{Y}_1,\,\cdot\,\rangle &(\xi_\mathsf{F}-\mathrm{i})^2\left(\frac{\hat{Y}\cdot\eta_\mathsf{F}}{\xi_\mathsf{F}^2+1}\right)^2\langle\hat{Y}_1\otimes\hat{Y}_1,\,\cdot\,\rangle\\
0&I-\hat{Y}_1\langle\hat{Y}_1,\,\cdot\,\rangle& (\xi_\mathsf{F}-\mathrm{i})\frac{\hat{Y}\cdot\eta_\mathsf{F}}{\xi_\mathsf{F}^2+1}\left(\langle\hat{Y}_1,\,\cdot\,\rangle-\hat{Y}_1\langle\hat{Y}_1\otimes\hat{Y}_1,\,\cdot\,\rangle\right)\\
0&0& I-2\hat{Y}_1\otimes^s\langle\hat{Y}_1,\,\cdot\,\rangle+\hat{Y}_1\otimes\hat{Y}_1\langle\hat{Y}_1\otimes\hat{Y}_1,\,\cdot\,\rangle
\end{array}
\right)\\
& \left(\begin{array}{ccc}
(\xi_\mathsf{F}-\mathrm{i})^2\left(\frac{\hat{Y}\cdot\eta_\mathsf{F}}{\xi_\mathsf{F}^2+1}\right)^2 &&\\
-2(\xi_\mathsf{F}-\mathrm{i})\left(\frac{\hat{Y}\cdot\eta_\mathsf{F}}{\xi_\mathsf{F}^2+1}\right)\langle\hat{Y}_2,\cdot\rangle &&\\
\langle\hat{Y}_2\otimes\hat{Y}_2,\cdot\rangle&&\\
&(\xi_\mathsf{F}-\mathrm{i})^2\left(\frac{\hat{Y}\cdot\eta_\mathsf{F}}{\xi_\mathsf{F}^2+1}\right)^2 &\\
&-2(\xi_\mathsf{F}-\mathrm{i})\left(\frac{\hat{Y}\cdot\eta_\mathsf{F}}{\xi_\mathsf{F}^2+1}\right)\langle\hat{Y}_2,\cdot\rangle &\\
&\langle\hat{Y}_2\otimes\hat{Y}_2,\cdot\rangle&\\
&&(\xi_\mathsf{F}-\mathrm{i})^2\left(\frac{\hat{Y}\cdot\eta_\mathsf{F}}{\xi_\mathsf{F}^2+1}\right)^2 \\
&&-2(\xi_\mathsf{F}-\mathrm{i})\left(\frac{\hat{Y}\cdot\eta_\mathsf{F}}{\xi_\mathsf{F}^2+1}\right)\langle\hat{Y}_2,\cdot\rangle \\
&&\langle\hat{Y}_2\otimes\hat{Y}_2,\cdot\rangle
 \end{array}
 \right)^T.
\end{split}
}
\]
We restrict $f$ to be in the kernel of the principal symbol of $\delta^\mathcal{B}_\mathsf{F}$. So
\begin{equation}\label{deltakernel22_ff}
\begin{split}
(\xi_\mathsf{F}-\mathrm{i}) f^{xx}_{xx}+\langle\eta_\mathsf{F},f^{xx}_{xy}\rangle_2=0,\\
(\xi_\mathsf{F}-\mathrm{i}) f^{xy}_{xx}+\langle\eta_\mathsf{F},f^{xy}_{xy}\rangle_2=0,\\
(\xi_\mathsf{F}-\mathrm{i}) f^{yy}_{xx}+\langle\eta_\mathsf{F},f^{yy}_{xy}\rangle_2=0,\\
(\xi_\mathsf{F}-\mathrm{i})f^{xx}_{xy}+\langle\eta_\mathsf{F},f^{xx}_{yy}\rangle_2=0,\\
(\xi_\mathsf{F}-\mathrm{i}) f^{xy}_{xy}+\langle\eta_\mathsf{F},f^{xy}_{yy}\rangle_2=0,\\
(\xi_\mathsf{F}-\mathrm{i})f^{yy}_{xy}+\langle\eta_\mathsf{F},f^{yy}_{yy}\rangle_2=0.
\end{split}
\end{equation}
Take $\hat{Y}=\hat{\eta}^\perp_\mathsf{F}$ in above, we obtain
\begin{eqnarray}
\langle\hat{\eta}_\mathsf{F}^\perp\otimes\hat{\eta}_\mathsf{F}^\perp,f^{xx}_{yy}\rangle=0,\label{S0_identity_1_F}\\
\langle\eta_\mathsf{F}\otimes\hat{\eta}_\mathsf{F}^\perp\otimes \hat{\eta}_\mathsf{F}^\perp,f^{xy}_{yy}\rangle=0,\label{S0_identity_2_F}\\
\langle\eta_\mathsf{F}\otimes\eta_\mathsf{F}\otimes\hat{\eta}_\mathsf{F}^\perp\otimes \hat{\eta}_\mathsf{F}^\perp,f^{yy}_{yy}\rangle=0.\label{S0_identity_3_F}
\end{eqnarray}

If $f$ is in the kernel of $\mathcal{M}_{22}$, we have
\[
\begin{split}
&\left(\begin{array}{ccc}
1 & 2(\xi_\mathsf{F}-\mathrm{i})\frac{\hat{Y}\cdot\eta_\mathsf{F}}{\xi_\mathsf{F}^2+1}\langle\hat{Y}_1,\,\cdot\,\rangle &(\xi_\mathsf{F}-\mathrm{i})^2\left(\frac{\hat{Y}\cdot\eta_\mathsf{F}}{\xi_\mathsf{F}^2+1}\right)^2\langle\hat{Y}_1\otimes\hat{Y}_1,\,\cdot\,\rangle\\
0&I-\hat{Y}_1\langle\hat{Y}_1,\,\cdot\,\rangle& (\xi_\mathsf{F}-\mathrm{i})\frac{\hat{Y}\cdot\eta_\mathsf{F}}{\xi_\mathsf{F}^2+1}\left(\langle\hat{Y}_1,\,\cdot\,\rangle-\hat{Y}_1\langle\hat{Y}_1\otimes\hat{Y}_1,\,\cdot\,\rangle\right)\\
0&0& I-2\hat{Y}_1\otimes^s\langle\hat{Y}_1,\,\cdot\,\rangle+\hat{Y}_1\otimes\hat{Y}_1\langle\hat{Y}_1\otimes\hat{Y}_1,\,\cdot\,\rangle
\end{array}
\right)\\
&\left(\begin{array}{c}
(\xi_\mathsf{F}-\mathrm{i})^2\left(\frac{\hat{Y}\cdot\eta_\mathsf{F}}{\xi_\mathsf{F}^2+1}\right)^2f^{xx}_{xx}-2(\xi_\mathsf{F}-\mathrm{i})\left(\frac{\hat{Y}\cdot\eta_\mathsf{F}}{\xi_\mathsf{F}^2+1}\right)\langle\hat{Y}_2,f^{xx}_{xy}\rangle+\langle\hat{Y}_2\otimes\hat{Y}_2,f^{xx}_{yy}\rangle\\
(\xi_\mathsf{F}-\mathrm{i})^2\left(\frac{\hat{Y}\cdot\eta_\mathsf{F}}{\xi_\mathsf{F}^2+1}\right)^2f^{xy}_{xx}-2(\xi_\mathsf{F}-\mathrm{i})\left(\frac{\hat{Y}\cdot\eta_\mathsf{F}}{\xi_\mathsf{F}^2+1}\right)\langle\hat{Y}_2,f^{xy}_{xy}\rangle+\langle\hat{Y}_2\otimes\hat{Y}_2,f^{xy}_{yy}\rangle\\
(\xi_\mathsf{F}-\mathrm{i})^2\left(\frac{\hat{Y}\cdot\eta_\mathsf{F}}{\xi_\mathsf{F}^2+1}\right)^2f^{yy}_{xx}-2(\xi_\mathsf{F}-\mathrm{i})\left(\frac{\hat{Y}\cdot\eta_\mathsf{F}}{\xi_\mathsf{F}^2+1}\right)\langle\hat{Y}_2,f^{yy}_{xy}\rangle+\langle\hat{Y}_2\otimes\hat{Y}_2,f^{yy}_{yy}\rangle
\end{array}
\right)=0.
\end{split}
\]

Applying
\[
\left(\begin{array}{ccc}
1 & &\\
&\langle\hat{Y}^\perp,\,\cdot\,\rangle&\\
&&\langle \hat{Y}^\perp\otimes\hat{Y}^\perp,\,\cdot\,\rangle
\end{array}
\right)
\]
from the left side of the above matrix, we get
\[
\begin{split}
&\left(\begin{array}{ccc}
1 & 2(\xi_\mathsf{F}-\mathrm{i})\frac{\hat{Y}\cdot\eta_\mathsf{F}}{\xi_\mathsf{F}^2+1}\langle\hat{Y},\,\cdot\,\rangle &(\xi_\mathsf{F}-\mathrm{i})^2\left(\frac{\hat{Y}\cdot\eta_\mathsf{F}}{\xi_\mathsf{F}^2+1}\right)^2\langle\hat{Y}\otimes\hat{Y},\,\cdot\,\rangle\\
0&\langle\hat{Y}^\perp,\,\cdot\,\rangle& (\xi_\mathsf{F}-\mathrm{i})\frac{\hat{Y}\cdot\eta_\mathsf{F}}{\xi_\mathsf{F}^2+1}\langle\hat{Y}^\perp\otimes\hat{Y},\,\cdot\,\rangle\\
0&0& \langle\hat{Y}^\perp\otimes\hat{Y}^\perp\otimes\hat{Y}\otimes\hat{Y},\,\cdot\,\rangle
\end{array}
\right)\\
&\left(\begin{array}{c}
(\xi_\mathsf{F}-\mathrm{i})^2\left(\frac{\hat{Y}\cdot\eta_\mathsf{F}}{\xi_\mathsf{F}^2+1}\right)^2f^{xx}_{xx}-2(\xi_\mathsf{F}-\mathrm{i})\left(\frac{\hat{Y}\cdot\eta_\mathsf{F}}{\xi_\mathsf{F}^2+1}\right)\langle\hat{Y},f^{xx}_{xy}\rangle+\langle\hat{Y}\otimes\hat{Y},f^{xx}_{yy}\rangle\\
(\xi_\mathsf{F}-\mathrm{i})^2\left(\frac{\hat{Y}\cdot\eta_\mathsf{F}}{\xi_\mathsf{F}^2+1}\right)^2f^{xy}_{xx}-2(\xi_\mathsf{F}-\mathrm{i})\left(\frac{\hat{Y}\cdot\eta_\mathsf{F}}{\xi_\mathsf{F}^2+1}\right)\langle\hat{Y},f^{xy}_{xy}\rangle+\langle\hat{Y}\otimes\hat{Y},f^{xy}_{yy}\rangle\\
(\xi_\mathsf{F}-\mathrm{i})^2\left(\frac{\hat{Y}\cdot\eta_\mathsf{F}}{\xi_\mathsf{F}^2+1}\right)^2f^{yy}_{xx}-2(\xi_\mathsf{F}-\mathrm{i})\left(\frac{\hat{Y}\cdot\eta_\mathsf{F}}{\xi_\mathsf{F}^2+1}\right)\langle\hat{Y},f^{yy}_{xy}\rangle+\langle\hat{Y}\otimes\hat{Y},f^{yy}_{yy}\rangle
\end{array}
\right)=0.
\end{split}
\]

The case $\eta_\mathsf{F}=0$ is easy, so we only consider the case $\eta_\mathsf{F}\neq 0$.
Now we take $\hat{Y}=\epsilon\hat{\eta}_\mathsf{F}+\sqrt{1-\epsilon^2}\hat{\eta}_\mathsf{F}^\perp$,  and $\hat{Y}^\perp=\epsilon\hat{\eta}^\perp_\mathsf{F}-\sqrt{1-\epsilon^2}\hat{\eta}_\mathsf{F}$.
 
Step 1. We start with the identity
\[
\begin{split}
(\xi_\mathsf{F}-\mathrm{i})^2\left(\frac{\hat{Y}\cdot\eta_\mathsf{F}}{\xi_\mathsf{F}^2+1}\right)^2\langle \hat{Y}^\perp\otimes\hat{Y}^\perp,f^{yy}_{xx}\rangle-2(\xi_\mathsf{F}-\mathrm{i})\left(\frac{\hat{Y}\cdot\eta_\mathsf{F}}{\xi_\mathsf{F}^2+1}\right)\langle\hat{Y}^\perp\otimes\hat{Y}^\perp\otimes\hat{Y},f^{yy}_{xy}\rangle\\
+\langle\hat{Y}^\perp\otimes\hat{Y}^\perp\otimes\hat{Y}\otimes\hat{Y},f^{yy}_{yy}\rangle=0.
\end{split}
\]
Denote $\mathfrak{C}=(\xi_\mathsf{F}-\mathrm{i})\frac{|\eta_\mathsf{F}|}{\xi_\mathsf{F}^2+1}$, we have
\begin{equation}\label{kernel_identity_22_ff}
{\scriptsize
\begin{split}
\epsilon^4\mathfrak{C}^2\langle\hat{\eta}_\mathsf{F}^\perp\otimes\hat{\eta}_\mathsf{F}^\perp,f^{yy}_{xx}\rangle-2\epsilon^3\sqrt{1-\epsilon^2}\mathfrak{C}^2\langle\hat{\eta}_\mathsf{F}\otimes\hat{\eta}_\mathsf{F}^\perp,f^{yy}_{xx}\rangle+\epsilon^2(1-\epsilon^2)\mathfrak{C}^2\langle\hat{\eta}_\mathsf{F}\otimes\hat{\eta}_\mathsf{F},f^{yy}_{xx}\rangle\\
-2\epsilon^4\mathfrak{C}\langle\hat{\eta}_\mathsf{F}^\perp\otimes\hat{\eta}_\mathsf{F}^\perp\otimes\hat{\eta}_\mathsf{F},f^{yy}_{xy}\rangle-2\epsilon^3\sqrt{1-\epsilon^2}\mathfrak{C}\langle\hat{\eta}_\mathsf{F}^\perp\otimes\hat{\eta}_\mathsf{F}^\perp\otimes\hat{\eta}_\mathsf{F}^\perp,f^{yy}_{xy}\rangle\\
+4\epsilon^3\sqrt{1-\epsilon^2}\mathfrak{C}\langle\hat{\eta}_\mathsf{F}\otimes\hat{\eta}_\mathsf{F}^\perp\otimes\hat{\eta}_\mathsf{F},f^{yy}_{xy}\rangle+4\epsilon^2(1-\epsilon^2)\mathfrak{C}\langle\hat{\eta}_\mathsf{F}\otimes\hat{\eta}_\mathsf{F}^\perp\otimes\hat{\eta}_\mathsf{F}^\perp,f^{yy}_{xy}\rangle\\
-2\epsilon^2(1-\epsilon^2)\mathfrak{C}\langle\hat{\eta}_\mathsf{F}\otimes\hat{\eta}_\mathsf{F}\otimes\hat{\eta}_\mathsf{F},f^{yy}_{xy}\rangle-2\epsilon\sqrt{(1-\epsilon^2)^3}\mathfrak{C}\langle\hat{\eta}_\mathsf{F}\otimes\hat{\eta}_\mathsf{F}\otimes\hat{\eta}_\mathsf{F}^\perp,f^{yy}_{xy}\rangle\\
+\epsilon^4\langle\hat{\eta}_\mathsf{F}^\perp\otimes\hat{\eta}_\mathsf{F}^\perp\otimes\hat{\eta}_\mathsf{F}\otimes{\eta}_\mathsf{F},f^{yy}_{yy}\rangle+2\epsilon^3\sqrt{1-\epsilon^2}\langle\hat{\eta}_\mathsf{F}^\perp\otimes\hat{\eta}_\mathsf{F}^\perp\otimes\hat{\eta}_\mathsf{F}\otimes{\eta}_\mathsf{F}^\perp,f^{yy}_{yy}\rangle+\epsilon^2(1-\epsilon^2)\langle\hat{\eta}_\mathsf{F}^\perp\otimes\hat{\eta}_\mathsf{F}^\perp\otimes\hat{\eta}^\perp_\mathsf{F}\otimes{\eta}^\perp_\mathsf{F},f^{yy}_{yy}\rangle\\
-2\epsilon^3\sqrt{1-\epsilon^2}\langle\hat{\eta}_\mathsf{F}\otimes\hat{\eta}^\perp_\mathsf{F}\otimes\hat{\eta}_\mathsf{F}\otimes{\eta}_\mathsf{F},f^{yy}_{yy}\rangle-4\epsilon^2(1-\epsilon^2)\langle\hat{\eta}_\mathsf{F}\otimes\hat{\eta}^\perp_\mathsf{F}\otimes\hat{\eta}_\mathsf{F}\otimes{\eta}_\mathsf{F}^\perp,f^{yy}_{yy}\rangle-2\epsilon\sqrt{(1-\epsilon^2)^3}\langle\hat{\eta}_\mathsf{F}\otimes\hat{\eta}^\perp_\mathsf{F}\otimes\hat{\eta}^\perp_\mathsf{F}\otimes{\eta}^\perp_\mathsf{F},f^{yy}_{yy}\rangle\\
+\epsilon^2(1-\epsilon^2)\langle\hat{\eta}_\mathsf{F}\otimes\hat{\eta}_\mathsf{F}\otimes\hat{\eta}_\mathsf{F}\otimes{\eta}_\mathsf{F},f^{yy}_{yy}\rangle+2\epsilon\sqrt{(1-\epsilon^2)^3}\langle\hat{\eta}_\mathsf{F}\otimes\hat{\eta}_\mathsf{F}\otimes\hat{\eta}_\mathsf{F}\otimes{\eta}_\mathsf{F}^\perp,f^{yy}_{yy}\rangle+(1-\epsilon^2)^2\langle\hat{\eta}_\mathsf{F}\otimes\hat{\eta}_\mathsf{F}\otimes\hat{\eta}^\perp_\mathsf{F}\otimes{\eta}^\perp_\mathsf{F},f^{yy}_{yy}\rangle=0.
\end{split}
}
\end{equation}

Taking 1st order derivative of \eqref{kernel_identity_22_ff} in $\epsilon$ at $\epsilon=0$, we obtain
\[
-2\mathfrak{C}\langle\hat{\eta}_\mathsf{F}\otimes\hat{\eta}_\mathsf{F}\otimes\hat{\eta}_\mathsf{F}^\perp, f^{yy}_{xy}\rangle+2\langle\hat{\eta}_\mathsf{F}\otimes\hat{\eta}_\mathsf{F}\otimes\hat{\eta}_\mathsf{F}\otimes\hat{\eta}_\mathsf{F}^\perp,f^{yy}_{yy}\rangle-2\langle\hat{\eta}_\mathsf{F}\otimes\hat{\eta}^\perp_\mathsf{F}\otimes \hat{\eta}^\perp_\mathsf{F}\otimes\hat{\eta}^\perp_\mathsf{F},f^{yy}_{yy}\rangle=0.
\]
Using the condition in \eqref{deltakernel22_ff}, the above equality reduces to
\begin{equation}\label{eq_51}
(\frac{|\eta_\mathsf{F}|^2}{\xi_\mathsf{F}^2+1}+1)\langle\hat{\eta}_\mathsf{F}\otimes\hat{\eta}_\mathsf{F}\otimes\hat{\eta}_\mathsf{F}\otimes\hat{\eta}_\mathsf{F}^\perp,f^{yy}_{yy}\rangle-\langle\hat{\eta}_\mathsf{F}\otimes\hat{\eta}^\perp_\mathsf{F}\otimes \hat{\eta}^\perp_\mathsf{F}\otimes\hat{\eta}^\perp_\mathsf{F},f^{yy}_{yy}\rangle=0.
\end{equation}

Taking 3rd order derivative and using the above identity
and using \eqref{deltakernel22_ff}, we have
\begin{equation}\label{eq_52}
\left(\frac{|\eta_\mathsf{F}|^4}{(\xi_\mathsf{F}^2+1)^2}+2\frac{|\eta_\mathsf{F}|^2}{\xi_\mathsf{F}^2+1}+1\right)\langle\hat{\eta}_\mathsf{F}\otimes\hat{\eta}^\perp_\mathsf{F}\otimes\hat{\eta}_\mathsf{F}\otimes\hat{\eta}_\mathsf{F},f^{yy}_{yy}\rangle-\left(\frac{|\eta_\mathsf{F}|^2}{\xi_\mathsf{F}^2+1}+1\right)\langle\hat{\eta}_\mathsf{F}^\perp\otimes\hat{\eta}_\mathsf{F}^\perp\otimes\hat{\eta}_\mathsf{F}\otimes{\eta}_\mathsf{F}^\perp,f^{yy}_{yy}\rangle=0.
\end{equation}

Taking 2nd order derivative, together with \eqref{deltakernel22_ff}, gives
\begin{equation}\label{eq_53}
\begin{split}
\left(\frac{|\eta_\mathsf{F}|^4}{(\xi_\mathsf{F}^2+1)^2}+2\frac{|\eta_\mathsf{F}|^2}{\xi_\mathsf{F}^2+1}+1\right)\langle\hat{\eta}_\mathsf{F}\otimes\hat{\eta}_\mathsf{F}\otimes\hat{\eta}_\mathsf{F}\otimes\hat{\eta}_\mathsf{F},f^{yy}_{yy}\rangle-4\left(\frac{|\eta_\mathsf{F}|^2}{\xi_\mathsf{F}^2+1}+1\right)\langle\hat{\eta}_\mathsf{F}\otimes\hat{\eta}^\perp_\mathsf{F}\otimes\hat{\eta}_\mathsf{F}\otimes\hat{\eta}^\perp_\mathsf{F}, f^{yy}_{yy}\rangle\\
+\langle\hat{\eta}^\perp_\mathsf{F}\otimes\hat{\eta}^\perp_\mathsf{F}\otimes\hat{\eta}^\perp_\mathsf{F}\otimes\hat{\eta}^\perp_\mathsf{F}, f^{yy}_{yy}\rangle=0.
\end{split}
\end{equation}

Taking 4th order derivative, together with \eqref{deltakernel22_ff}, we obtain
\[
\left(\frac{|\eta_\mathsf{F}|^4}{(\xi_\mathsf{F}^2+1)^2}+2\frac{|\eta_\mathsf{F}|^2}{\xi_\mathsf{F}^2+1}+1\right)\langle\hat{\eta}^\perp_\mathsf{F}\otimes\hat{\eta}^\perp_\mathsf{F}\otimes\hat{\eta}_\mathsf{F}\otimes\hat{\eta}_\mathsf{F}, f^{yy}_{yy}\rangle=0,
\]
and then
\[
\langle\hat{\eta}^\perp_\mathsf{F}\otimes\hat{\eta}^\perp_\mathsf{F}, f^{yy}_{xx}\rangle=\langle\hat{\eta}^\perp_\mathsf{F}\otimes\hat{\eta}^\perp_\mathsf{F}\otimes\hat{\eta}_\mathsf{F}, f^{yy}_{xy}\rangle=\langle\hat{\eta}^\perp_\mathsf{F}\otimes\hat{\eta}^\perp_\mathsf{F}\otimes\hat{\eta}_\mathsf{F}\otimes\hat{\eta}_\mathsf{F}, f^{yy}_{yy}\rangle=0.
\]

Step 2. We consider the identity
\[
\begin{split}
(\xi_\mathsf{F}-\mathrm{i})^2\left(\frac{\hat{Y}\cdot\eta_\mathsf{F}}{\xi_\mathsf{F}^2+1}\right)^2\langle\hat{Y}^\perp,f^{xy}_{xx}\rangle-2(\xi_\mathsf{F}-\mathrm{i})\left(\frac{\hat{Y}\cdot\eta_\mathsf{F}}{\xi_\mathsf{F}^2+1}\right)\langle\hat{Y}^\perp\otimes\hat{Y},f^{xy}_{xy}\rangle+\langle\hat{Y}^\perp\otimes\hat{Y}\otimes\hat{Y},f^{xy}_{yy}\rangle\\
+(\xi_\mathsf{F}-\mathrm{i})^3\left(\frac{\hat{Y}\cdot\eta_\mathsf{F}}{\xi_\mathsf{F}^2+1}\right)^3\langle\hat{Y}^\perp\otimes\hat{Y},f^{yy}_{xx}\rangle-2(\xi_\mathsf{F}-\mathrm{i})^2\left(\frac{\hat{Y}\cdot\eta_\mathsf{F}}{\xi_\mathsf{F}^2+1}\right)^2\langle\hat{Y}^\perp\otimes\hat{Y}\otimes\hat{Y},f^{yy}_{xy}\rangle\\
+(\xi_\mathsf{F}-\mathrm{i})\left(\frac{\hat{Y}\cdot\eta_\mathsf{F}}{\xi_\mathsf{F}^2+1}\right)\langle\hat{Y}^\perp\otimes\hat{Y}\otimes\hat{Y}\otimes\hat{Y},f^{yy}_{yy}\rangle=0.
\end{split}
\]

Similar as above, we obtain
\begin{equation}\label{kernel_identity_22_2_ff}
\begin{split}
\epsilon^3\mathfrak{C}^2\langle\hat{\eta}_\mathsf{F}^\perp,f^{xy}_{xx}\rangle-\epsilon^2\sqrt{1-\epsilon^2}\mathfrak{C}^2\langle\hat{\eta}_\mathsf{F},f^{xy}_{xx}\rangle\\
-2\epsilon^3\mathfrak{C}\langle\hat{\eta}^\perp_\mathsf{F}\otimes\hat{\eta}_\mathsf{F},f^{xy}_{xy}\rangle-2\epsilon^2\sqrt{1-\epsilon^2}\mathfrak{C}\langle\hat{\eta}^\perp_\mathsf{F}\otimes\hat{\eta}^\perp_\mathsf{F}-\hat{\eta}_\mathsf{F}\otimes\hat{\eta}_\mathsf{F},f^{xy}_{xy}\rangle+2\epsilon(1-\epsilon^2)\mathfrak{C}\langle\hat{\eta}_\mathsf{F}\otimes\hat{\eta}_\mathsf{F}^\perp, f^{xy}_{xy}\rangle\\
+\epsilon^3\langle\hat{\eta}^\perp_\mathsf{F}\otimes\hat{\eta}_\mathsf{F}\otimes\hat{\eta}_\mathsf{F},f^{xy}_{yy}\rangle-\epsilon^2\sqrt{1-\epsilon^2}\langle \hat{\eta}_\mathsf{F}\otimes\hat{\eta}_\mathsf{F}\otimes\hat{\eta}_\mathsf{F},f^{xy}_{yy}\rangle
+2\epsilon^2\sqrt{1-\epsilon^2}\langle \hat{\eta}_\mathsf{F}^\perp\otimes\hat{\eta}_\mathsf{F}\otimes\hat{\eta}_\mathsf{F}^\perp,f^{xy}_{yy}\rangle\\
+\epsilon(1-\epsilon^2)\langle\hat{\eta}^\perp_\mathsf{F}\otimes \hat{\eta}^\perp_\mathsf{F}\otimes\hat{\eta}^\perp_\mathsf{F},f^{xy}_{yy}\rangle-2\epsilon(1-\epsilon^2)\langle\hat{\eta}_\mathsf{F}\otimes \hat{\eta}_\mathsf{F}\otimes\hat{\eta}^\perp_\mathsf{F},f^{xy}_{yy}\rangle
-\sqrt{(1-\epsilon^2)^3}\langle\hat{\eta}_\mathsf{F}\otimes\hat{\eta}_\mathsf{F}^\perp\otimes\hat{\eta}_\mathsf{F}^\perp,f^{xy}_{yy}\rangle\\
+\epsilon^5\mathfrak{C}^3\langle\hat{\eta}_\mathsf{F}^\perp\otimes\hat{\eta}_\mathsf{F},f^{yy}_{xx}\rangle+\epsilon^4\sqrt{1-\epsilon^2}\mathfrak{C}^3\langle\hat{\eta}^\perp_\mathsf{F}\otimes\hat{\eta}_\mathsf{F}^\perp-\hat{\eta}_\mathsf{F}\otimes\hat{\eta}_\mathsf{F},f^{yy}_{xx}\rangle-\epsilon^3(1-\epsilon^2)\mathfrak{C}^3\langle\hat{\eta}_\mathsf{F}\otimes\hat{\eta}_\mathsf{F}^\perp, f^{yy}_{xx}\rangle\\
-2\epsilon^5\mathfrak{C}^2\langle\hat{\eta}^\perp_\mathsf{F}\otimes\hat{\eta}_\mathsf{F}\otimes\hat{\eta}_\mathsf{F},f^{yy}_{xy}\rangle+2\epsilon^4\sqrt{1-\epsilon^2}\mathfrak{C}^2\langle \hat{\eta}_\mathsf{F}\otimes\hat{\eta}_\mathsf{F}\otimes\hat{\eta}_\mathsf{F},f^{yy}_{xy}\rangle
-2\epsilon^4\sqrt{1-\epsilon^2}\mathfrak{C}^2\langle \hat{\eta}_\mathsf{F}^\perp\otimes\hat{\eta}_\mathsf{F}\otimes\hat{\eta}_\mathsf{F}^\perp,f^{yy}_{xy}\rangle\\
-2\epsilon^4\sqrt{1-\epsilon^2}\mathfrak{C}^2\langle \hat{\eta}_\mathsf{F}^\perp\otimes\hat{\eta}_\mathsf{F}^\perp\otimes\hat{\eta}_\mathsf{F},f^{yy}_{xy}\rangle
-2\epsilon^3(1-\epsilon^2)\mathfrak{C}^2\langle\hat{\eta}^\perp_\mathsf{F}\otimes \hat{\eta}^\perp_\mathsf{F}\otimes\hat{\eta}^\perp_\mathsf{F},f^{yy}_{xy}\rangle\\
+2\epsilon^3(1-\epsilon^2)\mathfrak{C}^2\langle\hat{\eta}_\mathsf{F}\otimes \hat{\eta}_\mathsf{F}\otimes\hat{\eta}^\perp_\mathsf{F},f^{yy}_{xy}\rangle+2\epsilon^3(1-\epsilon^2)\mathfrak{C}^2\langle\hat{\eta}_\mathsf{F}\otimes \hat{\eta}^\perp_\mathsf{F}\otimes\hat{\eta}_\mathsf{F},f^{yy}_{xy}\rangle\\
+2\epsilon^2\sqrt{(1-\epsilon^2)^3}\mathfrak{C}^2\langle\hat{\eta}_\mathsf{F}\otimes\hat{\eta}_\mathsf{F}^\perp\otimes\hat{\eta}_\mathsf{F}^\perp,f^{yy}_{xy}\rangle
+(-\epsilon^3(1-\epsilon^2)+\epsilon^5)\mathfrak{C}\langle\hat{\eta}_\mathsf{F}\otimes\hat{\eta}_\mathsf{F}^\perp\otimes\hat{\eta}_\mathsf{F}\otimes\hat{\eta}_\mathsf{F},f^{yy}_{yy}\rangle\\
-\epsilon^4\sqrt{1-\epsilon^2}\mathfrak{C}\langle\hat{\eta}_\mathsf{F}\otimes\hat{\eta}_\mathsf{F}\otimes\hat{\eta}_\mathsf{F}\otimes\hat{\eta}_\mathsf{F},f^{yy}_{yy}\rangle+\epsilon^4\sqrt{1-\epsilon^2}\mathfrak{C}\langle\hat{\eta}^\perp_\mathsf{F}\otimes\hat{\eta}^\perp_\mathsf{F}\otimes\hat{\eta}_\mathsf{F}\otimes\hat{\eta}_\mathsf{F},f^{yy}_{yy}\rangle\\
+(-2\epsilon^2(1-\epsilon^2)+2\epsilon^4)\sqrt{1-\epsilon^2}\mathfrak{C}\langle\hat{\eta}_\mathsf{F}\otimes\hat{\eta}_\mathsf{F}^\perp\otimes\hat{\eta}_\mathsf{F}\otimes\hat{\eta}^\perp_\mathsf{F},f^{yy}_{yy}\rangle\\
-2\epsilon^3(1-\epsilon^2)\mathfrak{C}\langle\hat{\eta}_\mathsf{F}\otimes\hat{\eta}_\mathsf{F}\otimes\hat{\eta}_\mathsf{F}\otimes\hat{\eta}^\perp_\mathsf{F},f^{yy}_{yy}\rangle+2\epsilon^3(1-\epsilon^2)\mathfrak{C}\langle\hat{\eta}^\perp_\mathsf{F}\otimes\hat{\eta}^\perp_\mathsf{F}\otimes\hat{\eta}_\mathsf{F}\otimes\hat{\eta}^\perp_\mathsf{F},f^{yy}_{yy}\rangle\\
+(-\epsilon(1-\epsilon^2)^2+\epsilon^3(1-\epsilon^2))\mathfrak{C}\langle\hat{\eta}_\mathsf{F}\otimes\hat{\eta}_\mathsf{F}^\perp\otimes\hat{\eta}^\perp_\mathsf{F}\otimes\hat{\eta}^\perp_\mathsf{F},f^{yy}_{yy}\rangle-\epsilon^2\sqrt{(1-\epsilon^2)^3}\mathfrak{C}\langle\hat{\eta}_\mathsf{F}\otimes\hat{\eta}_\mathsf{F}\otimes\hat{\eta}^\perp_\mathsf{F}\otimes\hat{\eta}^\perp_\mathsf{F},f^{yy}_{yy}\rangle\\
+\epsilon^2\sqrt{(1-\epsilon^2)^3}\mathfrak{C}\langle\hat{\eta}^\perp_\mathsf{F}\otimes\hat{\eta}^\perp_\mathsf{F}\otimes\hat{\eta}^\perp_\mathsf{F}\otimes\hat{\eta}^\perp_\mathsf{F},f^{yy}_{yy}\rangle=0.
\end{split}
\end{equation}

Taking 1st order derivative of \eqref{kernel_identity_22_2_ff} in $\epsilon$ at $\epsilon=0$, and using the condition in \eqref{deltakernel22_ff}, we obtain
\begin{equation}\label{eq_55}
\begin{split}
2\left(\frac{|\eta_\mathsf{F}|^2}{\xi_\mathsf{F}^2+1}+1\right)\langle\hat{\eta}_\mathsf{F}\otimes\hat{\eta}_\mathsf{F}\otimes\hat{\eta}_\mathsf{F}^\perp,f^{xy}_{yy}\rangle-\langle\hat{\eta}^\perp_\mathsf{F}\otimes \hat{\eta}^\perp_\mathsf{F}\otimes\hat{\eta}^\perp_\mathsf{F},f^{xy}_{yy}\rangle\\
+\frac{|\eta_\mathsf{F}|}{\xi_\mathsf{F}^2+1}(\xi_\mathsf{F}-\mathrm{i})\langle\hat{\eta}_\mathsf{F}\otimes\hat{\eta}_\mathsf{F}^\perp\otimes\hat{\eta}^\perp_\mathsf{F}\otimes\hat{\eta}^\perp_\mathsf{F},f^{yy}_{yy}\rangle=0.
\end{split}
\end{equation}
Taking 3rd order derivative and using the above identity, we have
using \eqref{deltakernel22_ff}, we have
\begin{equation}\label{eq_56}
\begin{split}
\left(\frac{|\eta_\mathsf{F}|^4}{(\xi_\mathsf{F}^2+1)^2}+2\frac{|\eta_\mathsf{F}|^2}{\xi_\mathsf{F}^2+1}+1\right)\langle\hat{\eta}^\perp_\mathsf{F}\otimes\hat{\eta}_\mathsf{F}\otimes\hat{\eta}_\mathsf{F},f^{xy}_{yy}\rangle\\
-\left(\frac{|\eta_\mathsf{F}|^4}{(\xi_\mathsf{F}^2+1)^2}+2\frac{|\eta_\mathsf{F}|^2}{\xi_\mathsf{F}^2+1}+1\right)\frac{|\eta_\mathsf{F}|}{\xi_\mathsf{F}^2+1}(\xi_\mathsf{F}-\mathrm{i})\langle\hat{\eta}_\mathsf{F}\otimes\hat{\eta}^\perp_\mathsf{F}\otimes\hat{\eta}_\mathsf{F}\otimes\hat{\eta}_\mathsf{F},f^{yy}_{yy}\rangle\\
+2\left(\frac{|\eta_\mathsf{F}|^2}{\xi_\mathsf{F}^2+1}+1\right)\frac{|\eta_\mathsf{F}|}{\xi_\mathsf{F}^2+1}(\xi_\mathsf{F}-\mathrm{i})\langle\hat{\eta}_\mathsf{F}^\perp\otimes\hat{\eta}^\perp_\mathsf{F}\otimes\hat{\eta}_\mathsf{F}\otimes\hat{\eta}^\perp_\mathsf{F},f^{yy}_{yy}\rangle\\
-2\left(\frac{|\eta_\mathsf{F}|^2}{\xi_\mathsf{F}^2+1}+1\right)\frac{|\eta_\mathsf{F}|}{\xi_\mathsf{F}^2+1}(\xi_\mathsf{F}-\mathrm{i})\langle\hat{\eta}_1\otimes\hat{\eta}_\mathsf{F}\otimes\hat{\eta}_\mathsf{F}\otimes\hat{\eta}^\perp_\mathsf{F},f^{yy}_{yy}\rangle\\
+\frac{|\eta_\mathsf{F}|}{\xi_\mathsf{F}^2+1}(\xi_\mathsf{F}-\mathrm{i})\langle\hat{\eta}_\mathsf{F}\otimes\hat{\eta}^\perp_\mathsf{F}\otimes\hat{\eta}^\perp_\mathsf{F}\otimes\hat{\eta}_\mathsf{F}^\perp,f^{yy}_{yy}\rangle=0.
\end{split}
\end{equation}

Taking 5th order derivative we have
\[
\mathfrak{C}^2\langle\hat{\eta}_\mathsf{F}^\perp\otimes\hat{\eta}_\mathsf{F},f^{yy}_{xx}\rangle-2\mathfrak{C}\langle\hat{\eta}^\perp_\mathsf{F}\otimes\hat{\eta}_\mathsf{F}\otimes\hat{\eta}_\mathsf{F},f^{yy}_{xy}\rangle+\langle\hat{\eta}_\mathsf{F}\otimes\hat{\eta}_\mathsf{F}^\perp\otimes\hat{\eta}_\mathsf{F}\otimes\hat{\eta}_\mathsf{F},f^{yy}_{yy}\rangle=0.
\]
Using the gauge condition we have
\[
\left(\frac{|\eta_\mathsf{F}|^4}{(\xi_\mathsf{F}^2+1)^2}+2\frac{|\eta_\mathsf{F}|^2}{\xi_\mathsf{F}^2+1}+1\right)\langle\hat{\eta}_\mathsf{F}\otimes\hat{\eta}_\mathsf{F}^\perp\otimes\hat{\eta}_\mathsf{F}\otimes\hat{\eta}_\mathsf{F},f^{yy}_{yy}\rangle=0
\]
and so
\[
\langle\hat{\eta}_\mathsf{F}^\perp\otimes\hat{\eta}_\mathsf{F},f^{yy}_{xx}\rangle=\langle\hat{\eta}^\perp_\mathsf{F}\otimes\hat{\eta}_\mathsf{F}\otimes\hat{\eta}_\mathsf{F},f^{yy}_{xy}\rangle=\langle\hat{\eta}_\mathsf{F}\otimes\hat{\eta}_\mathsf{F}^\perp\otimes\hat{\eta}_\mathsf{F}\otimes\hat{\eta}_\mathsf{F},f^{yy}_{yy}\rangle=0.
\]
Substituting into \eqref{eq_52}, we have
\[
\langle\hat{\eta}_\mathsf{F}^\perp\otimes\hat{\eta}_\mathsf{F}^\perp\otimes{\eta}_\mathsf{F}^\perp,f^{yy}_{xy}\rangle=\langle\hat{\eta}_\mathsf{F}^\perp\otimes\hat{\eta}_\mathsf{F}^\perp\otimes\hat{\eta}_\mathsf{F}\otimes{\eta}_\mathsf{F}^\perp,f^{yy}_{yy}\rangle=0.
\]

Using the trace-free condition
\[
\langle\hat{\eta}_\mathsf{F}^\perp\otimes\hat{\eta}_\mathsf{F},f^{xy}_{xy}\rangle+\langle\hat{\eta}_\mathsf{F}\otimes\hat{\eta}_\mathsf{F}^\perp\otimes\hat{\eta}_\mathsf{F}\otimes\hat{\eta}_\mathsf{F},f^{yy}_{yy}\rangle+\langle\hat{\eta}_1^\perp\otimes\hat{\eta}_\mathsf{F}^\perp\otimes\hat{\eta}_\mathsf{F}\otimes{\eta}_\mathsf{F}^\perp,f^{yy}_{yy}\rangle=0,
\]
we have
\[
\langle\hat{\eta}_\mathsf{F}^\perp,f^{xy}_{xx}\rangle=\langle\hat{\eta}_\mathsf{F}^\perp\otimes\hat{\eta}_\mathsf{F},f^{xy}_{xy}\rangle=\langle\hat{\eta}_\mathsf{F}^\perp\otimes\hat{\eta}_\mathsf{F}\otimes\hat{\eta}_\mathsf{F},f^{xy}_{yy}\rangle=0.
\]
Substituting into \eqref{eq_56}, we have
\[
-2\left(\frac{|\eta_\mathsf{F}|^2}{\xi_\mathsf{F}^2+1}+1\right)\langle\hat{\eta}_\mathsf{F}\otimes\hat{\eta}_\mathsf{F}\otimes\hat{\eta}_\mathsf{F}\otimes\hat{\eta}^\perp_\mathsf{F},f^{yy}_{yy}\rangle\\
+\langle\hat{\eta}_\mathsf{F}\otimes\hat{\eta}^\perp_\mathsf{F}\otimes\hat{\eta}^\perp_\mathsf{F}\otimes\hat{\eta}_\mathsf{F}^\perp,f^{yy}_{yy}\rangle=0.
\]
Together with \eqref{eq_51}, we have
\[
\langle\hat{\eta}_\mathsf{F}\otimes\hat{\eta}_\mathsf{F}\otimes\hat{\eta}_\mathsf{F}\otimes\hat{\eta}^\perp_\mathsf{F},f^{yy}_{yy}\rangle=\langle\hat{\eta}_\mathsf{F}\otimes\hat{\eta}^\perp_\mathsf{F}\otimes\hat{\eta}^\perp_\mathsf{F}\otimes\hat{\eta}_\mathsf{F}^\perp,f^{yy}_{yy}\rangle=0.
\]
Using the gauge condition, we have
\[
\langle\hat{\eta}_\mathsf{F}\otimes\hat{\eta}_\mathsf{F}\otimes\hat{\eta}^\perp_\mathsf{F},f^{yy}_{xy}\rangle=0.
\]
Using the trace-free condition, we have
we have
\[
\langle\hat{\eta}_\mathsf{F}\otimes\hat{\eta}^\perp_\mathsf{F},f^{xy}_{xy}\rangle=\langle\hat{\eta}_\mathsf{F}\otimes\hat{\eta}_\mathsf{F}\otimes\hat{\eta}^\perp_\mathsf{F},f^{xy}_{yy}\rangle=0.
\]
Using \eqref{eq_55} again, we have
\[
\langle\hat{\eta}^\perp_\mathsf{F}\otimes \hat{\eta}^\perp_\mathsf{F}\otimes\hat{\eta}^\perp_\mathsf{F},f^{xy}_{yy}\rangle=0.
\]

Next we consider the even powers of $\epsilon$. Dividing by $\epsilon^2\sqrt{1-\epsilon^2}$, we obtain
\begin{equation}\label{kernel_identity_22_2_2_ff}
\begin{split}
-\mathfrak{C}^2\langle\hat{\eta}_\mathsf{F},f^{xy}_{xx}\rangle
-2\mathfrak{C}\langle\hat{\eta}^\perp_\mathsf{F}\otimes\hat{\eta}^\perp_\mathsf{F}-\hat{\eta}_\mathsf{F}\otimes\hat{\eta}_\mathsf{F},f^{xy}_{xy}\rangle
-\langle \hat{\eta}_\mathsf{F}\otimes\hat{\eta}_\mathsf{F}\otimes\hat{\eta}_\mathsf{F},f^{xy}_{yy}\rangle\\
+2\langle \hat{\eta}_1^\perp\otimes\hat{\eta}_\mathsf{F}\otimes\hat{\eta}_\mathsf{F}^\perp,f^{xy}_{yy}\rangle
+\epsilon^2\mathfrak{C}^3\langle\hat{\eta}^\perp_\mathsf{F}\otimes\hat{\eta}_\mathsf{F}^\perp-\hat{\eta}_\mathsf{F}\otimes\hat{\eta}_\mathsf{F},f^{yy}_{xx}\rangle
+2\epsilon^2\mathfrak{C}^2\langle \hat{\eta}_\mathsf{F}\otimes\hat{\eta}_\mathsf{F}\otimes\hat{\eta}_\mathsf{F},f^{yy}_{xy}\rangle\\
-2\epsilon^2\mathfrak{C}^2\langle \hat{\eta}_\mathsf{F}^\perp\otimes\hat{\eta}_\mathsf{F}\otimes\hat{\eta}_\mathsf{F}^\perp,f^{yy}_{xy}\rangle
-2\epsilon^2\mathfrak{C}^2\langle \hat{\eta}_\mathsf{F}^\perp\otimes\hat{\eta}_\mathsf{F}^\perp\otimes\hat{\eta}_\mathsf{F},f^{yy}_{xy}\rangle\\
+2(1-\epsilon^2)\mathfrak{C}^2\langle\hat{\eta}_\mathsf{F}\otimes\hat{\eta}_\mathsf{F}^\perp\otimes\hat{\eta}_\mathsf{F}^\perp,f^{yy}_{xy}\rangle
-\epsilon^2\mathfrak{C}\langle\hat{\eta}_\mathsf{F}\otimes\hat{\eta}_\mathsf{F}\otimes\hat{\eta}_\mathsf{F}\otimes\hat{\eta}_\mathsf{F},f^{yy}_{yy}\rangle\\
+\epsilon^2\mathfrak{C}\langle\hat{\eta}^\perp_\mathsf{F}\otimes\hat{\eta}^\perp_\mathsf{F}\otimes\hat{\eta}_\mathsf{F}\otimes\hat{\eta}_\mathsf{F},f^{yy}_{yy}\rangle\
-2(1-2\epsilon^2)\mathfrak{C}\langle\hat{\eta}_\mathsf{F}\otimes\hat{\eta}_\mathsf{F}^\perp\otimes\hat{\eta}_\mathsf{F}\otimes\hat{\eta}^\perp_\mathsf{F},f^{yy}_{yy}\rangle\\
-(1-\epsilon^2)\mathfrak{C}\langle\hat{\eta}_\mathsf{F}\otimes\hat{\eta}_\mathsf{F}\otimes\hat{\eta}^\perp_\mathsf{F}\otimes\hat{\eta}^\perp_\mathsf{F},f^{yy}_{yy}\rangle
+(1-\epsilon^2)\mathfrak{C}\langle\hat{\eta}^\perp_\mathsf{F}\otimes\hat{\eta}^\perp_\mathsf{F}\otimes\hat{\eta}^\perp_\mathsf{F}\otimes\hat{\eta}^\perp_\mathsf{F},f^{yy}_{yy}\rangle=0.
\end{split}
\end{equation}
Setting $\epsilon=0$, and using the gauge condition, we have
\[
\begin{split}
-\left(\frac{|\eta_\mathsf{F}|^4}{(\xi_\mathsf{F}^2+1)^2}+2\frac{|\eta_\mathsf{F}|^2}{\xi_\mathsf{F}^2+1}+1\right)\langle \hat{\eta}_\mathsf{F}\otimes\hat{\eta}_\mathsf{F}\otimes\hat{\eta}_\mathsf{F},f^{xy}_{yy}\rangle
-2\left(\frac{|\eta_\mathsf{F}|^2}{\xi_\mathsf{F}^2+1}+1\right)\langle \hat{\eta}_1^\perp\otimes\hat{\eta}_\mathsf{F}\otimes\hat{\eta}_\mathsf{F}^\perp,f^{xy}_{yy}\rangle\\
-2\left(\frac{|\eta_\mathsf{F}|^2}{\xi_\mathsf{F}^2+1}+1\right)\frac{|\eta_\mathsf{F}|}{\xi_\mathsf{F}^2+1}(\xi_\mathsf{F}-\mathrm{i})\langle\hat{\eta}_\mathsf{F}\otimes\hat{\eta}_\mathsf{F}^\perp\otimes\hat{\eta}_\mathsf{F}\otimes\hat{\eta}^\perp_\mathsf{F},f^{yy}_{yy}\rangle-\frac{|\eta_\mathsf{F}|}{\xi_\mathsf{F}^2+1}(\xi_\mathsf{F}-\mathrm{i})\langle\hat{\eta}_\mathsf{F}\otimes\hat{\eta}_\mathsf{F}\otimes\hat{\eta}^\perp_\mathsf{F}\otimes\hat{\eta}^\perp_\mathsf{F},f^{yy}_{yy}\rangle\\
+\frac{|\eta_\mathsf{F}|}{\xi_\mathsf{F}^2+1}(\xi_\mathsf{F}-\mathrm{i})\langle\hat{\eta}^\perp_\mathsf{F}\otimes\hat{\eta}^\perp_\mathsf{F}\otimes\hat{\eta}^\perp_\mathsf{F}\otimes\hat{\eta}^\perp_\mathsf{F},f^{yy}_{yy}\rangle=0,
\end{split}
\]
which can be rewritten as
\[
\begin{split}
\left(\frac{|\eta_\mathsf{F}|^4}{(\xi_\mathsf{F}^2+1)^2}+2\frac{|\eta_\mathsf{F}|^2}{\xi_\mathsf{F}^2+1}+1\right)\langle \hat{\eta}_\mathsf{F}\otimes\hat{\eta}_\mathsf{F},f^{xy}_{yy}\rangle
-2\left(\frac{|\eta_\mathsf{F}|^2}{\xi_\mathsf{F}^2+1}+1\right)\langle \hat{\eta}_1^\perp\otimes\hat{\eta}_\mathsf{F}^\perp,f^{xy}_{xy}\rangle\\
-2\left(\frac{|\eta_\mathsf{F}|^4}{(\xi_\mathsf{F}^2+1)^4}+\frac{|\eta_\mathsf{F}|^2}{\xi_\mathsf{F}^2+1}\right)\langle\hat{\eta}_\mathsf{F}\otimes\hat{\eta}_\mathsf{F}^\perp\otimes\hat{\eta}_\mathsf{F}\otimes\hat{\eta}^\perp_\mathsf{F},f^{yy}_{yy}\rangle
+\frac{|\eta_\mathsf{F}|^2}{\xi_\mathsf{F}^2+1}\langle\hat{\eta}^\perp_\mathsf{F}\otimes\hat{\eta}^\perp_\mathsf{F}\otimes\hat{\eta}^\perp_\mathsf{F}\otimes\hat{\eta}^\perp_\mathsf{F},f^{yy}_{yy}\rangle=0.
\end{split}
\]

Taking 2nd order derivative gives
\begin{equation}
\begin{split}
\left(\frac{|\eta_\mathsf{F}|^4}{(\xi_\mathsf{F}^2+1)^2}+2\frac{|\eta_\mathsf{F}|^2}{\xi_\mathsf{F}^2+1}+1\right)\langle\hat{\eta}^\perp_\mathsf{F}\otimes\hat{\eta}^\perp_\mathsf{F}\otimes\hat{\eta}_\mathsf{F}\otimes\hat{\eta}_\mathsf{F},f^{yy}_{yy}\rangle\\
-\left(\frac{|\eta_\mathsf{F}|^4}{(\xi_\mathsf{F}^2+1)^2}+2\frac{|\eta_\mathsf{F}|^2}{\xi_\mathsf{F}^2+1}+1\right)\langle\hat{\eta}_\mathsf{F}\otimes\hat{\eta}_\mathsf{F}\otimes\hat{\eta}_\mathsf{F}\otimes\hat{\eta}_\mathsf{F},f^{yy}_{yy}\rangle
+4\left(\frac{|\eta_\mathsf{F}|^2}{\xi_\mathsf{F}^2+1}+1\right)\langle\hat{\eta}_\mathsf{F}\otimes\hat{\eta}_\mathsf{F}^\perp\otimes\hat{\eta}_\mathsf{F}\otimes\hat{\eta}^\perp_\mathsf{F},f^{yy}_{yy}\rangle\\+\langle\hat{\eta}_\mathsf{F}\otimes\hat{\eta}_\mathsf{F}\otimes\hat{\eta}^\perp_\mathsf{F}\otimes\hat{\eta}^\perp_\mathsf{F},f^{yy}_{yy}\rangle
-\langle\hat{\eta}^\perp_\mathsf{F}\otimes\hat{\eta}^\perp_\mathsf{F}\otimes\hat{\eta}^\perp_\mathsf{F}\otimes\hat{\eta}^\perp_\mathsf{F},f^{yy}_{yy}\rangle=0.
\end{split}
\end{equation}

Together with \eqref{eq_53} and the trace-free conditions
\[
\langle\hat{\eta}_\mathsf{F}^\perp\otimes\hat{\eta}_\mathsf{F}^\perp,f^{xy}_{xy}\rangle+\langle\hat{\eta}_\mathsf{F}\otimes\hat{\eta}_\mathsf{F}^\perp\otimes\hat{\eta}_\mathsf{F}^\perp\otimes \hat{\eta}_\mathsf{F}, f^{yy}_{yy}\rangle+\langle\hat{\eta}_\mathsf{F}^\perp\otimes\hat{\eta}_\mathsf{F}^\perp\otimes\hat{\eta}_\mathsf{F}^\perp\otimes\hat{\eta}_\mathsf{F}^\perp,f^{yy}_{yy}\rangle=0,
\]
\[
\langle\hat{\eta}_\mathsf{F}\otimes\hat{\eta}_\mathsf{F},f^{xy}_{xy}\rangle+\langle\hat{\eta}_1\otimes\hat{\eta}_\mathsf{F}\otimes\hat{\eta}_\mathsf{F}\otimes\hat{\eta}_\mathsf{F},f^{yy}_{yy}\rangle+\langle\hat{\eta}_\mathsf{F}\otimes\hat{\eta}_\mathsf{F}^\perp\otimes\hat{\eta}_\mathsf{F}^\perp\otimes \hat{\eta}_\mathsf{F}, f^{yy}_{yy}\rangle=0,
\]
we now have
\[
\begin{split}
\left(\begin{array}{ccccc}
0&0&\frac{|\eta_\mathsf{F}|^4}{(\xi_\mathsf{F}^2+1)^2}+2\frac{|\eta_\mathsf{F}|^2}{\xi_\mathsf{F}^2+1}+1&\frac{|\eta|^2}{\xi^2}+1&1\\
\frac{|\eta_\mathsf{F}|^4}{(\xi_\mathsf{F}^2+1)^2}+2\frac{|\eta_\mathsf{F}|^2}{\xi_\mathsf{F}^2+1}+1&-2(\frac{|\eta_\mathsf{F}|^2}{\xi_\mathsf{F}^2+1}+1)&0&2(\frac{|\eta_\mathsf{F}|^4}{(\xi_\mathsf{F}^2+1)^2}+\frac{|\eta_\mathsf{F}|^2}{\xi_\mathsf{F}^2+1})&\frac{|\eta_\mathsf{F}|^2}{\xi_\mathsf{F}^2+1}\\
0&0&-\frac{|\eta_\mathsf{F}|^4}{(\xi_\mathsf{F}^2+1)^2}-2\frac{|\eta_\mathsf{F}|^2}{\xi_\mathsf{F}^2+1}-1&4(\frac{|\eta_\mathsf{F}|^2}{\xi_\mathsf{F}^2+1}+1)&-1\\
1 & 0 & 1&1&1\\
0&1&0&1&1
\end{array}\right)\\
\left(\begin{array}{c}
\langle\hat{\eta}_\mathsf{F}\otimes\hat{\eta}_\mathsf{F},f^{xy}_{xy}\rangle\\
\langle\hat{\eta}_\mathsf{F}^\perp\otimes\hat{\eta}_\mathsf{F}^\perp,f^{xy}_{xy}\rangle\\
\langle\hat{\eta}_\mathsf{F}\otimes\hat{\eta}_\mathsf{F}\otimes\hat{\eta}_\mathsf{F}\otimes\hat{\eta}_\mathsf{F},f^{yy}_{yy}\rangle\\
\langle\hat{\eta}_\mathsf{F}\otimes\hat{\eta}^\perp_\mathsf{F}\otimes \hat{\eta}_\mathsf{F}\otimes\hat{\eta}_\mathsf{F}^\perp, f^{yy}_{yy}\rangle\\
\langle\hat{\eta}_\mathsf{F}^\perp\otimes\hat{\eta}_\mathsf{F}^\perp\otimes\hat{\eta}_\mathsf{F}^\perp\otimes\hat{\eta}_\mathsf{F}^\perp,f^{yy}_{yy}\rangle
\end{array}\right)=0.
\end{split}
\]
So
\[
\begin{split}
&\langle\hat{\eta}_\mathsf{F}\otimes\hat{\eta}_\mathsf{F},f^{xy}_{xy}\rangle=
\langle\hat{\eta}_\mathsf{F}^\perp\otimes\hat{\eta}^\perp_\mathsf{F},f^{xy}_{xy}\rangle=
\langle\hat{\eta}_\mathsf{F}\otimes\hat{\eta}_\mathsf{F}\otimes\hat{\eta}_\mathsf{F}\otimes\hat{\eta}_\mathsf{F},f^{yy}_{yy}\rangle\\
=&\langle\hat{\eta}_\mathsf{F}\otimes\hat{\eta}_\mathsf{F}^\perp\otimes \hat{\eta}_\mathsf{F}\otimes\hat{\eta}_\mathsf{F}^\perp, f^{yy}_{yy}\rangle=
\langle\hat{\eta}_\mathsf{F}^\perp\otimes\hat{\eta}_\mathsf{F}^\perp\otimes\hat{\eta}_\mathsf{F}^\perp\otimes\hat{\eta}_\mathsf{F}^\perp,f^{yy}_{yy}\rangle=0.
\end{split}
\]
Using the gauge conditions, we also have
\[
\begin{split}
\langle\hat{\eta}_\mathsf{F},f^{xy}_{xx}\rangle=\langle\hat{\eta}_\mathsf{F}\otimes\hat{\eta}_\mathsf{F}\otimes\hat{\eta}_\mathsf{F},f^{xy}_{yy}\rangle=\langle\hat{\eta}^\perp_\mathsf{F}\otimes\hat{\eta}^\perp_\mathsf{F}\otimes\hat{\eta}_\mathsf{F},f^{xy}_{yy}\rangle=\langle\hat{\eta}_\mathsf{F}\otimes\hat{\eta}_\mathsf{F}\otimes\hat{\eta}_\mathsf{F},f^{yy}_{xy}\rangle\\
=\langle\hat{\eta}_\mathsf{F}\otimes\hat{\eta}_\mathsf{F},f^{yy}_{xx}\rangle=\langle\hat{\eta}_\mathsf{F}\otimes\hat{\eta}_\mathsf{F}^\perp\otimes\hat{\eta}_\mathsf{F}^\perp, f^{yy}_{xy}\rangle=0.
\end{split}
\]
Step 3. Now we have
\[
(\xi_\mathsf{F}-\mathrm{i})^2\left(\frac{\hat{Y}\cdot\eta_\mathsf{F}}{\xi_\mathsf{F}^2+1}\right)^2f^{xx}_{xx}-2(\xi_\mathsf{F}-\mathrm{i})\left(\frac{\hat{Y}\cdot\eta_\mathsf{F}}{\xi_\mathsf{F}^2+1}\right)\langle\hat{Y},f^{xx}_{xy}\rangle+\langle\hat{Y}\otimes\hat{Y},f^{xx}_{yy}\rangle=0,
\]
and then
\[
\begin{split}
&\epsilon^2\mathfrak{C}^2f^{xx}_{xx}-2\epsilon^2\mathfrak{C}\langle\hat{\eta}_\mathsf{F},f^{xx}_{xy}\rangle
-2\epsilon\sqrt{1-\epsilon^2}\mathfrak{C}\langle\hat{\eta}_\mathsf{F}^\perp,f^{xx}_{xy}\rangle\\
&\quad\quad\quad\quad+\epsilon^2\langle\hat{\eta}_\mathsf{F}\otimes\hat{\eta}_\mathsf{F},f^{xx}_{yy}\rangle+2\epsilon\sqrt{1-\epsilon^2}\langle\hat{\eta}_\mathsf{F}^\perp\otimes\hat{\eta}_\mathsf{F},f^{xx}_{yy}\rangle=0.
\end{split}
\]
Using the identity $\langle\hat{\eta}_\mathsf{F}\otimes\hat{\eta}_\mathsf{F},f^{xy}_{xy}\rangle=\langle\hat{\eta}_\mathsf{F}^\perp\otimes\hat{\eta}_\mathsf{F}^\perp,f^{xy}_{xy}\rangle=0$ and the trace-free condition
\[
f^{xx}_{xx}+\langle\hat{\eta}_\mathsf{F}\otimes\hat{\eta}_\mathsf{F},f^{xy}_{xy}\rangle+\langle\hat{\eta}_\mathsf{F}^\perp\otimes\hat{\eta}_\mathsf{F}^\perp,f^{xy}_{xy}\rangle=0,
\]
we have $f^{xx}_{xx}=0$, and then $\langle\hat{\eta}_\mathsf{F},f^{xx}_{xy}\rangle=\langle\hat{\eta}_\mathsf{F}\otimes\hat{\eta}_\mathsf{F},f^{xx}_{yy}\rangle=0$. So we have
\[
-2\mathfrak{C}\langle\hat{\eta}_\mathsf{F}^\perp,f^{xx}_{xy}\rangle+2\langle\hat{\eta}_\mathsf{F}^\perp\otimes\hat{\eta}_\mathsf{F},f^{xx}_{yy}\rangle=0.
\]
Using the gauge condition, we have
\[
(\frac{|\eta_\mathsf{F}|^2}{\xi_\mathsf{F}^2+1}+1)\langle\hat{\eta}_\mathsf{F}^\perp\otimes\hat{\eta}_\mathsf{F},f^{xx}_{yy}\rangle=0,
\]
and consequently
\[
\langle\hat{\eta}_\mathsf{F}^\perp,f^{xx}_{xy}\rangle=\langle\hat{\eta}_\mathsf{F}^\perp\otimes\hat{\eta}_\mathsf{F},f^{xx}_{yy}\rangle=0.
\]

Now we can conclude that $f= 0$. This shows that the principal symbol of $N_\mathsf{F}$ is elliptic at base infinity restricted to the kernel of the principal symbol of $\delta^\mathcal{B}_\mathsf{F}$.

\end{proof}
\subsection{The gauge condition and the proof of the main results}

 In this section, we use
\[
\begin{array}{rrrr}
(x^2\partial_x)\otimes (x^2\partial_x)\otimes\frac{\mathrm{d}x}{x^2},&\quad (x^2\partial_x)\otimes (x^2\partial_x)\otimes\frac{\mathrm{d}y}{x},\\
 2(x^2\partial_x)\otimes^s (x\partial_y)\otimes\frac{\mathrm{d}x}{x^2},&\quad 2(x^2\partial_x)\otimes^s (x\partial_y)\otimes\frac{\mathrm{d}y}{x}, \\
  (x\partial_y)\otimes (x\partial_y)\otimes\frac{\mathrm{d}x}{x^2},&\quad  (x\partial_y)\otimes (x\partial_y)\otimes\frac{\mathrm{d}y}{x}
  \end{array}
\]
as a basis for scattering symmetric $(2,1)$-tensors,
and
\[
(x^2\partial_x)\otimes\frac{\mathrm{d}x}{x^2},\,(x^2\partial_x)\otimes\frac{\mathrm{d}y}{x},\,(x\partial_y)\otimes\frac{\mathrm{d}x}{x^2},\,(x\partial_y)\otimes\frac{\mathrm{d}y}{x}
\]
for scattering $(1,1)$-tensors.

Then the principal symbol of $\frac{\mathrm{d}'_\mathsf{F}}{\mathrm{i}}:S^2_1 {^{sc}}TX\rightarrow S^2_2{^{sc}}TX$ on scattering symmetric $(2,1)$-tensors is
\[
\left(\begin{array}{cccccc}
\xi+\mathrm{i}\mathsf{F} &&&&&\\
\frac{1}{2}\eta_2&\frac{1}{2} (\xi+\mathrm{i}\mathsf{F})&&&&\\
&\eta_2\otimes^s&&&&\\
&&\xi+\mathrm{i}\mathsf{F} &&&\\
&&\frac{1}{2}\eta_2&\frac{1}{2}( \xi+\mathrm{i}\mathsf{F})&&\\
&&&\eta_2\otimes^s&&\\
&&&&\xi+\mathrm{i}\mathsf{F} &\\
&&&&\frac{1}{2}\eta_2&\frac{1}{2}( \xi+\mathrm{i}\mathsf{F})\\
&&&&&\eta_2\otimes^s
\end{array}\right)+R.
\]
Here $R$ is the principal symbol of an operator in $\mathrm{Diff}^0_{sc}(X,S^2_1{^{sc}TX};S^2_2{^{sc}TX})$.
The principal symbol of $\frac{\delta_\mathsf{F}}{\mathrm{i}}:S^2_2{^{sc}}TX\rightarrow S^2_1{^{sc}}TX$ on scattering symmetric $(2,2)$-tensors is
\[
\left(\begin{array}{ccccccccc}
\xi-\mathrm{i}\mathsf{F} &\langle\eta_2,\cdot\rangle&&&&&&&\\
0&\xi-\mathrm{i}\mathsf{F} &\langle\eta_2,\cdot\rangle&&&&&&\\
&&&\xi-\mathrm{i}\mathsf{F} &\langle\eta_2,\cdot\rangle&&&&\\
&&&0&\xi-\mathrm{i}\mathsf{F} &\langle\eta_2,\cdot\rangle&&&\\
&&&&&&\xi-\mathrm{i}\mathsf{F} &\langle\eta_2,\cdot\rangle&\\
&&&&&&0&\xi-\mathrm{i}\mathsf{F} &\langle\eta_2,\cdot\rangle
\end{array}\right)+R'.
\]
Here $R'$ is the principal symbol of an operator in $\mathrm{Diff}^0_{sc}(X,S^2_2{^{sc}TX};S^2_1{^{sc}TX})$.
In the following, we can consider the operators semiclassically with a small parameter $h=\mathsf{F}^{-1}$, therefore ignoring lower order terms like $R,R'$.

Then the (semiclassical) principal symbol of $-\delta_\mathsf{F}\mathrm{d}'_\mathsf{F}:S^2_1{^{sc}}TX\rightarrow S^2_1{^{sc}}TX$ on scattering symmetric $(2,1)$-tensors is
\[
{\tiny
\left(\begin{array}{cccccc}
\xi^2+\mathsf{F}^2+\frac{1}{2}|\eta|^2 & \frac{1}{2}(\xi+\mathrm{i}\mathsf{F})\langle\eta_2,\cdot\rangle&&&&\\
\frac{1}{2}(\xi-\mathrm{i}\mathsf{F})\eta_2&\frac{1}{2}(\xi^2+\mathsf{F}^2)+\langle\eta_2,\cdot\rangle\eta_2\otimes^s&&&&\\
&&\xi^2+\mathsf{F}^2+\frac{1}{2}|\eta|^2 & \frac{1}{2}(\xi+\mathrm{i}\mathsf{F})\langle\eta_2,\cdot\rangle&&\\
&&\frac{1}{2}(\xi-\mathrm{i}\mathsf{F})\eta_2&\frac{1}{2}(\xi^2+\mathsf{F}^2)+\langle\eta_2,\cdot\rangle\eta_2\otimes^s&&\\
&&&&\xi^2+\mathsf{F}^2+\frac{1}{2}|\eta|^2 & \frac{1}{2}(\xi+\mathrm{i}\mathsf{F})\langle\eta_2,\cdot\rangle\\
&&&&\frac{1}{2}(\xi-\mathrm{i}\mathsf{F})\eta_2&\frac{1}{2}(\xi^2+\mathsf{F}^2)+\langle\eta_2,\cdot\rangle\eta_2\otimes^s
\end{array}\right).
}
\]
Note that
\[
\langle\eta_2,\cdot\rangle\eta_2\otimes^s=\frac{1}{2}|\eta|^2+\frac{1}{2}\eta_2\langle\eta_2,\cdot\rangle.
\]
Here we ignore the terms in $\mathrm{Diff}^1_{sc}(X,S^2_1{^{sc}TX};S^2_1{^{sc}TX})+\mathsf{F}\mathrm{Diff}^0_{sc}(X,S^2_1{^{sc}TX};S^2_1{^{sc}TX})$.\\

The trace operator 
\[
\mu:S^{2}_{2}{^{sc}}TX\rightarrow S^{1}_{1}{^{sc}}TX,\quad(\mu f)^i_j=\delta_k^\ell f^{ik}_{j\ell}
\]
 has principal symbol
\[
\left(\begin{array}{ccccccccc}
1 & 0& 0&0 &\langle \mathrm{Id},\cdot\rangle &0&0&0&0\\
0&1 & 0&0 & 0&\langle \mathrm{Id},\cdot\rangle &0&0&0\\
0&0&0&1 & 0&0 &0 &\langle \mathrm{Id},\cdot\rangle &0\\
0&0&0&0&1 &0 &0 & 0&\langle \mathrm{Id},\cdot\rangle
\end{array}\right),
\]
where $\langle \mathrm{Id},f^{\star y}_{\star y}\rangle=\sum_{i=1}^2 f^{\star y^i}_{\star y^i}$ with $\star =x$ or $y$. We can consider $\mathrm{Id}=\delta^{y^i}_{y^j}$ as a $(1,1)$-tensor in $y$.
The dual operator of $\mu$,
\[
\lambda:S^{1}_{1}{^{sc}}TX\rightarrow S^{2}_{2}{^{sc}}TX,\quad(\lambda w)^{ik}_{j\ell}=\delta^j_\ell w^i_k
\]
 has principal symbol
\[
\left(\begin{array}{cccc}
1 & 0 & 0 & 0\\
0 & \frac{1}{2}& 0 & 0\\
0&0&0&0\\
0&0&\frac{1}{2}&0\\
\frac{1}{4}\mathrm{Id}&0&0&\frac{1}{4}\\
0&\frac{1}{2}\sym_2\mathrm{Id}&0&0\\
0&0&0&0\\
0&0&\frac{1}{2}\sym_1\mathrm{Id}&0\\
0&0&0&\sym_1\sym_2\mathrm{Id}
\end{array}\right).
\]
Here $\sym_1$ represents symmetrization in upper indices, and $\sym_2$ in lower indices.

By \cite[Lemma 2.1]{de2021generic}, we have (in dimension $n=3$)
\[
\mathcal{B}f=(\mathrm{Id}-\frac{4}{5}\lambda\mu)f,\quad \text{for } f\in S^2_2TM.
\]
So we have the mapping properties
\[
\mathrm{d}_\mathsf{F}^\mathcal{B}=\mathcal{B}\mathrm{d}'_\mathsf{F}=\mathrm{d}'_\mathsf{F}-\frac{4}{5}\lambda\mu\mathrm{d}'_\mathsf{F}:\mathcal{B}S^2_1 {^{sc}}TX\rightarrow \mathcal{B}S^2_2{^{sc}}TX,
\]
and
\[
\Delta_\mathsf{F}^\mathcal{B}:=\delta_\mathsf{F}^\mathcal{B}\mathrm{d}_\mathsf{F}^\mathcal{B}=\delta_\mathsf{F}\mathrm{d}'_\mathsf{F}-\frac{4}{5}\delta_\mathsf{F}\lambda\mu\mathrm{d}'_\mathsf{F}:\mathcal{B}S^2_1 {^{sc}}TX\rightarrow \mathcal{B}S^2_1{^{sc}}TX.
\]
We need to prove the invertibility of $\Delta_\mathsf{F}^\mathcal{B}$ as in \cite{paper1}.\\

By calculation, the principal symbol of $\frac{1}{\mathrm{i}}\mu\mathrm{d}'_\mathsf{F}:S^2_1 {^{sc}}TX\rightarrow S^2_2{^{sc}}TX$ is
\[
\left(\begin{array}{cccccc}
\xi+\mathrm{i}\mathsf{F} &0&\frac{1}{2}\langle\eta_1,\cdot\rangle&\frac{1}{2}(\xi+\mathrm{i}\mathsf{F})\langle\mathrm{Id},\cdot\rangle&0&0\\
\frac{1}{2}\eta_2&\frac{1}{2}(\xi+\mathrm{i}\mathsf{F})&0&\langle\mathrm{Id},\cdot\rangle\eta_2\otimes^s&0&0\\
0&0&\xi+\mathrm{i}\mathsf{F} &0&\frac{1}{2}\langle\mathrm{Id},\cdot\rangle\eta_2&\frac{1}{2}(\xi+\mathrm{i}\mathsf{F})\langle\mathrm{Id},\cdot\rangle\\
0&0&\frac{1}{2}\eta_2&\frac{1}{2}(\xi+\mathrm{i}\mathsf{F})&0&\langle\mathrm{Id},\cdot\rangle\eta_2\otimes^s\\
\end{array}\right),
\]
where
\[
\langle\mathrm{Id},\cdot\rangle\eta_2\otimes^s=\frac{1}{2}\langle\eta_1,\cdot\rangle+\frac{1}{2}\eta_2\langle\mathrm{Id},\cdot\rangle.
\]
The principal symbol of $\frac{1}{\mathrm{i}}\delta_\mathsf{F}\lambda:S^2_2 {^{sc}}TX\rightarrow S^2_1{^{sc}}TX$ is

\[
\left(\begin{array}{cccc}
\xi-\mathrm{i}\mathsf{F}&\frac{1}{2}\langle\eta_2,\cdot\rangle&0&0\\
0&\frac{1}{2}(\xi-\mathrm{i}\mathsf{F})&0&0\\
\frac{1}{4}\eta_1&0&\frac{1}{2}(\xi-\mathrm{i}\mathsf{F})&\frac{1}{4}\langle\eta_2,\cdot\rangle\\
\frac{1}{4}(\xi-\mathrm{i}\mathsf{F})\mathrm{Id}&\frac{1}{2}\langle\eta_2,\cdot\rangle \sym_2\mathrm{Id}&0&\frac{1}{4}(\xi-\mathrm{i}\mathsf{F})\\
0&0&\frac{1}{2}\langle\eta_2,\cdot\rangle \sym_1\mathrm{Id}&0\\
0&0&\frac{1}{2}(\xi-\mathrm{i}\mathsf{F}) \sym_1\mathrm{Id}&\langle\eta_2,\cdot\rangle \sym_1\sym_2\mathrm{Id}
\end{array}\right).
\]
We note that
\[
\frac{1}{2}\langle\eta_2,\cdot\rangle \sym_2\mathrm{Id}=\frac{1}{4}(\eta_1+\langle\eta_2,\cdot\rangle\mathrm{Id}).
\]
By tedious calculation, the principal symbol of $-\delta_\mathsf{F}\lambda\mu\mathrm{d}_\mathsf{F}:S^2_1 {^{sc}}TX\rightarrow S^2_1{^{sc}}TX$ is
\[
{\scriptsize
\left(\begin{array}{cccccc}
\xi^2+\mathsf{F}^2+\frac{1}{4}|\eta|^2&\frac{1}{4}(\xi+\mathrm{i}\mathsf{F})\langle\eta_2,\cdot\rangle&\frac{1}{2}(\xi-\mathrm{i}\mathsf{F})\langle\eta_1,\cdot\rangle&\mathfrak{D}_{14}&0&0\\
\frac{1}{4}(\xi-\mathrm{i}\mathsf{F})\eta_2 &\frac{1}{4}(\xi^2+\mathsf{F}^2)&0&\mathfrak{D}_{24}&0&0\\
\frac{1}{4}(\xi+\mathrm{i}\mathsf{F})\eta_1&0&\mathfrak{D}_{33}&\mathfrak{D}_{34}&\frac{1}{4}(\xi-\mathrm{i}\mathsf{F})\langle\eta_1,\cdot\rangle&\frac{1}{4}(\xi^2+\mathsf{F}^2)\langle\mathrm{Id},\cdot\rangle\\
\mathfrak{D}_{41}&\mathfrak{D}_{42}&\mathfrak{D}_{43}&\mathfrak{D}_{44}&0&\mathfrak{D}_{46}\\
0&0&\frac{1}{2}(\xi+\mathrm{i}\mathsf{F})\eta_1\otimes^s&0&\frac{1}{4}\eta_1\otimes^s\langle\eta_1,\cdot\rangle&\frac{1}{4}(\xi+\mathrm{i}\mathsf{F})\eta_1\otimes^s\langle\mathrm{Id},\cdot\rangle\\
0&0&\mathfrak{D}_{63}&\mathfrak{D}_{64}&\frac{1}{4}(\xi-\mathrm{i}\mathsf{F})\mathrm{Id}\otimes^s\langle\eta_1,\cdot\rangle&\mathfrak{D}_{66}
\end{array}\right),
}
\]
where
\[
\mathfrak{D}_{14}=\frac{1}{2}(\xi^2+\mathsf{F}^2)\langle\mathrm{Id},\cdot\rangle+\frac{1}{4}|\eta|^2\langle\mathrm{Id},\cdot\rangle+\frac{1}{4}\langle\eta_1\otimes\eta_2,\cdot\rangle,\quad
\mathfrak{D}_{41}=\frac{1}{4}(\xi^2+\mathsf{F}^2)\mathrm{Id}\otimes+\frac{1}{8}(|\eta|^2\mathrm{Id}\otimes+\eta_1\eta_2),
\]
\[
\mathfrak{D}_{24}=\frac{1}{4}(\xi-\mathrm{i}\mathsf{F})(\langle\eta_1,\cdot\rangle+\langle\mathrm{Id},\cdot\rangle\eta_2),\quad
\mathfrak{D}_{42}=\frac{1}{8}(\xi+\mathrm{i}\mathsf{F})(\eta_1+\langle\eta_2,\cdot\rangle\mathrm{Id}),
\]
\[
\mathfrak{D}_{33}=\frac{1}{2}(\xi^2+\mathsf{F}^2)+\frac{1}{8}\eta_1\langle\eta_1,\cdot\rangle+\frac{1}{8}|\eta|^2,
\]
\[
\mathfrak{D}_{34}=\frac{1}{8}(\xi+\mathrm{i}\mathsf{F})\eta_1\langle\mathrm{Id},\cdot\rangle+\frac{1}{8}(\xi+\mathrm{i}\mathsf{F})\langle\eta_2,\cdot\rangle,\quad
\mathfrak{D}_{43}=\frac{1}{8}(\xi-\mathrm{i}\mathsf{F})\mathrm{Id}\langle\eta_1,\cdot\rangle+\frac{1}{8}(\xi-\mathrm{i}\mathsf{F})\eta_2,
\]
\[
\mathfrak{D}_{36}=\frac{1}{4}(\xi^2+\mathsf{F}^2)\langle\mathrm{Id},\cdot\rangle+\frac{1}{8}(\langle\eta_1\otimes\eta_2,\cdot\rangle+|\eta|^2\langle\mathrm{Id},\cdot\rangle),\quad
\mathfrak{D}_{63}=\frac{1}{2}(\xi^2+\mathsf{F}^2)\sym_1\mathrm{Id}+\frac{1}{4}((\eta_1\otimes^s)\eta_2+|\eta|^2\sym_1\mathrm{Id}),
\]
\[
\mathfrak{D}_{44}=\frac{1}{8}(\xi^2+\mathsf{F}^2)(1+\mathrm{Id}\langle\mathrm{Id},\cdot\rangle)+\frac{1}{8}(\eta_1\langle\eta_1,\cdot\rangle+\langle\eta_1\otimes\eta_2,\cdot\rangle\mathrm{Id}+\langle\mathrm{Id},\cdot\rangle\eta_1\eta_2+|\eta|^2\langle\mathrm{Id},\cdot\rangle\mathrm{Id}),
\]
\[
\mathfrak{D}_{46}=\frac{1}{4}(\xi-\mathrm{i}\mathsf{F})\langle\mathrm{Id},\cdot\rangle\eta_2\otimes^s=\frac{1}{8}(\xi-\mathrm{i}\mathsf{F})\langle\eta_1,\cdot\rangle+\frac{1}{8}(\xi-\mathrm{i}\mathsf{F})\eta_2\otimes\langle\mathrm{Id},\cdot\rangle,
\]
\[
\mathfrak{D}_{64}=\frac{1}{2}(\xi+\mathrm{i}\mathsf{F})\langle\eta_2,\cdot\rangle \sym_1\sym_2\mathrm{Id}\otimes=\frac{1}{4}(\xi+\mathrm{i}\mathsf{F})\eta_1\otimes^s+\frac{1}{4}(\xi+\mathrm{i}\mathsf{F})\mathrm{Id}\otimes^s\langle\eta_2,\cdot\rangle,
\]
\[
\mathfrak{D}_{66}=\frac{1}{4}(\xi^2+\mathsf{F}^2)\mathrm{Id}\otimes^s\langle\mathrm{Id},\cdot\rangle+\frac{1}{4}\eta_1\otimes^s\langle\eta_1,\cdot\rangle+\frac{1}{4}\eta_1\otimes^s\eta_2\otimes\langle\mathrm{Id},\cdot\rangle+\frac{1}{4}|\eta|^2\mathrm{Id}\otimes^s\langle\mathrm{Id},\cdot\rangle+\frac{1}{4}\mathrm{Id}\otimes^s\langle\eta_1\otimes\eta_2,\cdot\rangle.
\]
Here have denoted $\mathrm{Id}\otimes^s=\sym_1\sym_2\mathrm{Id}\otimes$.\\

Next we simplify the principal symbol of $-\delta_\mathsf{F}\lambda\mu\mathrm{d}_\mathsf{F}$ on trace-free symmetric $(2,1)$-tensors.
Now we assume that $f$ is trace-free, and calculate
\begin{flalign*}
\quad\quad&\mathfrak{D}_{11}f^{xx}_x+\mathfrak{D}_{12}f^{xx}_y+\mathfrak{D}_{13}f^{xy}_x+\mathfrak{D}_{14}f^{xy}_y\\
=&(\xi^2+\mathsf{F}^2+\frac{1}{4}|\eta|^2)f^{xx}_x+\frac{1}{4}(\xi+\mathrm{i}\mathsf{F})\langle\eta_2,f^{xx}_y\rangle +\frac{1}{2}(\xi-\mathrm{i}\mathsf{F})\langle\eta_1,f^{xy}_x\rangle \\
&\quad+\frac{1}{2}(\xi^2+\mathsf{F}^2)\langle\mathrm{Id},f^{xy}_y\rangle+\frac{1}{4}|\eta|^2\langle\mathrm{Id},f^{xy}_y\rangle+\frac{1}{4}\langle\eta_1\otimes\eta_2,f^{xy}_y\rangle\\
=&(\frac{1}{2}\xi^2+\frac{1}{2}\mathsf{F}^2+\frac{1}{4}|\eta|^2)(f^{xx}_x+\langle\mathrm{Id},f^{xy}_y\rangle)+\frac{1}{2}(\xi^2+\mathsf{F}^2)f^{xx}_x+\frac{1}{4}(\xi+\mathrm{i}\mathsf{F})\langle\eta_2,f^{xx}_y\rangle\\
&+\frac{1}{2}(\xi-\mathrm{i}\mathsf{F})\langle\eta_1,f^{xy}_x\rangle+\frac{1}{4}\langle\eta_1\otimes\eta_2,f^{xy}_y\rangle\\
=&\frac{1}{2}(\xi^2+\mathsf{F}^2)f^{xx}_x+\frac{1}{4}(\xi+\mathrm{i}\mathsf{F})\langle\eta_2,f^{xx}_y\rangle+\frac{1}{2}(\xi-\mathrm{i}\mathsf{F})\langle\eta_1,f^{xy}_x\rangle+\frac{1}{4}\langle\eta_1\otimes\eta_2,f^{xy}_y\rangle,&
\end{flalign*}

\begin{flalign*}
\quad\quad&\mathfrak{D}_{21}f^{xx}_x+\mathfrak{D}_{22}f^{xx}_y+\mathfrak{D}_{24}f^{xy}_y\\
=&\frac{1}{4}(\xi-\mathrm{i}\mathsf{F})\eta_2f^{xx}_x +\frac{1}{4}(\xi^2+\mathsf{F}^2)f^{xx}_y+\frac{1}{4}(\xi-\mathrm{i}\mathsf{F})(\langle\eta_1,f^{xy}_{y}\rangle+\langle\mathrm{Id},f^{xy}_{y}\rangle\eta_2)\\
=&\frac{1}{4}(\xi-\mathrm{i}\mathsf{F})\eta_2(f^{xx}_x+\langle\mathrm{Id},f^{xy}_y\rangle)+\frac{1}{4}(\xi^2+\mathsf{F}^2)f^{xx}_y+\frac{1}{4}(\xi-\mathrm{i}\mathsf{F})\langle\eta_1,f^{xy}_{y}\rangle\\
=&\frac{1}{4}(\xi^2+\mathsf{F}^2)f^{xx}_y+\frac{1}{4}(\xi-\mathrm{i}\mathsf{F})\langle\eta_1,f^{xy}_{y}\rangle,&
\end{flalign*}

\begin{flalign*}
\quad\quad&\mathfrak{D}_{31}f^{xx}_x+\mathfrak{D}_{33}f^{xy}_x+\mathfrak{D}_{34}f^{xy}_y+\mathfrak{D}_{35}f^{yy}_x+\mathfrak{D}_{36}f^{yy}_y\\
=&\frac{1}{4}(\xi+\mathrm{i}\mathsf{F})\eta_1f^{xx}_x+\frac{1}{2}(\xi^2+\mathsf{F}^2+\frac{1}{4}|\eta|^2)f^{xy}_x+\frac{1}{8}\eta_1\langle\eta_1,f^{xy}_x\rangle+\frac{1}{8}(\xi+\mathrm{i}\mathsf{F})\eta_1\langle\mathrm{Id},f^{xy}_{y}\rangle\\
&+\frac{1}{8}(\xi+\mathrm{i}\mathsf{F})\langle\eta_2,f^{xy}_{y}\rangle+\frac{1}{4}(\xi-\mathrm{i}\mathsf{F})\langle\eta_1,f^{yy}_x\rangle+\frac{1}{4}(\xi^2+\mathsf{F}^2)\langle\mathrm{Id},f^{yy}_y\rangle\\
&+\frac{1}{8}(\langle\eta_1\otimes\eta_2,f^{yy}_y\rangle+|\eta|^2\langle\mathrm{Id},f^{yy}_y\rangle)\\
=&\frac{1}{8}(\xi+\mathrm{i}\mathsf{F})\eta_1f^{xx}_x+\frac{1}{4}(\xi^2+\mathsf{F}^2)f^{xy}_x+\frac{1}{8}\eta_1\langle\eta_1,f^{xy}_x\rangle
+\frac{1}{8}(\xi+\mathrm{i}\mathsf{F})\langle\eta_2,f^{xy}_{y}\rangle\\
&+\frac{1}{4}(\xi-\mathrm{i}\mathsf{F})\langle\eta_1,f^{yy}_x\rangle+\frac{1}{8}\langle\eta_1\otimes\eta_2,f^{yy}_y\rangle,&
\end{flalign*}

\begin{flalign*}
\quad\quad&\mathfrak{D}_{41}f^{xx}_x+\mathfrak{D}_{42}f^{xx}_y+\mathfrak{D}_{43}f^{xy}_x+\mathfrak{D}_{44}f^{xy}_y+\mathfrak{D}_{46}f^{yy}_y\\
=&\frac{1}{4}(\xi^2+\mathsf{F}^2)f^{xx}_x\mathrm{Id}+\frac{1}{8}(|\eta|^2\mathrm{Id}+\eta_1\eta_2)f^{xx}_x+\frac{1}{8}(\xi+\mathrm{i}\mathsf{F})(\eta_1f^{xx}_y+\langle\eta_2,f^{xx}_y\rangle\mathrm{Id})+\frac{1}{8}(\xi-\mathrm{i}\mathsf{F})\langle\eta_1,f^{xy}_x\rangle\mathrm{Id}\\
&+\frac{1}{8}(\xi-\mathrm{i}\mathsf{F})\eta_2f^{xy}_x
+\frac{1}{8}(\xi^2+\mathsf{F}^2)(f^{xy}_y+\mathrm{Id}\langle\mathrm{Id},f^{xy}_y\rangle)+\frac{1}{8}(\eta_1\langle\eta_1,f^{xy}_y\rangle+\langle\eta_1\otimes\eta_2,f^{xy}_y\rangle\mathrm{Id})+\frac{1}{8}\langle\mathrm{Id},f^{xy}_y\rangle\eta_1\eta_2\\
&+\frac{1}{8}|\eta|^2\langle\mathrm{Id},f^{xy}_y\rangle\mathrm{Id}+\frac{1}{8}(\xi-\mathrm{i}\mathsf{F})\langle\eta_1,f^{yy}_y\rangle+\frac{1}{8}(\xi-\mathrm{i}\mathsf{F})\eta_2\otimes\langle\mathrm{Id},f^{yy}_y\rangle\\
=&\frac{1}{8}(\xi^2+\mathsf{F}^2)f^{xx}_x\mathrm{Id}+\frac{1}{8}(\xi+\mathrm{i}\mathsf{F})(\eta_1f^{xx}_y+\langle\eta_2,f^{xx}_y\rangle\mathrm{Id})+\frac{1}{8}(\xi-\mathrm{i}\mathsf{F})\langle\eta_1,f^{xy}_x\rangle\mathrm{Id}+\frac{1}{8}(\xi^2+\mathsf{F}^2)f^{xy}_y\\
&+\frac{1}{8}(\eta_1\langle\eta_1,f^{xy}_y\rangle+\langle\eta_1\otimes\eta_2,f^{xy}_y\rangle\mathrm{Id})+\frac{1}{8}(\xi-\mathrm{i}\mathsf{F})\langle\eta_1,f^{yy}_y\rangle,&
\end{flalign*}

\begin{flalign*}
\quad\quad&\mathfrak{D}_{53}f^{xy}_x+\mathfrak{D}_{55}f^{yy}_x+\mathfrak{D}_{56}f^{yy}_y\\
=&\frac{1}{2}(\xi+\mathrm{i}\mathsf{F})\eta_1\otimes^s f^{xy}_x+\frac{1}{4}\eta_1\otimes^s\langle\eta_1,f^{yy}_x\rangle+\frac{1}{4}(\xi+\mathrm{i}\mathsf{F})\eta_1\otimes^s\langle\mathrm{Id},f^{yy}_y\rangle\\
=&\frac{1}{4}(\xi+\mathrm{i}\mathsf{F})\eta_1\otimes^s f^{xy}_x+\frac{1}{4}\eta_1\otimes^s\langle\eta_1,f^{yy}_x\rangle,&
\end{flalign*}

\begin{flalign*}
\quad\quad&\mathfrak{D}_{63}f^{xy}_x+\mathfrak{D}_{64}f^{xy}_y+\mathfrak{D}_{65}f^{yy}_x+\mathfrak{D}_{66}f^{yy}_y\\
=&\frac{1}{2}(\xi^2+\mathsf{F}^2)\sym_1\mathrm{Id}f^{xy}_x+\frac{1}{4}((\eta_1\otimes^s)\eta_2f^{xy}_x+|\eta|^2\sym_1\mathrm{Id}f^{xy}_x)+\frac{1}{4}(\xi+\mathrm{i}\mathsf{F})\eta_1\otimes^sf^{xy}_y\\
&+\frac{1}{4}(\xi+\mathrm{i}\mathsf{F})\mathrm{Id}\otimes^s\langle\eta_2,f^{xy}_y\rangle+\frac{1}{4}(\xi-\mathrm{i}\mathsf{F})\sym_1\mathrm{Id}\langle\eta_1,f^{yy}_x\rangle+\frac{1}{4}(\xi^2+\mathsf{F}^2)\mathrm{Id}\otimes^s\langle\mathrm{Id},f^{yy}_y\rangle\\
&+\frac{1}{4}\eta_1\otimes^s\langle\eta_1,f^{yy}_y\rangle+\frac{1}{4}\eta_1\otimes^s\langle\mathrm{Id},f^{yy}_y\rangle\eta_2+\frac{1}{4}|\eta|^2\mathrm{Id}\otimes^s\langle\mathrm{Id},\cdot\rangle+\frac{1}{4}\mathrm{Id}\otimes^s\langle\eta_1\otimes\eta_2,\cdot\rangle\\
=&\frac{1}{4}(\xi^2+\mathsf{F}^2)\sym_1\mathrm{Id}f^{xy}_x+\frac{1}{4}(\xi+\mathrm{i}\mathsf{F})\eta_1\otimes^sf^{xy}_y+\frac{1}{4}(\xi+\mathrm{i}\mathsf{F})\mathrm{Id}\otimes^s\langle\eta_2,f^{xy}_y\rangle+\frac{1}{4}(\xi-\mathrm{i}\mathsf{F})\sym_1\mathrm{Id}\langle\eta_1,\cdot\rangle\\
&+\frac{1}{4}\eta_1\otimes^s\langle\eta_1,f^{yy}_y\rangle+\frac{1}{4}\mathrm{Id}\otimes^s\langle\eta_1\otimes\eta_2,\cdot\rangle.&
\end{flalign*}

So, acting on trace-free tensors, the principal symbol of $-\delta_\mathsf{F}\lambda\mu\mathrm{d}_\mathsf{F}:\mathcal{B}S^2_1 {^{sc}}TX\rightarrow S^2_1{^{sc}}TX$ reduces to
\[
{\tiny
\left(\begin{array}{cccccc}
\frac{1}{2}(\xi^2+\mathsf{F}^2)&\frac{1}{4}(\xi+\mathrm{i}\mathsf{F})\langle\eta_2,\cdot\rangle&\frac{1}{2}(\xi-\mathrm{i}\mathsf{F})\langle\eta_1,\cdot\rangle&\frac{1}{4}\langle\eta_1\otimes\eta_2,\cdot\rangle&0&0\\
0&\frac{1}{4}(\xi^2+\mathsf{F}^2)&0&\frac{1}{4}\xi-\mathrm{i}\mathsf{F})(\langle\eta_1,\cdot\rangle&0&0\\
\frac{1}{8}(\xi+\mathrm{i}\mathsf{F})\eta_1&0&\frac{1}{4}(\xi^2+\mathsf{F}^2)+\frac{1}{8}\eta_1\langle\eta_1,\cdot\rangle
&\frac{1}{8}(\xi+\mathrm{i}\mathsf{F})\langle\eta_2,\cdot\rangle&\frac{1}{4}(\xi-\mathrm{i}\mathsf{F})\langle\eta_1,\cdot\rangle&\frac{1}{8}\langle\eta_1\otimes\eta_2,\cdot\rangle\\
\frac{1}{8}(\xi^2+\mathsf{F}^2)\mathrm{Id}&\frac{1}{8}(\xi+\mathrm{i}\mathsf{F})(\eta_1+\langle\eta_2,\cdot\rangle\mathrm{Id})&\frac{1}{8}(\xi-\mathrm{i}\mathsf{F})\eta_2&C_{44}&0&\frac{1}{8}(\xi-\mathrm{i}\mathsf{F})\langle\eta_1,\cdot\rangle\\
0&0&\frac{1}{4}(\xi+\mathrm{i}\mathsf{F})\eta_1\otimes^s&0&\frac{1}{4}\eta_1\otimes^s\langle\eta_1,\cdot\rangle&0\\
0&0&\frac{1}{4}(\xi^2+\mathsf{F}^2)\sym_1\mathrm{Id}&C_{64}&\frac{1}{4}(\xi-\mathrm{i}\mathsf{F})\sym_1\mathrm{Id}\langle\eta_1,\cdot\rangle&C_{66}
\end{array}\right),
}
\]
where
\[
\begin{split}
 C_{44}&=\frac{1}{8}(\xi^2+\mathsf{F}^2)+\frac{1}{8}(\eta_1\langle\eta_1,\cdot\rangle+\langle\eta_1\otimes\eta_2,\cdot\rangle\mathrm{Id}),\\
 C_{64}&=\frac{1}{4}(\xi+\mathrm{i}\mathsf{F})(\eta_1\otimes^s+\mathrm{Id}\otimes^s\langle\eta_2,\cdot\rangle),\\
C_{66}&=\frac{1}{4}\eta_1\otimes^s\langle\eta_1,\cdot\rangle+\frac{1}{4}\mathrm{Id}\otimes^s\langle\eta_1\otimes\eta_2,\cdot\rangle.
\end{split}
\]

 The principal symbol of $\frac{\mathrm{d}'_\mathsf{F}}{\mathrm{i}}:S^2_0 {^{sc}}TX\rightarrow S^2_1{^{sc}}TX$ on scattering symmetric $(2,0)$-tensors is
\[
\left(\begin{array}{ccc}
\xi+\mathrm{i}\mathsf{F} &&\\
\eta_2&&\\
&\xi+\mathrm{i}\mathsf{F} &\\
&\eta_2&\\
&&\xi+\mathrm{i}\mathsf{F} \\
&&\eta_2
\end{array}\right)+R.
\]
The principal symbol of $\frac{\delta_\mathsf{F}}{\mathrm{i}}:S^2_1 {^{sc}}TX\rightarrow S^2_0{^{sc}}TX$ on scattering symmetric $(2,1)$-tensors is
\[
\left(\begin{array}{cccccc}
\xi-\mathrm{i}\mathsf{F} &\langle\eta_2,\cdot\rangle&&&&\\
&&(\xi-\mathrm{i}\mathsf{F})&\langle\eta_2,\cdot\rangle&&\\
&&&&\xi-\mathrm{i}\mathsf{F} &\langle\eta_2,\cdot\rangle
\end{array}\right)+R'.
\]
Here $R,R'$ are the principal symbols of operators in $\mathrm{Diff}^0_{sc}$.
Then the (semiclassical) principal symbol of $-\mathrm{d}_\mathsf{F}\delta_\mathsf{F}:S^2_1 {^{sc}}TX\rightarrow S^2_1{^{sc}}TX$ on scattering symmetric $(2,1)$-tensors is
\[
\left(\begin{array}{cccccc}
\xi^2+\mathsf{F}^2&(\xi+\mathrm{i}\mathsf{F})\langle\eta_2,\cdot\rangle &&&&\\
(\xi-\mathrm{i}\mathsf{F})\eta_2&\eta_2\langle\eta_2,\cdot\rangle&&&&\\
&&(\xi^2+\mathsf{F}^2)&(\xi+\mathrm{i}\mathsf{F})\langle\eta_2,\cdot\rangle &&\\
&&(\xi-\mathrm{i}\mathsf{F})\eta_2&\eta_2\langle\eta_2,\cdot\rangle&&\\
&&&&\xi^2+\mathsf{F}^2&(\xi+\mathrm{i}\mathsf{F})\langle\eta_2,\cdot\rangle \\
&&&&(\xi-\mathrm{i}\mathsf{F})\eta_2&\eta_2\langle\eta_2,\cdot\rangle\\
\end{array}\right).
\]
Here we also ignore terms in $\mathrm{Diff}^1_{sc}(X,S^2_1{^{sc}TX};S^2_1{^{sc}TX})+\mathsf{F}\mathrm{Diff}^0_{sc}(X,S^2_1{^{sc}TX};S^2_1{^{sc}TX})$.

The trace operator on (scattering) symmetric $(2,1)$-tensors
\[
\mu:S^{2}_{1}{^{sc}}TX\rightarrow S^{1}_{0}{^{sc}}TX,\quad(\mu f)^i=\delta_k^\ell f^{ik}_{\ell}
\]
has principal symbol
\[
\left(\begin{array}{cccccc}
1 & 0& 0&\langle \mathrm{Id},\cdot\rangle &0&0\\
0& 0& 1&0 & 0&\langle \mathrm{Id},\cdot\rangle\\
\end{array}\right),
\]
and its formal dual
\[
\lambda:S^{1}_{0}{^{sc}}TX\rightarrow S^{2}_{1}{^{sc}}TX,\quad(\mu w)^{ij}_k=\delta_j^k f^{i}
\]
 has principal symbol
\[
\left(\begin{array}{cc}
1 & 0\\
0& 0 \\
0&\frac{1}{2}\\
\frac{1}{2}\mathrm{Id}&0\\
0&0\\
0&\mathrm{Sym}_1\mathrm{Id}
\end{array}\right).
\]

By calculation,
the principal symbol of $-\lambda\mu\mathrm{d}'_\mathsf{F}\delta_\mathsf{F}$ is
\[
{\scriptsize
\left(\begin{array}{cccccc}
\xi^2+\mathsf{F}^2& (\xi+\mathrm{i}\mathsf{F})\langle\eta_2,\cdot\rangle& (\xi-\mathrm{i}\mathsf{F})\langle\eta_1,\cdot\rangle&\langle\eta_1\otimes\eta_2,\cdot\rangle&0&0\\
0&0&0&0&0&0\\
0& 0& \frac{1}{2}(\xi^2+\mathsf{F}^2)& \frac{1}{2}(\xi+\mathrm{i}\mathsf{F})\langle\eta_2,\cdot\rangle&  \frac{1}{2}(\xi-\mathrm{i}\mathsf{F})\langle\eta_1,\cdot\rangle&\frac{1}{2}\langle\eta_1\otimes\eta_2,\cdot\rangle\\
\frac{1}{2}(\xi^2+\mathsf{F}^2)\mathrm{Id}& \frac{1}{2}(\xi+\mathrm{i}\mathsf{F})\mathrm{Id}\langle\eta_2,\cdot\rangle& \frac{1}{2}(\xi-\mathrm{i}\mathsf{F})\mathrm{Id}\langle\eta_1,\cdot\rangle&\frac{1}{2}\mathrm{Id}\otimes^s\langle\eta_1\otimes\eta_2,\cdot\rangle&0&0\\
0&0&0&0&0&0\\
0& 0& (\xi^2+\mathsf{F}^2)\mathrm{Id}\otimes^s&(\xi+\mathrm{i}\mathsf{F})\mathrm{Id}\otimes^s\langle\eta_2,\cdot\rangle&  (\xi-\mathrm{i}\mathsf{F})\mathrm{Id}\otimes^s\langle\eta_1,\cdot\rangle&\mathrm{Id}\otimes^s\langle\eta_1\otimes\eta_2,\cdot\rangle
\end{array}\right).
}
\]

By \cite[Lemma 2.2]{de2021generic}, we have
\[
\mathrm{d}^\mathcal{B}_\mathsf{F}\delta^\mathcal{B}_\mathsf{F}=\mathcal{B}\mathrm{d}_\mathsf{F}\delta_\mathsf{F}=\mathrm{d}_\mathsf{F}\delta_\mathsf{F}-\frac{1}{2}\lambda\mu\mathrm{d}_\mathsf{F}\delta_\mathsf{F}:\mathcal{B}S^{2}_{1}{^{sc}}TX\rightarrow \mathcal{B}S^{2}_{1}{^{sc}}TX.
\]
Define
\[
\mathfrak{D}_1=\left(\begin{array}{cccccc}
-\partial_x+\mathsf{F} &&\nabla^1_y\cdot&&&\\
&-\partial_x+\mathsf{F}&&\nabla^1_y\cdot&&\\
&&-\partial_x+\mathsf{F}&&\nabla^1_y\cdot&\\
&&&-\partial_x+\mathsf{F}&&\nabla^1_y\cdot
\end{array}\right),
\]
acting on tensors in $S^{2}_{1}{^{sc}}TX$,
where $\nabla^1_y\cdot$ takes divergence (with respect to $\frac{h(x,y)\mathrm{d}y^2}{x^2}$) in $y$ on the upper indices.
The principal symbol of $\mathfrak{D}_1^*\mathfrak{D}_1:S^{2}_{1}{^{sc}}TX\rightarrow S^{2}_{1}{^{sc}}TX$ is
\[
{\tiny
\begin{split}
&\sigma(\mathfrak{D}_1^*\mathfrak{D}_1)\\
=&\left(\begin{array}{cccc}
\xi-\mathrm{i}\mathsf{F}&&&\\
&\xi-\mathrm{i}\mathsf{F}&&\\
-\frac{1}{2}\eta_1&&\frac{1}{2}(\xi-\mathrm{i}\mathsf{F})&\\
&-\frac{1}{2}\eta_1&&\frac{1}{2}(\xi-\mathrm{i}\mathsf{F})\\
&&-\eta_1\otimes^s&\\
&&&-\eta_1\otimes^s
\end{array}\right)
\left(\begin{array}{cccccc}
\xi+\mathrm{i}\mathsf{F}&&-\langle\eta_1,\cdot\rangle&&&\\
&\xi+\mathrm{i}\mathsf{F}&&-\langle\eta_1,\cdot\rangle&&\\
&&(\xi+\mathrm{i}\mathsf{F})&&-\langle\eta_1,\cdot\rangle&\\
&&&(\xi+\mathrm{i}\mathsf{F})&&-\langle\eta_1,\cdot\rangle
\end{array}\right)\\
=&\left(\begin{array}{cccccc}
\xi^2+\mathsf{F}^2&&-(\xi-\mathrm{i}\mathsf{F})\langle\eta_1,\cdot\rangle&&&\\
&\xi^2+\mathsf{F}^2&&-(\xi-\mathrm{i}\mathsf{F})\langle\eta_1,\cdot\rangle&&\\
-\frac{1}{2}(\xi+\mathrm{i}\mathsf{F})\eta_1&&\frac{1}{2}(\xi^2+\mathsf{F}^2+|\eta|^2)&&-\frac{1}{2}(\xi-\mathrm{i}\mathsf{F})\langle\eta_1,\cdot\rangle&\\
&-\frac{1}{2}(\xi+\mathrm{i}\mathsf{F})\eta_1&&\frac{1}{2}(\xi^2+\mathsf{F}^2+|\eta|^2)&&-\frac{1}{2}(\xi-\mathrm{i}\mathsf{F})\langle\eta_1,\cdot\rangle\\
&&-(\xi+\mathrm{i}\mathsf{F})\eta_1\otimes^s&&|\eta|^2&\\
&&&-(\xi+\mathrm{i}\mathsf{F})\eta_1\otimes^s&&|\eta|^2
\end{array}\right).
\end{split}
}
\]

By calculation, we obtain that the principal symbol of $-\delta_\mathsf{F}\lambda\mu\mathrm{d}_\mathsf{F}+\frac{1}{2}(-\mathrm{d}_\mathsf{F}\delta_\mathsf{F}+\frac{1}{2}\lambda\mu\mathrm{d}_\mathsf{F}\delta_\mathsf{F})+\frac{1}{4}\mathfrak{D}_1^*\mathfrak{D}_1:\mathcal{B}S^{2}_{1}{^{sc}}TX\rightarrow S^{2}_{1}{^{sc}}TX$ as
\[
\left(\begin{array}{ccc}
A_1&&\\
&A_2&\\
&&A_3
\end{array}\right),
\]
where
\[
A_1=\left(\begin{array}{cc}
\xi^2+\mathsf{F}^2&\frac{1}{2}(\xi+\mathrm{i}\mathsf{F})\langle\eta_2,\cdot\rangle\\
\frac{1}{2}(\xi-\mathrm{i}\mathsf{F})\eta_2&\frac{1}{2}(\xi^2+\mathsf{F}^2)+\frac{1}{2}\eta_2\langle\eta_2,\cdot\rangle
\end{array}\right),
\]

\[
A_2=\left(\begin{array}{cc}
\frac{3}{4}(\xi^2+\mathsf{F}^2)+\frac{1}{8}|\eta|^2+\frac{1}{8}\eta_1\langle\eta_1,\cdot\rangle
&\frac{1}{2}(\xi+\mathrm{i}\mathsf{F})\langle\eta_2,\cdot\rangle\\
\frac{1}{2}(\xi-\mathrm{i}\mathsf{F})\eta_2&\frac{1}{4}(\xi^2+\mathsf{F}^2)+\frac{1}{8}|\eta|^2+\frac{1}{8}\eta_1\langle\eta_1,\cdot\rangle+\frac{1}{2}\eta_2\langle\eta_2,\cdot\rangle
\end{array}\right),
\]
and
\[
A_3=\left(\begin{array}{cc}
\frac{1}{2}(\xi^2+\mathsf{F}^2)+\frac{1}{4}|\eta|^2+\frac{1}{4}\eta_1\otimes^s\langle\eta_1,\cdot\rangle&\frac{1}{2}(\xi+\mathrm{i}\mathsf{F})\langle\eta_2,\cdot\rangle\\
\frac{1}{2}(\xi-\mathrm{i}\mathsf{F})\eta_2&\frac{1}{4}|\eta|^2+\frac{1}{4}\eta_1\otimes^s\langle\eta_1,\cdot\rangle+\frac{1}{2}\eta_2\langle\eta_2,\cdot\rangle
\end{array}\right).
\]

By tedious calculation, we see that the principal symbol of $-\delta_\mathsf{F}\mathrm{d}_\mathsf{F}-\frac{4}{5}(-\delta_\mathsf{F}\lambda\mu\mathrm{d}_\mathsf{F}+\frac{1}{2}(-\mathrm{d}_\mathsf{F}\delta_\mathsf{F}+\frac{1}{2}\lambda\mu\mathrm{d}_\mathsf{F}\delta_\mathsf{F}))-\frac{1}{5}\mathfrak{D}_1^*\mathfrak{D}_1$ is
\[
\left(\begin{array}{ccc}
B_1&&\\
&B_2&\\
&&B_3
\end{array}\right),
\]
where
\[
B_1=\left(\begin{array}{cc}
\frac{1}{5}(\xi^2+\mathsf{F}^2)+\frac{1}{2}|\eta|^2&\frac{1}{10}(\xi+\mathrm{i}\mathsf{F})\langle\eta_2,\cdot\rangle\\
\frac{1}{10}(\xi-\mathrm{i}\mathsf{F})\eta_2&\frac{1}{10}(\xi^2+\mathsf{F}^2)+\frac{1}{2}|\eta|^2+\frac{1}{10}\eta_2\langle\eta_2,\cdot\rangle
\end{array}\right),
\]

\[
B_2=\left(\begin{array}{cc}
\frac{2}{5}(\xi^2+\mathsf{F}^2+|\eta|^2)-\frac{1}{10}\eta_1\langle\eta_1,\cdot\rangle&\frac{1}{10}(\xi+\mathrm{i}\mathsf{F})\langle\eta_2,\cdot\rangle\\
\frac{1}{10}(\xi-\mathrm{i}\mathsf{F})\eta_2&\frac{3}{10}(\xi^2+\mathsf{F}^2)+\frac{2}{5}|\eta|^2+\frac{1}{10}\eta_2\langle\eta_2,\cdot\rangle-\frac{1}{10}\eta_1\langle\eta_1,\cdot\rangle
\end{array}\right),
\]
and
\[
B_3=\left(\begin{array}{cc}
\frac{3}{5}(\xi^2+\mathsf{F}^2+|\eta|^2)-\frac{1}{5}\eta_1\otimes^s\langle\eta_1,\cdot\rangle&\frac{1}{10}(\xi+\mathrm{i}\mathsf{F})\langle\eta_2,\cdot\rangle\\
\frac{1}{10}(\xi-\mathrm{i}\mathsf{F})\eta_2&\frac{1}{2}(\xi^2+\mathsf{F}^2)+\frac{3}{10}|\eta|^2+\frac{1}{10}\eta_2\langle\eta_2,\cdot\rangle-\frac{1}{5}\eta_1\otimes^s\langle\eta_1,\cdot\rangle
\end{array}\right).
\]
Notice that
\[
B_1=\left(\begin{array}{cc}
\frac{1}{10}(\xi^2+\mathsf{F}^2)+\frac{1}{2}|\eta|^2&0\\
0&\frac{1}{10}(\xi^2+\mathsf{F}^2)+\frac{1}{2}|\eta|^2
\end{array}\right)+\frac{1}{10}\left(\begin{array}{cc}
\xi^2+\mathsf{F}^2&(\xi+\mathrm{i}\mathsf{F})\langle\eta_2,\cdot\rangle\\
(\xi-\mathrm{i}\mathsf{F})\eta_2&\eta_2\langle\eta_2,\cdot\rangle
\end{array}\right),
\]

\[
\begin{split}
B_2=\left(\begin{array}{cc}
\frac{3}{10}(\xi^2+\mathsf{F}^2)+\frac{2}{5}|\eta|^2-\frac{1}{10}\eta_1\langle\eta_1,\cdot\rangle&0\\
0&\frac{3}{10}(\xi^2+\mathsf{F}^2)+\frac{2}{5}|\eta|^2-\frac{1}{10}\eta_1\langle\eta_1,\cdot\rangle
\end{array}\right)\\
+\frac{1}{10}\left(\begin{array}{cc}
(\xi^2+\mathsf{F}^2)&(\xi+\mathrm{i}\mathsf{F})\langle\eta_2,\cdot\rangle\\
(\xi-\mathrm{i}\mathsf{F})\eta_2&\eta_2\langle\eta_2,\cdot\rangle
\end{array}\right),
\end{split}
\]
and
\[
\begin{split}
B_3=\left(\begin{array}{cc}
\frac{1}{2}(\xi^2+\mathsf{F}^2)+\frac{3}{10}|\eta|^2-\frac{1}{5}\eta_1\otimes^s\langle\eta_1,\cdot\rangle&0\\
0&\frac{1}{2}(\xi^2+\mathsf{F}^2)+\frac{3}{10}|\eta|^2-\frac{1}{5}\eta_1\otimes^s\langle\eta_1,\cdot\rangle
\end{array}\right)\\
+\frac{1}{10}\left(\begin{array}{cc}
\xi^2+\mathsf{F}^2&(\xi+\mathrm{i}\mathsf{F})\langle\eta_2,\cdot\rangle\\
(\xi-\mathrm{i}\mathsf{F})\eta_2&\eta_2\langle\eta_2,\cdot\rangle
\end{array}\right).
\end{split}
\]

Notice that
\[
\begin{split}
&\left(\begin{array}{cc}
\frac{1}{10}(\xi^2+\mathsf{F}^2)+\frac{1}{2}|\eta|^2&0\\
0&\frac{1}{10}(\xi^2+\mathsf{F}^2)+\frac{1}{2}|\eta|^2
\end{array}\right)\\
=&\frac{1}{10}\sigma\left(\begin{array}{cc}
\nabla_\mathsf{F}^*\nabla_\mathsf{F}&0\\
0&\nabla_\mathsf{F}^*\nabla_\mathsf{F}
\end{array}\right)+\frac{2}{5}\sigma\left(\begin{array}{cc}
-\Delta_y&0\\
0&-\Delta_y
\end{array}\right),
\end{split}
\]

\[
\begin{split}
&\left(\begin{array}{cc}
\frac{3}{10}(\xi^2+\mathsf{F}^2)+\frac{2}{5}|\eta|^2-\frac{1}{10}\eta_1\langle\eta_1,\cdot\rangle&0\\
0&\frac{3}{10}(\xi^2+\mathsf{F}^2)+\frac{2}{5}|\eta|^2-\frac{1}{10}\eta_1\langle\eta_1,\cdot\rangle
\end{array}\right)\\
=&\frac{3}{10}\sigma\left(\begin{array}{cc}
\nabla_\mathsf{F}^*\nabla_\mathsf{F}&0\\
0&\nabla_\mathsf{F}^*\nabla_\mathsf{F}
\end{array}\right)+\frac{1}{10}\sigma\left(\begin{array}{cc}
-\Delta_y+\nabla_y^1\nabla_y^1\cdot&0\\
0&-\Delta_y+\nabla_y^1\nabla_y^1\cdot
\end{array}\right),
\end{split}
\]
and
\[
\begin{split}
&\left(\begin{array}{cc}
\frac{1}{2}(\xi^2+\mathsf{F}^2)+\frac{3}{10}|\eta|^2-\frac{1}{5}\eta_1\otimes^s\langle\eta_1,\cdot\rangle&0\\
0&\frac{1}{2}(\xi^2+\mathsf{F}^2)+\frac{3}{10}|\eta|^2-\frac{1}{5}\eta_1\otimes^s\langle\eta_1,\cdot\rangle
\end{array}\right)\\
=&\frac{1}{10}\sigma\left(\begin{array}{cc}
\nabla_\mathsf{F}^*\nabla_\mathsf{F}&0\\
0&\nabla_\mathsf{F}^*\nabla_\mathsf{F}
\end{array}\right)+\frac{2}{5}\sigma\left(\begin{array}{cc}
-\partial_x^2+\mathsf{F}^2&0\\
0&-\partial_x^2+\mathsf{F}^2
\end{array}\right)+\frac{1}{5}\sigma\left(\begin{array}{cc}
-\Delta_y+\nabla_y^{1,s}\nabla_y^1\cdot&0\\
0&-\Delta_y+\nabla_y^{1,s}\nabla_y^1\cdot
\end{array}\right).
\end{split}
\]
Here $\nabla_\mathsf{F}=e^{-\mathsf{F}/x}\nabla e^{\mathsf{F}/x}$ with $\nabla$ gradient relative to $g_{sc}$.
Writing tensors as vectors, we define the operator
\[
\mathfrak{D}_2=\left(\begin{array}{cccccc}
\sqrt{\frac{2}{5}}\nabla_y&&&&&\\
&\sqrt{\frac{2}{5}}\nabla_y&&&&\\
&&\sqrt{\frac{1}{5}}\nabla_\mathsf{F}&&&\\
&&&\sqrt{\frac{1}{5}}\nabla_\mathsf{F}&&\\
&&&&\sqrt{\frac{2}{5}}(\partial_x+\mathsf{F})&\\
&&&&&\sqrt{\frac{2}{5}}(\partial_x+\mathsf{F})
\end{array}\right).
\]
Also define
\[
\mathfrak{D}_3=\left(\begin{array}{cccccc}
0&&&&&\\
&0&&&&\\
&&\sqrt{\frac{1}{10}}\nabla_y&&&\\
&&&\sqrt{\frac{1}{10}}\nabla_y&&\\
&&&&\sqrt{\frac{1}{5}}\nabla_y&\\
&&&&&\sqrt{\frac{1}{5}}\nabla_y
\end{array}\right)
\]
and
\[
\mathfrak{D}_4=\left(\begin{array}{cccccc}
0&&&&&\\
&0&&&&\\
&&\sqrt{\frac{1}{10}}\nabla^1_y\cdot&&&\\
&&&\sqrt{\frac{1}{10}}\nabla^1_y\cdot&&\\
&&&&\sqrt{\frac{1}{5}}\nabla^1_y\cdot&\\
&&&&&\sqrt{\frac{1}{5}}\nabla^1_y\cdot
\end{array}\right).
\]
We also take their adjoints $\mathfrak{D}_2^*,\mathfrak{D}_3^*,\mathfrak{D}_4^*$ with respect to the inner product on $S_1^2 {^{sc}}TX$ induced by the metric $g_{sc}$.
One can verify that acting on $\mathcal{B}S^2_1{^{sc}}TX$
\[
\begin{split}
&\sigma(-\delta_\mathsf{F}\mathrm{d}_\mathsf{F}-\frac{4}{5}(-\delta_\mathsf{F}\lambda\mu\mathrm{d}_\mathsf{F}+\frac{1}{2}(-\mathrm{d}_\mathsf{F}\delta_\mathsf{F}+\frac{1}{2}\lambda\mu\mathrm{d}_\mathsf{F}\delta_\mathsf{F})))-\frac{1}{5}\sigma(\mathfrak{D}_1^*\mathfrak{D}_1)\\
=&\frac{1}{10}\sigma(\nabla_\mathsf{F}^*\nabla_\mathsf{F})-\frac{1}{10}\sigma(\mathrm{d}_\mathsf{F}\delta_\mathsf{F})+\sigma(\mathfrak{D}_2^*\mathfrak{D}_2)+\sigma(\mathfrak{D}_3^*\mathfrak{D}_3)-\sigma(\mathfrak{D}_4^*\mathfrak{D}_4).
\end{split}
\]
Notice that $-\mathrm{d}_\mathsf{F}\delta_\mathsf{F}+\frac{1}{2}\lambda\mu\mathrm{d}_\mathsf{F}\delta_\mathsf{F}=-\mathrm{d}^\mathcal{B}_\mathsf{F}\delta^\mathcal{B}_\mathsf{F}$.
Projecting to the trace-free subspace $\mathcal{B}S^2_1{^{sc}}TX$,  and with all the adjoints retaken with respect to the inner product on $\mathcal{B}S^2_1{^{sc}}TX$, we have the following decomposition of the operator $-\delta^\mathcal{B}_\mathsf{F}\mathrm{d}^\mathcal{B}_\mathsf{F}:\mathcal{B}S^2_1{^{sc}}TX\rightarrow\mathcal{B}S^2_1{^{sc}}TX$
\[
\begin{split}
&\sigma(-\delta^\mathcal{B}_\mathsf{F}\mathrm{d}^\mathcal{B}_\mathsf{F})=\sigma(-\delta_\mathsf{F}\mathrm{d}_\mathsf{F}+\frac{4}{5}\delta_\mathsf{F}\lambda\mu\mathrm{d}_\mathsf{F})\\
=&\frac{1}{10}\sigma(\nabla_\mathsf{F}^*\nabla_\mathsf{F})+\frac{1}{2}\sigma(-\mathrm{d}^\mathcal{B}_\mathsf{F}\delta^\mathcal{B}_\mathsf{F})+\frac{1}{5}\sigma(\mathfrak{D}_1^*\mathfrak{D}_1)+\sigma(\mathfrak{D}_2^*\mathfrak{D}_2)+\sigma(\mathfrak{D}_3^*\mathfrak{D}_3)-\sigma(\mathfrak{D}_4^*\mathfrak{D}_4).
\end{split}
\]
Since we have been considering the principal symbols as semiclassical ones,
this shows that
\[
-\Delta_\mathsf{F}^\mathcal{B}=-\delta^\mathcal{B}_\mathsf{F}\mathrm{d}^\mathcal{B}_\mathsf{F}=\frac{1}{10}\nabla_\mathsf{F}^*\nabla_\mathsf{F}-\frac{1}{2}\mathrm{d}_\mathsf{F}^\mathcal{B}\delta^\mathcal{B}_\mathsf{F}+\frac{1}{5}\mathfrak{D}_1^*\mathfrak{D}_1+\mathfrak{D}_2^*\mathfrak{D}_2+\mathfrak{D}_3^*\mathfrak{D}_3-\mathfrak{D}_4^*\mathfrak{D}_4+A+B,
\]
with $A\in\mathrm{Diff}^1_{sc}(X,\mathcal{B}S^2_1{^{sc}TX};\mathcal{B}S^2_1{^{sc}TX})$ and $B\in h^{-1}\mathrm{Diff}^0_{sc}(X,\mathcal{B}S^2_1{^{sc}TX};\mathcal{B}S^2_1{^{sc}TX})$. It is easy to see that
\[
\|\mathfrak{D}_3v\|^2-\|\mathfrak{D}_4v\|^2\geq -\|v\|_{L^2_{sc}}.
\]

Similar to \cite{paper1}, we have the invertibility of $-\Delta_\mathsf{F}^\mathcal{B}$:
\begin{lemma}
The operator $-\Delta^\mathcal{B}_\mathsf{F}\in\mathrm{Diff}_{sc}^{2,0}(X;\mathcal{B}S^2_1{^{sc}TX},\mathcal{B}S^2_1{^{sc}TX})$, considered as a map $\dot{H}^{1,0}_{sc}\rightarrow (\dot{H}^{1,0}_{sc})^*=\overline{H}^{-1,0}_{sc}$, is invertible for $\mathsf{F}$ sufficiently large.
\end{lemma}

The rest of the proofs for Theorem \ref{thmlocal22} and  Theorem \ref{thmglobal22} is similar to the proofs for $L_{1,1}$ in \cite{paper1}.

\bibliographystyle{abbrv}
\bibliography{biblio}

\begin{thebibliography}{10}

\bibitem{AR}
Y.~Anikonov and V.~Romanov.
\newblock On uniqueness of determination of a form of first degree by it
  integrals along geodesics.
\newblock {\em J. Inverse Ill-Posed Probl.}, 5:467--480, 1997.

\bibitem{dairbekov2006integral}
N.~S. Dairbekov.
\newblock Integral geometry problem for nontrapping manifolds.
\newblock {\em Inverse Problems}, 22(2):431, 2006.

\bibitem{de2021generic}
M.~de~Hoop, T.~Saksala, G.~Uhlmann, and J.~Zhai.
\newblock Generic uniqueness and stability for the mixed ray transform.
\newblock {\em Transactions of the American Mathematical Society},
  374(09):6085--6144, 2021.

\bibitem{de2018mixed}
M.~V. de~Hoop, T.~Saksala, and J.~Zhai.
\newblock Mixed ray transform on simple 2-dimensional {R}iemannian manifolds.
\newblock {\em Proceedings of the American Mathematical Society},
  147(11):4901--4913, 2019.

\bibitem{de2020recovery}
M.~V. de~Hoop, G.~Uhlmann, and A.~Vasy.
\newblock Recovery of material parameters in transversely isotropic media.
\newblock {\em Archive for Rational Mechanics and Analysis}, 235:141--165,
  2020.

\bibitem{de2019inverting}
M.~V. de~Hoop, G.~Uhlmann, and J.~Zhai.
\newblock Inverting the local geodesic ray transform of higher rank tensors.
\newblock {\em Inverse Problems}, 35(11):115009, 2019.

\bibitem{melrose2020spectral}
R.~B. Melrose.
\newblock Spectral and scattering theory for the {L}aplacian on asymptotically
  euclidian spaces.
\newblock In {\em Spectral and scattering theory}, pages 85--130. CRC Press,
  2020.

\bibitem{Muk2}
R.~G. Mukhometov.
\newblock On the problem of integral geometry ({R}ussian).
\newblock {\em Math. problems of geophysics. Akad. Nauk SSSR, Sibirsk., Otdel.,
  Vychisl., Tsentr, Novosibirsk}, 6, 1975.

\bibitem{Muk1}
R.~G. Mukhometov.
\newblock On a problem of reconstructing {R}iemannian metrics.
\newblock {\em Sibirsk. Mat. Zh.}, 22:119--135, 1987.

\bibitem{paternain2013tensor}
G.~P. Paternain, M.~Salo, and G.~Uhlmann.
\newblock Tensor tomography on surfaces.
\newblock {\em Inventiones Mathematicae}, 193(1):229--247, 2013.

\bibitem{paternain2015invariant}
G.~P. Paternain, M.~Salo, and G.~Uhlmann.
\newblock Invariant distributions, {B}eurling transforms and tensor tomography
  in higher dimensions.
\newblock {\em Mathematische Annalen}, 363(1-2):305--362, 2015.

\bibitem{paternain2019lens}
G.~P. Paternain, G.~Uhlmann, and H.~Zhou.
\newblock Lens rigidity for a particle in a {Y}ang--{M}ills field.
\newblock {\em Communications in Mathematical Physics}, 366(2):681--707, 2019.

\bibitem{Shara}
V.~A. Sharafutdinov.
\newblock {\em Integral geometry of tensor fields}, volume~1.
\newblock Walter de Gruyter, 1994.

\bibitem{stefanov2004stability}
P.~Stefanov and G.~Uhlmann.
\newblock Stability estimates for the {X}-ray transform of tensor fields and
  boundary rigidity.
\newblock {\em Duke Mathematical Journal}, 123(3):445--467, 2004.

\bibitem{stefanov2018inverting}
P.~Stefanov, G.~Uhlmann, and A.~Vasy.
\newblock Inverting the local geodesic {X}-ray transform on tensors.
\newblock {\em Journal d'Analyse Mathematique}, 136(1):151--208, 2018.

\bibitem{stefanov2021local}
P.~Stefanov, G.~Uhlmann, and A.~Vasy.
\newblock Local and global boundary rigidity and the geodesic {X}-ray transform
  in the normal gauge.
\newblock {\em Annals of Mathematics}, 194(1):1--95, 2021.

\bibitem{UV}
G.~Uhlmann and A.~Vasy.
\newblock The inverse problem for the local geodesic ray transform.
\newblock {\em Invent. Math.}, 205:83--120, 2016.

\bibitem{paper1}
G.~Uhlmann and J.~Zhai.
\newblock Invertibility of local geodesic transverse and mixed ray transforms
  {I}: basic cases.
\newblock {\em arXiv preprint arXiv:2401.09017}, 2024.

\bibitem{zhou2018lens}
H.~Zhou.
\newblock Lens rigidity with partial data in the presence of a magnetic field.
\newblock {\em Inverse Problems and Imaging}, 12(6):1365--1387, 2018.

\bibitem{zhou2018local}
H.~Zhou.
\newblock The local magnetic ray transform of tensor fields.
\newblock {\em SIAM Journal on Mathematical Analysis}, 50(2):1753--1778, 2018.

\bibitem{zou2019partial}
Y.~Zou.
\newblock Partial global recovery in the elastic travel time tomography problem
  for transversely isotropic media.
\newblock {\em arXiv preprint arXiv:1910.01052}, 2019.

\end{thebibliography}
\end{document}